\documentclass[a4paper,11pt,leqno]{article}
\usepackage[utf8]{inputenc}
\usepackage[T1]{fontenc}
\usepackage[english]{babel}
\usepackage{todonotes}
\usepackage{babelbib}
\usepackage{amsmath}
\usepackage{amsfonts}
\usepackage{mathtools}
\usepackage{amssymb}
\usepackage{booktabs}
\usepackage{amsthm}
\usepackage{gauss}
\usepackage{xcolor}
\usepackage{setspace}
\usepackage{amsthm}
\usepackage{graphicx}
\usepackage{hyperref}
\usepackage{textcomp}
\usepackage{blkarray}
\usepackage{multirow}
\usepackage{cancel}
\usepackage{array}
\usepackage{pgf}
\usepackage{tikz}
\usepackage{tikz-3dplot}
\usepackage{enumerate}
\usepackage{comment}
\usetikzlibrary{matrix,arrows}
\usepackage{mathrsfs}
\usepackage{bbm,yhmath,stmaryrd}
\usepackage{geometry}
\usepackage{soul}
\setstcolor{red}
\newcommand \Z{\mathbb Z}
\newcommand \N{\mathbb N}
\newcommand \cN{\mathscr N}
\newcommand \C{\mathbb C}
\newcommand \R{\mathbb R}
\newcommand \Q{\mathbb Q}
\newcommand \gO {\mathcal O}

\newcommand \Ld {\mathscr {L}}
\newcommand \D {\mathbb{D}}
\newcommand \de {\Delta}
\newcommand \U {\mathscr U}
\newcommand \CP {\mathbb P}
\newcommand \A {\mathbb A}

\newcommand \T {\mathbb{T}}

\newcommand \G {\mathbb{G}}
\newcommand \X {\mathscr{X}}

\newcommand \cZ {\mathscr{Z}}

\newcommand \cD {\mathcal{D}}

\newcommand \fl {\longrightarrow}
\newcommand \eps {\varepsilon}

\newcommand \De {\Delta}

\DeclareMathOperator{\Pic}{Pic}
\DeclareMathOperator{\ord}{ord}

\DeclareMathOperator{\Hom}{Hom}
\DeclareMathOperator{\Bl}{Bl}

\DeclareMathOperator \Id {Id}

\DeclareMathOperator \GL {GL}
\DeclareMathOperator \Div {Div}

\DeclareMathOperator{\Star}{Star}
\DeclareMathOperator{\Spec}{Spec}

\DeclareMathOperator{\Spf}{Spf}
\DeclareMathOperator{\an}{an}
\DeclareMathOperator{\Sk}{Sk}
\DeclareMathOperator{\red}{red}

\DeclareMathOperator{\val}{val}

\DeclareMathOperator{\Trop}{Trop}

\usepackage{hyperref}
\usepackage[utf8]{inputenc}
\usepackage{graphicx}
\usepackage{amsmath}
\usepackage[capitalise]{cleveref}
\usepackage{enumerate}
\usepackage{enumitem}
\usepackage[all,cmtip]{xy}
\usepackage{tikz}
\usepackage{tikz-cd}
\usetikzlibrary{cd}
\usetikzlibrary{decorations.markings}
\usepackage{float}
\usetikzlibrary{arrows}
\usetikzlibrary{patterns}
\usepackage{afterpage}
\usepackage{csquotes}
\usepackage{combelow}
\usepackage[bottom]{footmisc}
\usepackage{mathtools}
\usepackage{nccmath}

\newtheoremstyle{plain2}    
{}            
{}            
{\itshape}    
{}            
{\bfseries}   
{.}           
{5pt plus 1pt minus 1pt}  
{{\thmnumber{#1} \thmname{#2}{\thmnote{ (#3)}}}}          
\theoremstyle{plain2}

\newtheorem{theo}[equation]{Theorem}
\newtheorem{prop}[equation]{Proposition}
\newtheorem{lem}[equation]{Lemma}

\newtheorem{cor}[equation]{Corollary}
\theoremstyle{definition}
\newtheorem{defn}[equation]{Definition}
\newtheorem{ex}[equation]{Example}

\newtheorem{rem}[equation]{Remark}

\theoremstyle{plain}
\newtheorem*{thmA}{Theorem A}

\newtheorem*{corC}{Corollary C}
\theoremstyle{plain2}
\newtheorem{thmx}{Theorem}
\newtheorem{corx}[thmx]{Corollary}

\numberwithin{equation}{subsection} 

\usetikzlibrary{arrows,automata}
\usetikzlibrary{positioning}

\tikzset{
    state/.style={
           rectangle,
           rounded corners,
           draw=gray, 
           minimum height=1em,
           },
}

\newcommand{\cV}{\ensuremath{\mathscr{V}}}
\newcommand{\cY}{\ensuremath{\mathscr{Y}}}

\newcommand{\cX}{\ensuremath{\mathscr{X}}}
\newcommand{\cU}{\ensuremath{\mathscr{U}}}
\newcommand{\cW}{\ensuremath{\mathscr{W}}}

\title{Toric geometry and integral affine structures \\ in non-archimedean mirror symmetry}
\author{Enrica Mazzon, Léonard Pille-Schneider}
\begin{document}
\date{}
\nocite{*}
\maketitle

\begin{abstract}
We study integral dlt models of a proper $\C((t))$-variety $X$ along a toric stratum of the special fiber. We prove that the associated Berkovich retraction - from the non-archimedean analytification of $X$ onto the dual complex of the model - is an affinoid torus fibration around the simplex corresponding to the toric stratum, which extends results in \cite{NXY}. 
As a byproduct, we construct new types of non-archimedean retractions for maximally degenerate families of quintic $3$-folds. These induce on $\mathbb{S}^3$ the same singular integral affine structures that arise on the dual complex of toric degenerations in the Gross-Siebert program, as well as on the Gromov-Hausdorff limit of the family.
\end{abstract}

\tableofcontents

\section*{Introduction}
Let $(X, L)$ be a polarized family of $n$-dimensional Calabi--Yau varieties over the punctured disk $\D^* \subset \C$; each fiber $X_t$ additionally carries a unique Ricci-flat Kähler-metric $\omega_t \in c_1(L_t)$, according to the celebrated Yau theorem.
We will be primarily interested in such families that are \emph{maximally degenerate}, in the following sense: the monodromy acting on the degree $n$ cohomology of the general fiber has a Jordan block of maximal (that is, $n+1$) size. 

In this setting, the Strominger-Yau-Zaslow conjecture predicts that the general fiber $X_t$ admits a fibration $\rho_t : X_t \fl S$, called an \emph{SYZ fibration}, whose base $S$ is a real $n$-dimensional topological manifold (even a sphere if the $X_t$ are strict Calabi--Yau), and whose fibers are special Lagrangian tori away from a discriminant locus of codimension $2$ in $S$. 
\\An SYZ fibration endows $S$ with a singular \emph{integral affine structure}, induced by action-angle coordinates. This means that, in the complement of the discriminant locus of the fibration, the transition functions between charts of $S$ are affine transformations in $\textrm{GL}_n(\mathbb{Z}) \ltimes \mathbb{R}^n$.
\\Moreover, the limit for $t \rightarrow 0$ of the metric spaces $(X_t,\omega_t)$ should correspond to the metric collapse of the torus fibers of $\rho_t$. Then, the (suitably rescaled) Gromov-Hausdorff limit of $(X_t,\omega_t)$ should coincide with the space $S$, endowed with a metric which in affine coordinates satisfies a real Monge--Ampère equation away from the discriminant locus.
\vspace{10pt}

While some examples of special Lagrangian torus fibrations can be produced, dealing with the general case seems very difficult. The insight of Kontsevich and Soibelman is to replace the above conjecture by an analogous one in the non-archimedean world, and to interpret the latter as an asymptotic limit of the complex phenomenon when $t \rightarrow 0$. We now elaborate on this idea.

Consider the field $K=\C((t))$ of Laurent power series, which comes equipped with the non-archimedean valuation $\ord_t$, order of vanishing at $t=0$; the family $X$ can be viewed as a variety over $K$. Within this framework, we associate with $X/K$ a topological space, called the \emph{Berkovich analytification} $X^{\an}$ of $X$; this is a space of real (semi)valuations on $X$ (see \cref{subsec:Berkovich geometry}).
\\A way to construct and visualize points of $X^{\an}$ is to consider \emph{models} of $X$ over $R=\C[[t]]$. Indeed, any suitably regular (dlt) model $\X$ of $X$ has an associated simplicial subset $\Sk(\X) \subset X^{\an}$, called the \emph{skeleton} of $\X$ and homeomorphic to the dual (intersection) complex of the special fiber of $\X$, and a continuous retraction $\rho_\X: X^{\an} \rightarrow \Sk(\X)$, inducing a homotopy equivalence (see \cref{subsec:skeletons,subsec:Berkovich retractions} for more details). It follows that $X^{\an}$ encodes geometric information coming from degenerations of $X$ and about combinatorics of models of $X$.
\\Among various models and associated skeletons, minimal (in the sense of MMP) models $\X$ of $X$ determine a canonical skeleton $\Sk(X)=\Sk(\X)$, called the \emph{essential skeleton} of $X$ and independent of the choice of the minimal model. The essential skeleton and the retractions $\rho_{\X}: X^{\an} \rightarrow \Sk(\X)=\Sk(X)$, which do depend on $\cX$, are of particular relevance in the non-archimedean reformulation of the SYZ conjecture as the following conjectures point out. 
\vspace{15pt}

The key idea is that Berkovich theory should allow to construct (non-unique) non-archimedean avatars 
of SYZ fibrations.
More precisely, in \cite{KontsevichSoibelman} Kontsevich and Soibelman conjecture that the essential skeleton $\Sk(X)$ can be endowed with an integral affine structure outside of a codimension 2 piecewise-affine subset $\Gamma \subset \Sk(X)$, such that the following holds. The space $\Sk(X)$ can be recovered from the Kähler geometry of $X$, as a (suitably rescaled) Gromov-Hausdorff limit of the metric spaces $(X_t, \omega_t)$. Moreover, the limiting metric on $\Sk(X)$ should satisfy the following: outside of $\Gamma$, the metric is given locally in affine coordinates by the Hessian of a convex function, satisfying a real Monge-Ampère equation. It is furthermore expected that this limiting affine structure can be recovered by a map $\rho: X^{\an} \rightarrow \Sk(X)$, which is a non-archimedean analog of the SYZ fibration. 
\\The construction of the above fibration is made more rigorous in \cite{NXY}. The authors prove that the retraction $\rho_\X$ associated with a minimal model $\X$ is an \emph{affinoid torus fibration} away from a codimension 2 locus of the base - the non-archimedean analog of a smooth torus fibration - and induces an integral affine structure there, as the SYZ heuristic and the conjecture by Kontsevich and Soibelman predict. Here in particular the transition functions of the integral affine structure are in $\textrm{GL}_n(\mathbb{Z}) \ltimes \mathbb{Z}^n$.

The local model for affinoid torus fibrations is the tropicalization map $\textrm{val}:\mathbb{T}^{\an} \rightarrow N_\mathbb{R}$, where $\mathbb{T}=\mathbb{G}_{m,K}^n$ and $N$ is the cocharacter lattice of the torus $\mathbb{T}$. Global examples of such retractions are given as follows: given a (non-proper) toric variety $\cY$ over $R=\mathbb{C}[[t]]$ which is a model of $\mathbb{T}$, the retraction $\rho_\cY$ is a restriction of $\mathrm{val}$. This reduces the proof of the result in \cite{NXY} to showing that minimal models $\X$ are in fact toric along one-dimensional strata of the special fiber when the latter is reduced.
\vspace{10pt}

At this point the base of the SYZ fibration appears to be well identified - as the essential skeleton or equivalently the dual complex of any minimal model - while the affine structure and the metric are not. In fact, the construction in \cite{NXY} yields integral affine structures that depend on the additional choice of a model, while the Kontsevich--Soibelman conjecture predicts uniqueness, at least of the metric space.
Moreover, the location and the nature of the singularities obtained in \cite{NXY} differ from previous constructions in mirror symmetry. 
\\Such discrepancy already appears in the case of quintic threefolds in $\mathbb{P}^4$. On one side, the constructions in \cite{Ruan2001,Gross2001} - using symplectic and toric geometry respectively - yield an affine structure on a triangulated $3$-sphere whose singularities are located away from the vertices. On the other side, the discriminant locus of the non-archimedean SYZ fibration constructed in \cite{NXY} passes through the vertices of the triangulation.
Moreover, the work by \cite{Li} for degenerations of Fermat hypersurfaces, and more generally by \cite{HJMM}, provides evidence that the affine structure on the Gromov-Hausdorff limit of the Kähler Ricci-flat metric on the nearby fibers coincides with the one constructed in \cite{Gro05}, and also has no singularities at the vertices of the natural triangulation of $\Sk(X)$.
\vspace{10pt}

In this paper we deal with the apparent incompatibility raised by the expected affine structures on the essential skeleton and the ones induced by the non-archimedean SYZ fibration in the sense of \cite{NXY}. To this purpose, we further develop the non-archimedean approach, and produce examples of a new type of non-archimedean retractions.
Our main result is a generalization of \cite[Proposition 5.4]{NXY}:

\begin{thmx}\label{intro:main thm Z}
Let $X/K$ be a smooth variety, and $\X/R$ be a dlt model of $X$ with reduced special fiber $\X_k= \sum_i D_i$, such that each $D_i$ is a Cartier divisor.
\\Let $Z$ be a stratum of $\X_k$, such that:
\begin{itemize}
\item $Z$ is a proper toric variety with toric boundary $\sum_{i\,|\, Z \not\subset D_i} Z \cap D_i$;
\item the conormal bundle $\nu_{Z/ \X}^*$ is a nef vector bundle on $Z$;
\item the intersection of $Z$ with any irreducible component of $\X_k$ is connected.
\end{itemize}
Then $\X$ is toric along $Z$ (in the sense of \cref{toricdef}).
\end{thmx}
Note that the dlt assumption, combined with the fact that $Z$ is toric, imply that $Z$ is smooth (see \cref{rem dlt}). 
By assumption, $Z$ is (a connected component of) the intersection of the divisors $D_j \subset \X_k$ containing $Z$ and the $D_j$'s are Cartier, so that $\nu^*_{Z/\X} = \bigoplus_{j \in J} \gO_Z(-D_j)$ and the nef assumption simply means that each of the $\gO_Z(-D_j)$ is a nef divisor.
\\Using the positivity of the conormal bundle, we then prove that in a formal neighbourhood of $Z$, $\X$ is isomorphic to the normal bundle of $Z$, which is a toric variety. This is similar in spirit to the classical work of \cite[Satz 7, p. 363]{Gra} on holomorphic tubular neighbourhoods, as well as Grothendieck's algebraization theorem \cite[Theorem 5.1.4]{EGA3.1}; the key technical point being the vanishing of the higher cohomology groups of the powers of $\nu^*_{Z/ \X}$ which allows us to extend combinatorial data from $Z$ to a formal neighbourhood. Note that if $Z$ is a rational curve, then the positivity assumption can always be achieved via a finite number of blow-ups (\cite[Proposition 5.4]{NXY}). 
\\The main consequence of the above theorem is the following:
\begin{corx} \label{intro:cor}
The Berkovich retraction $\rho_{\X} : X^{\an} \rightarrow \Sk(\X)$ is an affinoid torus fibration over $\Star(\tau_Z)$.
\end{corx}
The subset $\Star(\tau_Z)$ is the open star of the face determined by $Z$ (see \cref{defn:star}). \cref{intro:main thm Z} and \cref{intro:cor} show that the discriminant locus of the retraction $\rho_{\X}$ measures the defect of a stratum to being toric. 
\vspace{10pt}

The connection between toric geometry and mirror symmetry has been explored in several ways; in particular, the Gross--Siebert program 
considers 
\emph{toric degenerations} of Calabi--Yau varieties.
Such degenerations satisfy assumptions similar to the ones in  \cref{intro:main thm Z}, as the irreducible components of the special fiber are all assumed to be toric varieties; however, note that they are not assumed to be $\Q$-Cartier, so that there may not be an associated Berkovich retraction. In \cite{GrossSiebert2006} the authors then glue together the fans of the various components of the special fiber to combinatorially construct an affine structure on the skeleton of the degeneration; see \cref{rem:K3} for an example in dimension 2. In particular, we prove in \cref{cor fan structure} that this construction of integral affine structure can be performed using Berkovich retractions.
\\An important class of examples of toric degenerations is provided by Batyrev-Borisov degenerations \cite{BB}, as studied for instance in \cite{Gro05}: those are complete intersections in a toric variety described by combinatorial data. Since such a degeneration $X$ is embedded in a toric variety, one can define its tropicalization $\Trop(X)$, which is a balanced polyhedral complex, and there exists a continuous tropicalization map $\val : X^{\an} \fl \Trop(X)$.
In this setting, Yamamoto \cite{Ya} constructs a tropical contraction $\delta$ from the tropicalization of the variety to its essential skeleton, which induces on $\Sk(X)$ the singular integral affine structure defined in \cite{Gro05}. It then turns out that applying \cref{intro:main thm Z}, one can prove that the composition:
$$\rho := \delta \circ \val : X^{\an} \fl \Sk(X)$$
is indeed an affinoid torus fibration in codimension $1$, and induces on $\Sk(X)$ the singular affine structure defined in \cite{Gro05}.
The case of hypersurfaces of the form:
\begin{equation} \label{equ:degeneration}
X = \{ z_1\dots z_{n+1} + tF =0 \} \subset \CP^{n+1}_K,\tag{$\star$}\end{equation} 
where $F$ is a general homogenenous polynomial of degree $n+2$, is treated in \cite[Theorem A]{PS}, while the general case of complete intersections in a toric variety will appear in \cite{Ya2}. This bridges the gap between the Gross-Siebert framework and the non-archimedean approach of Kontsevich-Soibelman for constructing integral affine structures on the dual complex of Batyrev-Borisov degenerations.
\vspace{10pt}

In the case of quintic 3-fold degenerations of the form \ref{equ:degeneration}, we also provide a different, more \emph{ad hoc} construction of a global retraction:
\begin{thmx} \label{intro:main quintic}
There exists a continuous retraction $\pi : X^{\an} \rightarrow \Sk(X)$ such that
\begin{itemize}
\itemsep0pt
    \item $\pi$ can be written as a composition $\pi' \circ \rho_{\X'}$, with $\X'$ being an snc model of $X$ and $\pi' : \Sk(\X') \rightarrow \Sk(X)$ a piecewise-linear map;
    
    \item $\pi$ is an affinoid torus fibration outside a piecewise-linear locus $\Gamma$, that has codimension $2$ and is contained in the $2$-skeleton of $\Sk(X)$;

    
    \item $\pi$ induces an integral affine structure on $\Sk(X) \setminus \Gamma$, which is the same as the one constructed in \cite{Gro05,Li,HJMM}.
\end{itemize}
\end{thmx}
Notably, while this retraction agrees with the one from \cite[Theorem A]{PS} over the complement of the discriminant locus, both retractions differ over the discriminant locus, as was pointed out to us by Yuto Yamamoto, see \cref{sec:comparison}. This implies, in particular, that in the non-archimedean setting an affinoid torus fibration does not have in general a unique compactification over a discriminant locus of codimension $2$.
\\Finally, the non-archimedean retraction $\pi$ of \cref{intro:main quintic} yields an integral affine structure whose discriminant locus is of codimension 2. We recall that the $2$-codimensionality of the discriminant is expected from the SYZ heuristic at the topological level; in \cref{intro:main quintic}, this is achieved by construction of $\pi'$, building on the results in \cite{NXY} and on \cref{intro:main thm Z}.
In the setting of the Gross--Siebert program, given a singular integral affine manifold $B$, one can produce, using classical moment maps, a topological torus fibration over $B$, which however has a discriminant locus of codimension 1, see \cite[§2.1]{RS20}. In the series of recent or upcoming papers \cite{RZ21a,RZ21,RZ}, Ruddat and Zharkov develop a strategy which solves this problem - at least at the topological level - and works in arbitrary dimension. 
\vspace{10pt} 
\\Let us briefly describe the organization of the paper. 
\\In \cref{Sec:preliminaries}, we introduce some notation and collect some basic facts about toric varieties. We also define Berkovich spaces and recall the definition of skeletons and retractions we will be using.
\cref{Section toric thm} is devoted to the proof of \cref{intro:main thm Z}. Finally, in \cref{sec:quintic 3-fold} we study in detail the example of the degeneration of quintic 3-folds and prove \cref{intro:main quintic} by applying the results of the previous sections. We also compare our results to various constructions existing in the literature.

\vspace{10pt} \textbf{Acknowledgements.} 
We would like to thank S\'{e}bastien Boucksom, Antoine Ducros and Mirko Mauri for their comments on the first version of this paper. We are also grateful to Omid Amini, Laurent Moret-Bailly, Johannes Nicaise, Helge Ruddat, Yuto Yamamoto for helpful conversations.
Enrica Mazzon was partially supported by Max Planck Institute for Mathematics in Bonn during the preparation of this paper.

\vspace{10pt}
\section{Preliminaries} \label{Sec:preliminaries}
Throughout this paper, $k$ is an algebraically closed field of characteristic zero, $K = k((t))$ and $R= k[[t]]$. The field $K$ is endowed with the non-archimedean absolute value $\lvert \cdot \rvert = e^{-\ord_t}$, which makes $K$ a complete non-archimedean field with valuation ring $R$.

\subsection{Models} \label{prel:models}
Let $X$ be a separated scheme of finite type over $K$. A separated flat $R$-scheme $\X$ of finite type together with an isomorphism of $K$-schemes $\X \times_{R} K \simeq X$ is called an $R$-model of $X$. We denote by $\X_k = \X \times_R k$ the special fiber of $\X$, and by $\Div_0(\X)$ the group of Weil divisors on $\X$ supported on the special fiber.

If $Y$ is a normal variety and $D$ a Weil divisor on $Y$, whose irreducible decomposition is $D= \sum_{i \in I} a_i D_i$, a stratum of $D$ is a connected component of an intersection $D_J = \cap_{j \in J} D_j$ for some $J \subset I$. An open stratum of $D$ is a stratum $Z$ minus the irreducible components of $D$ not containing $Z$; this is denoted by $\mathring{Z}$.

\begin{defn}{\label{defn:dlt}}
Let $\X /R$ be a model of $X$. We say that $\X$ is a \emph{dlt} (divisorially log terminal) model of $X$ if the following conditions hold:
\begin{itemize}
\itemsep0em
    \item[-] the pair $(\X, \X_{k, \red})$ is log canonical in the sense of the Minimal Model Program (see \cite{KM});
    \item[-] the pair $(\X, \X_{k, \red})$ is simple normal crossing at the generic points of log canonical centers of $(\X, \X_{k,\red})$.
\end{itemize}
We will additionally say that a dlt model $\X$ is \emph{good} if each irreducible component of $\X_{k,\red}$ is $\Q$-Cartier.
\end{defn}
We will not give a precise definition of log canonical centers here, and refer the reader to \cite{KM}. However, if $(\X, \X_{k, \red})$ is dlt and defined over an algebraic curve - which is the most relevant case for applications - then the log canonical centers of $(\X, \X_{k, \red})$ are precisely the strata of $\X_k$ by \cite[4.16]{Kollar2013}, so that a dlt model $\X$ is simple normal crossing at the generic points of the strata of $\X_k$. If $\X_k$ is reduced, it follows from the approximation arguments of \cite[Corollary 4.4]{NXY} that this holds in the general case as well.
We refer the reader to \cite[\S 1.12-1.14]{NXY} for an overview on existence results of such models.

\subsection{Toric geometry}
Throughout this section, let $Z$ be an $r$-dimensional (normal) proper toric variety over $k$, in the sense of \cite{KK}.
This means that $Z$ is a normal $k$-variety, containing the torus $\T = \G^r_{m,k}$ as an open subset, and such that the torus action onto itself extends to an action on $Z$. The complement $\Delta_Z = Z \setminus \T$ is a reduced anticanonical Weil divisor in $Z$, called the toric boundary of $Z$; we write it as the sum of its irreducible components $\Delta_Z = \sum_{l \in L} Z_l$.
We write $N = \Hom(\G_{m,k}, \T)$ for the free abelian group of 1-parameter subgroups of $\T$, and $N_{\R} = N \otimes \R$.

The variety $Z$ can be described by a combinatorial object, called its \emph{fan} $\Sigma$. The fan lives inside the finite-dimensional vector space $N_{\R}$; $\Sigma$ is a collection of strictly convex rational polyhedral cones $ \Sigma = \{ \sigma \}_{\sigma \in \Sigma}$ inside $N_{\R}$, stable under intersection and such that each face of a cone in $\Sigma$ is itself in $\Sigma$. The cones of $\Sigma$ are in inclusion-reversing bijection with the strata of $\Delta_Z$; this induces in particular a bijection between the set of rays (i.e. 1-dimensional cones) of $\Sigma$ and the irreducible components of $\Delta_Z$.
\\The fan $\Sigma$ encodes various types of algebro-geometric information about $Z$. For instance, the variety $Z$ is smooth if and only if each top-dimensional cone of $\Sigma$ is $\GL(N)$-isomorphic to the standard octant $\R^r_{\geqslant 0} \subset \R^r
$.

Furthermore, in the case where $Z$ is smooth (which implies it has simple normal crossing boundary), the fan allows us to give a rather explicit description of the Picard group of $Z$. Indeed, each Cartier divisor can be moved via the torus action to a $\T$-invariant Cartier divisor, which is thus supported on the boundary. This provides a canonical set of generators of $\Pic(Z)$, and the kernel can be described as follows.
\\Write $
\Div^{\T}(Z) = \oplus_{l \in L} \Z Z_l \simeq \Z^L$ the abelian group of Weil divisors supported on the boundary.
The canonical map $q: \Div^{\T}(Z) \longrightarrow \Pic(Z)$ sends a divisor to its class; the map $p : M \fl \Div^{\T}(Z)$ sends a monomial $z^m$ to the principal divisor $\textrm{div}(z^m)$.
\begin{lem}[{\cite[3.4]{Fu}}]
The following sequence
$$0 \longrightarrow M \xrightarrow{ p} \Div^{\T}(Z) \xrightarrow{q} \Pic(Z) \longrightarrow 0$$
is exact.
\end{lem}
\noindent Let us rephrase this in term of coordinates, after fixing an isomorphism $N \simeq \Z^r$ and denoting $L = \{ u_1,\ldots, u_s \}$ the primitive generators of the 1-dimensional cones of $\Sigma$. Since by \cite[Lemma p.61]{Fu}, we have $\ord_{Z_l} (z^m) = \langle u_l, m \rangle$, we obtain $\textrm{div}(z^m) = \sum_{l \in L} \langle u_l, m \rangle Z_l$ and hence $p(m) = ( \langle u_l, m \rangle)_{l \in L}$. We deduce the following explicit description of $\Pic(Z)$:

\begin{cor} \label{lemma pic toric}
Let $u_l = (u_{l,1},\ldots,u_{l,r})$ for $l \in \{1,\ldots,s\}$. Then $\Pic(Z)$ is generated by the line bundles $\gO_{Z}(Z_l)$, with the $r$ relations:
$$ \gO_{Z}(\sum_{l=1}^{m} u_{l,1} Z_l) =\ldots=  \gO_{Z}(\sum_{l=1}^{m} u_{l,r} Z_l) =0.$$
In particular, the divisors in the $r$-tuple
$$\de = \sum_{l \in L} u_l \otimes Z_l \, \in  \, N \otimes \Div^{\T}(Z) \simeq (\Div^{\T}(Z))^r$$ are principal.
\end{cor}
We now want to describe how the fan $\Sigma$ encodes the intersection theory on $Z$. Each $1$-cycle in $Z$ being numerically equivalent to a sum of toric strata, it is enough to study the intersection numbers $(C \cdot Z_l)$, where $Z_l$ is a boundary component of $Z$ and $C$ is a 1-dimensional toric stratum, which is isomorphic to $\CP^1$ by properness.
The stratum $C$ is thus a rational curve with two marked points $p$ and $q$, which are the intersection points of $C$ with two components of $\De_Z$, denoted here by $Z_{p}$ and $Z_{q}$, with corresponding rays $\rho_p$ and $\rho_q$. The curve $C$ corresponds to a $(r-1)$-dimensional cone $\sigma_C$ of $\Sigma$, while the points $p$ and $q$ correspond to the maximal cones generated by $\langle\sigma_C,\rho_p\rangle $ and $\langle\sigma_C,\rho_q\rangle $.

\begin{lem}[{\cite[p. 99]{Fu}}] 
The primitive generators of the rays of the fan satisfy the following relation:
$$ u_{p} + u_{q} = - \sum_{u_l \in \sigma_C} (C \cdot Z_l)  u_l.$$
\end{lem}
\noindent Observing that we have $(C \cdot Z_p)= (C \cdot Z_q) =1$, and $(C \cdot Z_l)=0$ for any other $l$, this may be rewritten in a more synthetic way:
\begin{equation} \label{equ relation fan}\sum_{l \in L} (C \cdot Z_l) u_l =0.
\end{equation}

Note that this lemma holds even if $Z$ is not proper.
\\Finally, we have the following vanishing theorem for nef divisors on a proper toric variety.

\begin{prop} \label{lemma global generated}
Let $D$ be a nef Cartier divisor on a proper toric variety $Z$. Then $H^i(Z, \gO_Z(D))=0$ for $i>0$.
\end{prop}
\begin{proof}
The divisor $D$ being nef is equivalent to it being globally generated, by \cite[Theorem 3.1]{Mu}. Thus, the result follows directly from \cite[p. 74]{Fu}.
\end{proof}

Following \cite{NXY}, we will say that an $R$-scheme of finite type $\cZ$ is toric if there exists a toric $k$-scheme of finite type $\mathcal{Z}$, together with a toric morphism $t: \mathcal{Z} \fl \A^1_k$, such that $\cZ \simeq \mathcal{Z} \times_{\A^1} R$. Writing $\widehat{N}$ for the lattice of 1-parameter subgroups of the torus of $\mathcal{Z}$, such a scheme is described by a fan $\widehat{\Sigma}$ in $\widehat{N}_{\R}$, together with a linear map $\ord(t) : \lvert \widehat{\Sigma} \rvert \fl \R_{\ge 0}$, defined by $\ord(t)(n) = \ord_0 (t \circ n) $ for a 1-parameter subgroup $n: \G_m \rightarrow \mathcal{Z}$. Note that the map $\ord(t)$ recovers the function $t$ uniquely, since it is a monomial.
\begin{defn}
\label{toricdef}
Let $\X$ be a normal $R$-scheme of finite type, and $Y$ be a stratum of $\X_k$.
We say that $\X$ is \emph{toric along Y} if there exists a toric $R$-scheme $\cZ$, a stratum $W$ of $\cZ_k$ and a formal isomorphism over $R$ 
$$\widehat{\X_{/Y}} \simeq \widehat{\cZ_{/W}}.$$
\end{defn}

\subsection{Berkovich spaces} \label{subsec:Berkovich geometry}
Let $X$ be a normal variety over $K$. We denote by $X^{\an}$ the \emph{Berkovich analytification} of $X$. Set-theoretically, it consists of pairs $x=(\xi_x,v_x)$ where $\xi_x \in X$ and $v_x$ is a real-valued valuation on the residue field at $\xi_x$ extending the valuation $\ord_t$ on $K$. We denote by $\mathscr{H}(x)$ the completion of the residue field at $\xi_x$ with respect to $v_x$. 
We endow $X^{\an}$ with the coarsest topology such that \begin{itemize}
\itemsep0em
    \item[-] the forgetful map $\iota: X^{\an} \rightarrow X$, which maps $x=(\xi_x,v_x)$ to $\xi_x$, is continuous;
    \item[-] for any Zariski open $U \subseteq X$ and any function $f \in \gO_X(U)$, the map: $$\lvert f \rvert : U^{\an} \coloneqq \iota^{-1}(U) \rightarrow \R,$$ which evaluates $f$ at $x$ associating the value $\lvert f \rvert (x) \coloneqq\exp(-v_x(f(\xi_x)))$, is continuous.
\end{itemize}
This makes $X^{\an}$ a Hausdorff topological space, which is compact if and only if $X$ is proper over $K$.

Assume that $X/K$ is proper, and let $\X/R$ be a proper
model of $X$. By the valuative criterion of properness, for any $x =(\xi_x,v_x) \in X^{\an}$ there is a unique lift of the point $\xi_x$ to the valuation ring $\mathscr{H}(x)^{\circ}$ of $\mathscr{H}(x)$:
\begin{center}
    \begin{tikzcd}
    \Spec \mathscr{H}(x) \arrow[r, "\xi_x"] \arrow[d]  & \X \arrow[d] \\
    \Spec \mathscr{H}(x)^{\circ} \arrow[ru, dashed] \arrow[r] & \Spec R.
   \end{tikzcd}
\end{center}
The image of the closed point of $\Spec \mathscr{H}(x)^{\circ}$ under the extended morphism $\Spec \mathscr{H}(x)^{\circ} \rightarrow \X$ is called the center (or specialization) of $x$ and denoted by $c_\X(x)$. The map $c_{\X} : X^{\an} \fl \X_k$ turns out to be anticontinuous, i.e. the preimage of an open subset of $\X_k$ by $c_{\X}$ is closed in $X^{\an}$.

\subsection{Skeletons} \label{subsec:skeletons}
Let $X$ be a smooth proper variety over $K$. To every dlt model $\X$ of $X$, with special fiber $\X_k = \sum_{i \in I} a_i D_i$, we can associate a cell complex encoding the combinatorics of the intersections of the components $D_i$, whose faces are in one-to-one correspondence with strata of $\X_k$. 
\begin{defn}
We call \emph{simplex} a topological space, endowed with a $\Z$-affine structure, which is $\Z$-affine isomorphic to a space of the form:
$$\tau= \{ \sum_{i=0}^m a_i w_i =1 \} \subset \R^{m+1}, \quad \textrm{ for some $a_i \in \N_{\geqslant 0}$}.$$
\end{defn}

\begin{defn} \label{defn:star}
Let $\X$ be a dlt model of $X$. 
To each stratum $Y$ of $\X_k$ which is a connected component of $D_J$, we associate a simplex:
$$ \tau_Y = \{ w \in \R_{\geqslant 0}^{|J|} \, | \sum_{j \in J} a_j w_j =1 \}.$$
We define the cell complex $\mathcal{D}(\X_k)$ by the following incidence relations: $\tau_Y$ is a face of $\tau_{Y'}$ if and only if $Y' \subset Y$.
\end{defn}

Given any dlt model $\X$ of $X$ over $R$, there is a natural embedding $i_{\X}$ of the dual complex $\mathcal{D}(\X_k)$ into $X^{\text{an}}$, given as follows. 
The vertices $v_i$ of $\mathcal{D}(\X_k)$ are in one-to-one correspondence with irreducible components $D_i$ of the special fiber $\X_k = \sum_{i \in I} a_i D_i$, so that we set $$i_{\X}(v_{i}) = v_{D_i} \coloneqq a^{-1}_i \ord_{D_i},$$ where the valuation $\ord_{D_i}$ associates to a meromorphic function $f \in K(X) \simeq K(\X)$ its vanishing order along $D_i$ - the normalisation by $a^{-1}_i$ ensuring that $v_{D_i}(t) = 1$. A valuation given in this way, for some dlt model $\X$ of $X$, is called \emph{divisorial}.
One can now somehow interpolate between those divisorial valuations using \emph{quasi-monomial} valuations, in order to embed $\mathcal{D}(\X_k)$ into $X^{\text{an}}$:

\begin{prop}[{\cite[Proposition 2.4.4]{MN}}]
Let $\X$ be a dlt model of $\X$, with special fiber $\X_k = \sum_{i \in I} a_i D_i$.
Let $J \subset I$ such that $D_J = \cap_{j \in J} D_j$ is non-empty, and $Y$ a connected component of $D_J$, with generic point $\eta$. 
We furthermore fix a local equation $z_j \in \gO_{\X,\eta}$ for $D_j$, for any $j \in J$.
\\Then, for any $w \in \tau_Y = \{ w \in \R^{|J|}_{\geqslant 0} \,| \sum_{j \in J} a_j w_j =1 \}$, there exists a unique valuation
$$v_w : \gO_{\X,\eta} \fl \R_{\geqslant 0} \cup \{ + \infty\}$$
 such that for every $f \in \gO_{\X,\eta}$, with expansion $f = \sum_{\beta \in \N^{|J|}} c_{\beta} z^{\beta}$ (with $c_{\beta}$ either zero or unit), we have:
$$v_w(f) = \min \{ (w \cdot \beta) \,| \beta \in \N^{|J|},  c_{\beta} \neq 0 \},$$
where $( \; \cdot \; )$ is the usual scalar product on $\R^{|J|}$.
\end{prop}
The above valuation is called the quasi-monomial valuation associated with the data $(Y, w)$. Then
\begin{align*}
    i_\X: & \quad \mathcal{D}(\X_k) \rightarrow X^{\text{an}} \\
    & \quad \tau_Y \ni w \mapsto v_w
\end{align*}
gives a well-defined continuous injective map from $\mathcal{D}(\X_k)$ to $X^{\text{an}}$. 
\begin{defn}
We call the image of $\mathcal{D}(\X_k)$ by $i_{\X}$ the \emph{skeleton} of $\X$, written as $\Sk(\X) \subset X^{\emph{an}}$. It is a cell complex of dimension at most $\dim X$.
\end{defn}
\noindent By compactness of $\mathcal{D}(\X_k)$, $i_{\X}$ induces a homeomorphism between $\mathcal{D}(\X_k)$ and $\Sk(\X)$, so that we will sometimes abusively identify $\mathcal{D}(\X_k)$ with $\Sk(\X)$.
\begin{defn}
Let $Y$ be a stratum of $\X_k$. We define $\Star(\tau_Y)$ as the union of open faces in $\Sk(\X)$ whose closure contains $\tau_Y$.
\end{defn}

\subsection{Berkovich retractions} \label{subsec:Berkovich retractions}
Let $\X$ be a good dlt model of a smooth proper $K$-variety $X$. We can now define a retraction for the inclusion $\Sk(\X) \subset X^{\text{an}}$ as follows: for any  $v \in X^{\an}$, there exists a minimal stratum $Y \subseteq  \cap_{j \in J} D_j$ of $\X_k$ such that the center $c_{\X}(v)$ of $v$ is contained in $Y$. We then associate to $v$ the quasi-monomial valuation $\rho_{\X}(v)$ corresponding to the data $(Y, w)$ with $w_j = \frac{1}{q}v(z_j)$, where $z_j$ is a local equation of $q D_j$ at $c_{\X}(v)$, for some $q \in \N_{>0}$. This should be seen as a monomial approximation of the valuation $v$, with respect to the model $\X$.
\begin{defn}
The above map $\rho_{\X} : X^{\an} \longrightarrow \Sk(\X)$ is the Berkovich retraction associated with the model $\X/R$.
\end{defn}
\noindent The Berkovich retraction is continuous, restricts to the identity on $\Sk(\X)$, and by \cite{Th,Be} $\rho_{\X}$ is a strong deformation retraction, i.e. there is a homotopy between $\rho_\X$ and the identity on $X^{\an}$ that fixes the points of $\Sk(\X)$. It follows that $X^{\an}$ and $\Sk(\X)$ are homotopy equivalent.

Let $Y$ be a stratum of $\X_k$. The formal scheme $\widehat{\X_{/Y}}$ admits a generic fiber $\mathfrak{X}_Y$ in the sense of Berkovich, which is a Berkovich space and can be explicitly described as the open subset of $X^{\an}$:
$$\mathfrak{X}_Y = \{ x \in X^{\an} \lvert \, c_{\X}(v_x) \in Y \}.$$
It furthermore coincides with $\rho^{-1}_{\X}(\Star(\tau_Y)) \subset X^{\an}$. This Berkovich space comes with a retraction: $$\rho_Y : \mathfrak{X}_Y \fl \Star(\tau_Y), $$
which coincides with the restriction of the retraction $\rho_{\X}$. Thus, the restriction of $\rho_{\X}$ over $\Star(\tau_Y)$ only depends on the formal completion $\widehat{\X_{/Y}}$.
\vspace{10pt}

An explicit example of Berkovich retractions, which we will use as a local model in the rest of the paper, is as follows. Let $\T = \G^n_{m,K}$ be a torus, with character lattice $M$ and cocharacter lattice $N$. We view the elements $m$ of $M$ as rational functions on $\T$, so that its analytification $\T^{\text{an}}$ comes with a continuous map:
\begin{align*}
\val: & \, \T^{\text{an}} \fl N_{\R}, \\
& v_x \longmapsto (m \mapsto v_x(m)),
\end{align*}
under the identification $N_{\R} = \Hom(M, \R)$.
The notation $\val$ can be understood as follows: fix an isomorphism $N \simeq \Z^n$, so that $\T = \Spec K[M] \simeq \Spec K[X_1^{\pm},\ldots,X_n^{\pm}]$, and $v_x(m) = v_x(X^m)= \sum_{i=1}^n m_i v_x(X_i)$, so that the map $\val$ reads:
\begin{align*}
\val: & \, \T^{\text{an}} \fl \R^n, \\
& v_x \longmapsto (v_x(X_i))_{i=1,\ldots,n}.
\end{align*}
Since $v_x(X_i) = - \log \lvert X_i(x) \rvert$, this is the non-archimedean analog of the map $(\C^*)^n \fl \R^n$ sending $(z_1,\ldots,z_n)$ to $-(\log \lvert z_1 \rvert,\ldots,\log \lvert z_n \rvert)$.

\noindent The map $\val$ admits a continuous section $ \zeta : N_{\R} \fl \T^{\text{an}}$, sending a point $n \in N_{\R}$ to the Gauss point of the affinoid torus $\val^{-1}(n)$.
More explicitly, for $x \in \T^{\an}$, the valuation $\zeta( \val(x))$ is the valuation on the function field of $\T$ defined by the following formula:
$$\zeta(\val(x)) \big( \sum_{m \in M} \alpha_m z^m \big) = \min_{\alpha_m \neq 0} \big( \ord_t(\alpha_m) +v_x(z^m) \big).$$

Now let $\X /R$ be a regular toric model of $\T$, i.e. a regular toric $R$-scheme such that $\X \times_R \; \Spec K = \T$, which we assume to have reduced special fiber. Such a model is described by a regular fan $\hat{\Sigma} \subset \hat{N}_{\R} = N_{\R} \times \R_{\ge 0}$, whose cones intersect $N_{\R} \times \{0 \}$ only at the origin.
\\We consider the following closed subset of $\T^{\an}$:
$$\widehat{\X}_{\eta} := \{ v_x \in \T^{\an} \lvert \, v_x \; \text{has a center on} \; \X \},$$
which admits a Berkovich retraction:
$$\rho_{\X} : \widehat{\X}_{\eta} \fl \Sk(\X)$$
defined as above. In this case, the map $\rho_{\X}$ can be described explicitly as follows: let $\Sigma_1$ be the polyhedral complex obtained by intersecting the fan $\hat{\Sigma}$ with $N_{\R} \times \{1\}$. There is a natural identification between $\Sigma_1$ and $\mathcal{D}(\X_k)$, sending a vertex of $\cD(\X_k)$ to the primitive generator of the corresponding ray of $\hat{\Sigma}$, and then extending on each face by linearity. Moreover, it follows from \cite[Theorem A.4]{GJKM} that $\zeta ( \lvert \Sigma_1 \rvert) = \Sk(\X) \subset \T^{\an}$.
\begin{prop}[{\cite[Example 3.5]{NXY}}] \label{prop: toric retraction}
The equality:
$$\widehat{\X}_{\eta} = \val^{-1} (\lvert \Sigma_1 \rvert) $$
holds, and $\rho_{\X} = \val_{| \widehat{\X}_{\eta}}$.
\end{prop}
\begin{proof}
We start by proving the first equality. Let $x \in \T^{\an}$, we know from \cref{lem:toric retraction} below that $x$ has a center on $\X$ if and only if $\zeta(\val(x))$ has a center on $\X$. Thus, it is enough to prove that for $n \in N_{\R}$ and $y = \zeta(n)$, $y$ has a center on $\X$ if and only if $n \in \lvert \Sigma_1 \rvert$.
\\By \cite[Lemma  A.1]{GJKM}, if $y \in \zeta(N_{\R})$ has a center on $\X$, it must be the closure of a torus orbit $Y \subset \X_k$. By \cite[Theorem 6]{KK}, there exists a cone $\sigma \in \hat{\Sigma}$ such that the generic point of $Y$ is contained in the associated toric affine chart $\X_{\sigma} = \Spec R[\check{\sigma} \cap \hat{M}]$. In particular, for any monomial $z^m$ that is regular on $\X_{\sigma}$, we have $v_y(z^m) \ge 0$. In other words, writing $y = \zeta(n)$, we have $\langle n, m \rangle \ge 0$ for all $m \in \check{\sigma}$, so that $n \in \sigma$. Since $v_y(t) =1$, $y \in \zeta(\lvert \Sigma_1 \rvert)$.
\\By the same argument, if $n \in \lvert \Sigma_1 \rvert$, there exists a cone $\sigma$ such that $n \in \sigma$, which means that $v_{\zeta(n)}$ has positive value on each monomial $m \in \check{\sigma}$, and thus has a center on $\X_{\sigma}$ and in particular on $\X$.

To prove the second equality, since $\rho_{\X}$ is the identity on $\zeta(\lvert \Sigma_1 \rvert) = \Sk(\X)$, we merely have to prove that $\rho_{\X} = \rho_{\X} \circ \val$. However this follows directly from the definition of $\rho_{\X}$, and the fact that $c_{{\X}}(x) \in \overline{c_{{\X}}( \zeta (\val(x)))}$ for $x \in \widehat{\X}_{\eta}$ by \cref{lem:toric retraction}. Indeed, $\rho_{\X}(x)$ only depends on the values $v_x(z)$, where $z$ is a local equation for a component of $\X_k$ at $c_{\X}(x)$. Since $\X$ is a toric model, these local equations can be taken to be monomials, so that the result follows from the fact that $x$ and $\zeta(\val(x))$ take the same values on monomials.
\end{proof}

\begin{lem} \label{lem:toric retraction}
Let $x \in \T^{\an}$. Then $x$ has a center on $\X$ if and only if $\zeta(\val(x))$ has a center on $\X$. Moreover, if this holds, we have $c_{{\X}}(x) \in \overline{c_{{\X}}( \zeta (\val(x)))}$.
\end{lem}

\begin{proof}
Let $\X \subset \bar{\X}$ be a toric compactification of $\X$, i.e. a proper toric $R$-scheme containing $\X$ as a torus-invariant open subset. By the valuative criterion of properness, any valuation on $\T^{\an}$ has a center on $\bar{\X}$. We write $c_{\bar{\X}}(x)$ for the center of $x \in \T^{\an}$. 
We start by proving that $c_{\bar{\X}}( \zeta (\val(x))$ is the generic point of the closure $Z$ of the torus orbit $O(\sigma)$ in $\bar{\X}$ containing $c_{\bar{\X}}(x)$ (in particular, $c_{\bar{\X}}(x)$ must be contained in the toric interior of $Z$).
We may work on the toric affine chart $\bar{\X}_{\sigma} = \Spec R[\check{\sigma} \cap \hat{M}]$ associated with $Z$. 
Since the valuation $\zeta(\val(x))$ is monomial, it is enough to prove that $\zeta(\val(x))(z^m) = v_x(z^m) \ge 0$ for $m \in \check{\sigma} \cap M$ and that $\zeta(\val(x))(z)>0$ for $z$ a local equation of any torus invariant divisor containing $Z$, to have that $c_{\bar{\X}}( \zeta (\val(x)))$ lies in $Z$. Since $z^m$ is regular on $\bar{\X}_{\sigma}$, the first condition holds; the local equation $z$ is monomial and $\zeta(\val(x))(z) = v_x(z) > 0$ since $Z$ contains $c_{\bar{\X}}(x)$. 

Now assume that $v_x$ is centered on $\X$, i.e. $c_{\bar{\X}}(x) \in \X$. Since $c_{\bar{\X}}(x)$ is contained in the interior of the toric stratum $Z$  and lies in $\X$, so does the generic point of $Z$, i.e.  $c_{\bar{\X}}(\zeta(\val(x))) \in \X$. This implies that $\zeta(\val(x))$ has a center on $\X$.
Conversely, if $\zeta(\val(x))$ has a center $Z$ on $\X$, then $Z \subset \X$ and $c_{\bar{\X}}(x) \in Z$ as mentioned above; thus, $c_{\bar{\X}}(x) \in \X$, which concludes the proof.
\end{proof}

\subsection{Affinoid torus fibrations and integral affine structures}
\label{subsec:defn affine structure}

Let $X$ be a smooth proper variety over $K$.
\begin{defn} \label{defn:affinoid torus fibration}
Let $\rho : X^{\text{an}} \fl B$ be a continuous map to a topological space $B$. For any point $b \in B$, we say that $\rho$ is an \emph{affinoid torus fibration at $b$} if there exists an open neighbourhood $U$ of $b$ in $B$, such that the restriction to $\rho^{-1}(U)$ fits into a commutative diagram:
$$\begin{tikzpicture}
\matrix(m)[matrix of math nodes, row sep=3em,
    column sep=2.5em, text height=1.5ex, text depth=0.25ex]
{\rho^{-1}(U) & \val^{-1}(V)  \\
    U  & V, \\ };
	\path 
	(m-1-1) edge[->] node[auto] {$\simeq$} (m-1-2)
			edge[->] node[auto] {$ \rho $} (m-2-1)
         (m-1-2) edge[->] node[auto] {$\val$} (m-2-2)
         (m-2-1) edge[->] node[auto] {$\simeq$} (m-2-2);
\end{tikzpicture}$$
$V$ being an open subset of $\R^n$, the upper horizontal map an isomorphism of analytic spaces, the lower horizontal map a homeomorphism, and the map $\val$ defined as in \cref{subsec:Berkovich retractions}.
\end{defn}

\begin{ex} \label{ex:affinoid torus over max face}
It follows from the definition of good dlt model $\X$ of $X$ that the Berkovich retraction $\rho_{\X}: X^{\an} \rightarrow \Sk(\X)$ 
is an affinoid torus fibration over the interior of the maximal faces $\tau$ of $\Sk(\X)$. Indeed, the retraction over $\textrm{Int}({\tau})$ only depends on the formal completion of $\X$ along the corresponding 0-dimensional stratum $p$. The pair $(\X, \X_k)$ is snc at $p$, hence the claim.
\end{ex}
\begin{ex}
If $\X /R$ is a toric model of $X = \T$, it follows from \cref{prop: toric retraction} that the Berkovich retraction:
$$\rho_{\X}: \widehat{\X}_{\eta} \fl \Sk(\X)$$
is an affinoid torus fibration over the interior of $\Sk(\X)$.
This also holds when $X$ is a regular proper toric variety over $K$, and $\X$ a regular proper toric model, by \cite[Theorem A.4]{GJKM}.
\end{ex}
Note that the above definition implies that $B$ is a topological manifold at $b$; in the case of a Berkovich retraction $\rho_{\X}$, this does not necessarily hold at every point of $\Sk(\X)$.
\vspace{10pt}

Given a continuous map $\rho : X^{\text{an}} \fl B$, we denote by $B^{\textrm{sm}}$ the locus of points in $B$ where $\rho$ is an affinoid torus fibration at; we call $B \setminus B^{\textrm{sm}}$ the \emph{discriminant} or {singular locus} of $B$. $B^{\textrm{sm}}$ is endowed with an integral affine structure; we recall the definition and describe such structure.

\begin{defn}
An \emph{integral affine structure} on a topological
manifold is an atlas of charts with transition functions in  $\textrm{GL}_n(\mathbb{Z}) \ltimes \mathbb{R}^n$. 
\end{defn}

\begin{defn}
An \emph{integral affine function} on an open subset of $\R^n$ is a continuous real-valued function locally of the form $f(x_1,\ldots,x_n)=a_1 x_1 + \ldots + a_n x_n +b $, with $a_i \in \Z$ and $b \in \R$. We denote by $\textrm{Aff}_{\R^n}$ the sheaf of integral affine functions on $\R^n$.
\end{defn}

\begin{lem}[{\cite[2.1]{KontsevichSoibelman}}]
An integral affine structure on a topological manifold $M$ is equivalent to the datum of a subsheaf $\mathrm{Aff}_{M}$ of the sheaf of continuous functions on $M$ such that $(M, \mathrm{Aff}_{M})$ is locally isomorphic to $(\R^n, \mathrm{Aff}_{\R^n})$.
\end{lem}

If $\rho$ is an affinoid torus fibration over $B^{\textrm{sm}} \subseteq B$, the integral affine structure on $B^{\textrm{sm}}$ is the pull-back of $\textrm{Aff}_{\R^n}$ via the charts in \cref{defn:affinoid torus fibration}. An alternative description of this structure is given in \cite[4.1, Theorem 1]{KontsevichSoibelman}: let $U \subset B^{\textrm{sm}}$ be a connected open subset. Then if $h$ is an invertible analytic function on $\rho^{-1}(U)$, its modulus $\lvert h \rvert$ is constant on the fibers of $\rho$ by the maximum principle, so that it defines a continuous function on the base. We now have:
$$ \mathrm{Aff}_{ B^{\textrm{sm}}}(U) = \{ - \log\lvert h \rvert \, | \, h \in \mathcal{O}^{\times}_{X^{\an}} (\rho^{-1}(U))\}.$$

\begin{rem}
Given an integral affine structure on a topological manifold $M$, there is a \emph{monodromy representation} $$T: \pi_1(M) \rightarrow \textrm{GL}_n(\mathbb{Z}) \ltimes \mathbb{R}^n$$ defined by covering a loop in $M$ by affine charts and composing the corresponding transition functions. See \cite[2.2]{KontsevichSoibelman} for more details.
\end{rem}
\subsection{The Calabi--Yau case}
Let $X/K$ be a smooth $n$-dimensional Calabi--Yau variety: here, this means that $K_X = \gO_X$ (note that this includes for instance the case of abelian varieties). This is the main case we are interested in for applications.
\\We consider a distinguished class of models of $X$, which we call \emph{good minimal dlt} models. Note that other references may define them in a slightly different way.
\begin{defn}
Let $X/K$ be a Calabi--Yau variety. A minimal dlt model of $X$ is a dlt model $\X/R$, such that the logarithmic relative canonical divisor is trivial, i.e.
$$ K^{\log}_{\X/R} \coloneqq K_{\X/R} + \X_{k, \red} - \X_k \sim \gO_\X.$$ 
We furthermore say that a minimal model $\X$ is good if it is good in the sense of \cref{defn:dlt}.
\end{defn}
The existence of such models is known when $X$ is defined over an algebraic curve (and is expected to hold in the general case).
\begin{theo}[{\cite[Theorem 1.13]{NXY}}]
Let $X/K$ be a projective Calabi--Yau variety, and assume that $X$ is defined over an algebraic curve.
Then there exists a good minimal model $\X/R$ of $X$. Furthermore, there exists a finite extension $K'/K$ such that the base change $X_{K'}$ admits a good minimal model with reduced special fiber.
\end{theo}
Such models are not unique, but they turn out to have the same skeleton inside $X^{\an}$ by \cite{NX} (even though the triangulation may differ), which is thus a canonical piecewise-linear space associated with $X$.
\begin{defn}
Let $X/K$ be a Calabi--Yau variety. The \emph{essential skeleton} $\Sk(X) \subset X^{\an}$ is the skeleton of any minimal model $\X/R$ of $X$.
\end{defn}
The essential skeleton can also be defined intrinsically as the locus where the weight function associated with a non-vanishing section $\omega \in H^0(X, K_X)$, $\text{wt}_{\omega} : X^{\an} \fl \R$ defined in \cite{MN} reaches its minimum; we refer the reader to $\cite{MN}$ and $\cite{NX}$ for details.
\begin{defn}
Let $X/K$ be a Calabi--Yau variety. We will say that $X$ is \emph{maximally degenerate} if the skeleton $\Sk(X)$ has maximal dimension, i.e. $\dim \Sk(X) = n$.
\end{defn}

\begin{ex}
In the $2$-dimensional case, maximally degenerate Calabi--Yau surfaces coincide with either $K3$ surfaces of Type III, or maximally degenerate abelian surfaces.
\end{ex}

The maximally degenerate case is of great interest to mirror symmetry. In such case, Kontsevich and Soibelman conjecture that the Gromov-Hausdorff limit of the (rescaled) Kähler Ricci-flat metrics on the fibers $X_t$ is the essential skeleton of $X$, endowed with a metric which is given in local affine coordinates by the Hessian $\frac{\partial^2 \phi}{\partial x_i \partial x_j}$ of a convex function $\phi$.
This is known in the case of abelian varieties by the work of \cite{Odaka2018}, and for Fermat hypersurfaces in $\CP_{\C}^{n+1}$ by \cite{Li}. See also the results in \cite{Li2} for recent progress on this conjecture in the general case.
\\More precisely, the metric spaces $(X_t, \omega_t)$ are conjectured to ``look like'' the total space of a Lagrangian torus fibration over $\Sk(X)$, submersive away from a singular locus of real codimension 2; the affine structure on the base being induced by action-angle coordinates. Building on these considerations, it is reasonable to expect this affine structure to be non-singular in (real) codimension one.

One of the main motivations in non-archimedean mirror symmetry is to reconstruct this affine structure through Berkovich spaces, interpreting Berkovich retractions as non-archimedean avatars of Lagrangian torus fibrations. The main theorem in \cite{NXY} establishes the following.

\begin{theo}[{\cite[Theorem 6.1]{NXY}}]
Let $X/K$ be a maximally degenerate projective Calabi--Yau variety, and let $\X/R$ be a good minimal dlt model of $X$ with reduced special fiber. Then the Berkovich retraction
$$ \rho_{\X} : X^{\an} \fl \Sk(X)$$
is an affinoid torus fibration away from the codimension 2 locus of the triangulation induced by the homeomorphism $\Sk(X) \simeq \mathcal{D}(\X_k)$.
\end{theo}
\noindent This statement is proved by showing that $\X$ is toric along the 1-dimensional strata of the special fiber, which yields on the way a complete description of such models along these strata. 
This description also provides a way to compute the singular $\Z$-affine structure induced on $\Sk(X)$, as well as its monodromy, see \cref{sec:curve}.
\begin{ex}If $S$ is a K3 surface of Type III, and $\X/R$ a good minimal dlt model of $S$, then the map $\rho_{\X}$ is an affinoid torus fibration away from the vertices of $\mathcal{D}(\X_k)$.
\\The induced $\Z$-affine structure with isolated singularities on the 2-sphere matches the one constructed in \cite[1.2]{GHK} and \cite[Proposition 3.10]{Eng}, and has no singularity at a vertex $v_D$ if and only if the corresponding component $D$ of $\X_k$ is toric.
\end{ex}

\section{Toric structure along toric strata}
\label{Section toric thm}

In this section we prove \cref{intro:main thm Z}. We recall the statement and fix the notation.

\begin{thmA}
Let $X/K$ be a smooth projective variety of dimension $n$, and $\X/R$ be a dlt model of $X$ with reduced special fiber $\X_k= \sum_\alpha D_\alpha$, such that every $D_\alpha$ is a Cartier divisor.
\\Let $Z= D_0 \cap D_1 \cap \ldots \cap D_{n-r}$ be an $r$-dimensional stratum of $\X_k$, such that:
\begin{itemize}
\itemsep0pt
\item $\mathring{Z} \subset Z$ is a torus embedding, where $\mathring{Z}= Z \setminus \cup_{\alpha \neq 0,1,\ldots,n-r}D_\alpha$;
\item the conormal bundle $\nu_{Z/ \X}^*$ is a nef vector bundle on $Z$;
\item for each $\alpha \notin \{0,...,n-r\}$, the intersection $D_{\alpha} \cap Z$ is either empty or connected.
\end{itemize}
Then the formal completion $\widehat{\X_{/Z}}$ is isomorphic to the formal completion of the normal bundle $\mathcal{N}=\nu_{Z/\X}$ along the zero section. In particular, $\X$ is toric along $Z$ (in the sense of \cref{toricdef}).
\end{thmA}

\noindent Note that the assumptions in \cref{intro:main thm Z} imply that $Z$ is the smooth complete intersection of the irreducible components $D_j$ of $\X_k$ containing $Z$, and thus has simple normal crossing boundary, see \cref{rem dlt}.
Since $Z$ is a complete intersection, the conormal bundle $\nu^*_{Z/ \X}$ is the direct sum of the line bundles $\gO_Z(-D_j)$. Hence, the nefness assumption simply means that the $D_j$'s containing $Z$ are anti-nef divisors on $Z$.

As an immediate consequence of the theorem, we prove that 

\begin{corC}
The retraction $\rho_\X: X^{\an} \rightarrow \Sk(\X)$ is an $n$-dimensional affinoid torus fibration over $\Star(\tau_Z)$. In particular, the integral affine structure induced by $\rho_\X$ on the complement of the faces of $\Sk(\X)$ of codimension $\geqslant 2$ extends to $\Star(\tau_Z)$ with no singularities.
\end{corC}
\begin{proof}
Although this follows from \cref{intro:main thm Z} by \cite[Theorem 6.1]{NXY} (end of the proof) and by \cite[\S 3.4]{NXY}, we sketch the proof for reader's convenience. 
\\As mentioned in \cref{subsec:Berkovich retractions}, the retraction $\rho_{\X}$ over $\Star(\tau_Z)$ only depends on $\widehat{\X_{/Z}}$, so that by \cref{intro:main thm Z} we may assume that $\X$ is a toric $R$-scheme. The equality $\rho_{\X} = \val$ now holds over $\Star(\tau_Z)$ by \cref{prop: toric retraction}, so that it follows from \cref{defn:affinoid torus fibration} that $\rho_{\X}$ is an affinoid fibration over $\Star(\tau_Z)$.
\end{proof}


\subsection{Notation and strategy}

We set $J=\{0,1,\ldots, n-r\}$ such that $Z=\cap_{j \in J} D_j$. Since for every irreducible component $D$ of $\X_k$, the intersection $D \cap Z$ is connected by assumption, this allows us to denote by $D_l$ with $l \in L$ the components of $\X_k$ intersecting $Z$ transversally along $Z_l \coloneqq Z \cap D_l$, so that the toric boundary of $Z$ is given by $\De_Z = \sum_{l \in L} Z_l$.

\begin{rem} \label{rem dlt}
The dlt assumption on $\X$ and the toricness of $Z$ ensure that $Z$ is smooth, and that $(Z, \De_Z)$ is an snc pair.
Indeed, the singular locus of $Z$ is a union of torus invariant subvarieties (see \cite[Proposition 11.1.2]{CoxLittleSchenck2011}), hence the generic point of a component of the singular locus is the generic point of a stratum of $\De_Z$. However, $(Z,\De_Z)$ is a dlt pair, thus snc at the generic point of each stratum of $\De_Z$.
\end{rem}

\begin{rem}
The smoothness of $Z$ and the assumption that the components of $\X_k$ are Cartier divisors imply that $\X$ is regular at any point of $Z$. Indeed, for any point $p \in Z$ and $j \in J$, let $z_j \in \gO_{\X,p}$ be a local equation of $D_j$ at $p$. As $\gO_{Z,p} \simeq \gO_{\X,p} / (z_0,\ldots,z_{n-r})$ is a regular local ring of dimension $r$, $(z_0,\ldots,z_{n-r})$ can be extended to form a regular system of parameters for $\gO_{\X,p}$.
\end{rem}

We denote by $\Sigma \subset N_{\R}$ the fan of $Z$. Its rays are given by $\R_{\geqslant 0} u_l$ for $l \in L$, with primitive generators $u_l$; the maximal cones of $\Sigma$ are in bijection with the set of unordered $r$-tuples $\{i_1, \ldots,i_r\} \in L^r$ such that $\cap_{\beta=1}^{r}D_{i_\beta} \cap Z \neq \varnothing$. For a maximal cone $\sigma$ of $\Sigma$, we write $L_\sigma \coloneqq \{l \in L \,|\, u_l \in \sigma \}$.

\begin{lem} \label{lemma det}
For any maximal cone $\sigma$ of $\Sigma$, we have $\det ((u_l)_{l \in L_\sigma}) = \pm 1.$
\end{lem}
\begin{proof}
The smoothness of $Z$ (see \cref{rem dlt}) implies that the primitive generators of $\sigma$ form a $\Z$-basis of $N$, which is equivalent to the condition $\det ((u_l)_{l \in L_\sigma}) = \pm 1.$
\end{proof}

Let $\mathcal{N} \coloneqq \nu_{Z/ \X} \xrightarrow{p} Z$ be the normal bundle of $Z$ in $\X$, and denote by $Z \subset \mathcal{N}$ the zero section. For each $j$, we choose a Cartier divisor $F_j$ on $Z$ such that $ \gO_{\X}(D_j)_{| Z}= \gO_Z(F_j)$, so that $ \mathcal{N}= \oplus_{j \in J} \gO_Z(F_j)$.
Since any Cartier divisor on $Z$ is linearly equivalent to a toric one, for any $j=1,\ldots, n-r$, there exist integers $\lambda_{j,l}$ such that 
\begin{equation} \label{equ:lambda relation}
    \gO_Z(F_j) = \gO_Z\big(- \sum_{l \in L} \lambda_{j,l}Z_l\big).
\end{equation}
For $j=0$ we set $\lambda_{0,l} \coloneqq 1 - \sum_{j \in J \setminus \{0\}} \lambda_{j,l}$ and verify that
\begin{align*}
    \gO_Z(F_0) 
    =  \gO_{\X}(D_0 - \X_k)_{| Z} 
     = \gO_Z\Big(-\sum_{j \in J \setminus \{0\}} F_j - \sum_{l \in L} Z_l \Big)
    = \gO_Z(- \sum_{l \in L} \lambda_{0,l} Z_l).
\end{align*}
We obtain that  $\mathcal{N} = \oplus_{j \in J} \gO_Z\big(- \sum_{l \in L} \lambda_{j,l}Z_l\big)$ and for all $l$ in $L$ \begin{equation} \label{equ special fb}
   \sum_{j \in J} \lambda_{j,l} =1.
\end{equation}
The normal bundle $\mathcal{N}$ is a toric variety of dimension $n+1$. The corresponding fan $\hat{\Sigma}$ lies in $N_\R \times \R^J$ and consists of the following cones and their faces (see \cite[\S 7.3]{CoxLittleSchenck2011} for a reference). Let $e_0,\ldots,e_{n-r}$ be the standard basis of $\R^J$; given a cone $\sigma \in \Sigma$, we have 
$$ \hat{\sigma}= \textrm{Cone}( (0,e_0),\ldots,(0,e_{n-r}),(u_l,(\lambda_{j,l})) \, | \, u_l \in \sigma) \in \hat{\Sigma}.$$
In particular, we denote the rays of $\hat{\Sigma}$ by $$v_j=(0,e_j) \textrm{ for }j \in J, \quad v_l=(u_l, (\lambda_{j,l})) \textrm{ for }l \in L.$$

\begin{prop} For any 1-dimensional toric stratum $C \subseteq Z$
\begin{equation} \label{equ fan}
\sum_{j \in J} (C \cdot D_j ) v_j + \sum_{l \in L} (C \cdot D_l) v_l =0 \; \text{in} \; N_{\R} \times \R^J.
\end{equation}
\end{prop}
\begin{proof}
The relation in \cref{equ fan} boils down to the following:
$$
\begin{cases}
\sum_{l \in L} (C \cdot D_l) u_l=0 \\
(C \cdot D_j ) + \sum_{l \in L} \lambda_{j,l} (C \cdot D_l)  =0 \; \forall j \in J. \\
\end{cases}
$$
The first one follows directly from \cref{equ relation fan} in the fan $\Sigma$ of $Z$; the second comes from the construction of $\lambda_l$, and in particular from $ C \cdot  D_j= C \cdot F_j= - C \cdot \sum_{l \in L} \lambda_{j,l} Z_l $.
\end{proof}

The map 
$$   \ord(t): \lvert \hat{\Sigma} \rvert \subset N_\R \times \R^J \rightarrow \R_{\geqslant 0} \quad
    (u,w) \mapsto \sum_{j=0}^{n-r} w_j $$
is $\Z$-linear, sends all the primitive generators of the rays of $\hat{\Sigma}$ to $1$ by \cref{equ special fb}, and is compatible with $\hat{\Sigma}$ and the fan of $\A^1_k$. Thus, it induces a toric morphism $t: \mathcal{N} \rightarrow \A^1_k$ whose fiber over $0$ is the toric boundary of $\mathcal{N}$. The base change $\mathscr{N} \coloneqq \mathcal{N} \times_{\A^1} R$ to $R$ is a toric $R$-scheme, whose generic fiber is isomorphic to $\mathbb{G}_{m,K}^n$. The special fiber $\mathscr{N}_k$ can be written as $\cN_k = \sum_{i \in J \cup L} E_i$, where the combinatoric of intersections between components is exactly the same as in $\X_k$.
\vspace{5pt}

We prove \cref{intro:main thm Z} by constructing a formal isomorphism
$$ f : \widehat{\X_{/Z}} \xrightarrow{\simeq} \widehat{\mathscr{N}_{/Z}}.$$ More specifically, we proceed as follows. We set the notations $\mathfrak{X} =\widehat{\X_{/Z}} $ and $\mathfrak{N} = \widehat{\cN_{/Z}}$.
\begin{itemize}
    \item (\cref{proof:divisors,section:one max cone}) 
    Let $\sigma \in \Sigma$ be a maximal cone. Denote by $Z_\sigma$ and $\mathcal{N}_{\sigma}\coloneqq \mathcal{N}_{Z_{\sigma}/ \X}$ the corresponding toric affine charts in $Z$ and $\mathcal{N}$ respectively.
    This induces an open formal subscheme of $\mathfrak{N}$, which we denote by $\mathfrak{N}_{\sigma}$. 
    We construct a morphism 
    $$f_\sigma: \mathfrak{X} \setminus \big( \cup_{l \in L \setminus L_\sigma} D_l\big) \eqqcolon  \mathfrak{X}_{\sigma} \, \rightarrow \mathfrak{N}_{\sigma} ,$$
   in a similar manner to \cite{NXY}: we construct $n+1$ divisors $W^\sigma_j$ and $W^\sigma_i$ on $\X$, whose defining equations on the chart $\mathfrak{X}_{\sigma}$ yields the morphism $f_{\sigma}$. The equations are induced by sections of $\gO_Z(W^\sigma_j)$ and $\gO_Z(W^\sigma_i)$: these are first constructed on $Z$, then extended to $\mathfrak{X}$ by the nef condition on the conormal bundle $\nu_{Z/ \X}^*$ assumed in \cref{intro:main thm Z}, which ensures the vanishing of higher cohomology groups for the tensor powers of $\nu_{Z/ \X}^*$.
    
    \item (\cref{section:two max cones,proof:morphism}) Let $\sigma$ and $\sigma'$ be two maximal cones of $\Sigma$ intersecting along a face of codimension one. We establish relations among the respective divisors and construct sections on $\mathfrak{X}_{\sigma'}$ from those on $\mathfrak{X}_{\sigma}$. This allows us to prove that the morphisms $f_{\sigma}$ on the charts $\mathfrak{X}_{\sigma}$'s can be chosen so that they are compatible on the overlaps $\mathfrak{X}_{\sigma}  \cap \mathfrak{X}_{\sigma'}$. This yields a well defined morphism $f$ which extends
    the identity on $Z$ and preserves the ideal $\mathscr{I}_Z$, so that it
    turns out to be an isomorphism.
\end{itemize}


\subsection{Construction of the divisors} \label{proof:divisors}
We set 
$$\de := \sum_{l \in L} u_{l} \otimes D_l \, \in \, N \otimes \Div_0(\X) \simeq (\Div_0(\X))^r;$$ 
this is an $r$-tuple of divisors on $\X$. Moreover, the restriction of any of these to $Z$ is a principal divisor by \cref{lemma pic toric}.
Given a maximal cone $\sigma$ of $\Sigma$, for any $i \in L_{\sigma}$, we define 
$$
W^{\sigma}_{i} \coloneqq - \frac{\det (\de , (u_{l})_{l \in L_{\sigma} \setminus \{i\}})}{\det (u_i , (u_{l})_{l \in L_{\sigma} \setminus \{i\}})} \in \Div_0(\X)$$
where the column vectors $u_l$ are in the same order in the numerator and in the denominator, and the denominator has value $\pm 1$ by \cref{lemma det}.

\begin{lem} \label{lemma:W_i}
The divisor $W^{\sigma}_i$ has multiplicity $-1$ along $D_i$, multiplicity $0$ along $D_l$ for $l \in L_{\sigma} \setminus \{i\}$, and 
along $D_j$ for $j \in J$. In other words, we may write:
$$W^{\sigma}_i = -D_i + \sum_{l \in L \setminus L_{\sigma}} c_{i,l} D_l$$
for some coefficients $c_{i,l} \in \Z$. Moreover, the restriction of $W^\sigma_i$ to $Z$ is principal.
\end{lem}

\begin{proof}
The statement on the multiplicities follows from the definition of $W^\sigma_i$, as
$$W^{\sigma}_i = - \sum_{l \in L}  \frac{\det (u_l , (u_{l'})_{l' \in L_{\sigma} \setminus \{i\}})}{\det (u_i , (u_{l'})_{l' \in L_{\sigma} \setminus \{i\}})}  D_l.
$$ Moreover, $W^\sigma_i$ is a linear combination of the divisors of the $r$-tuple $\de$, hence its restriction to $Z$ is principal by \cref{lemma pic toric}.
\end{proof}

For $j \in J$, we define the divisor on $\X$ 
\begin{align} \label{equ:W_j}
\begin{split}
W^{\sigma}_j 
& \coloneqq -D_j  
- \sum_{l \in L} \lambda_{j,l} D_l 
- \sum_{i \in L_{\sigma}} \lambda_{j,i} W^{\sigma}_i \\
& = -D_j  
- \sum_{l \in L \setminus L_{\sigma}} \lambda_{j, l} D_{l} 
- \cancel{\sum_{l \in L_\sigma} \lambda_{j,l} D_l}
- \sum_{i \in L_{\sigma}} \lambda_{j,i} \big( \cancel{ - D_{i}} +  \sum_{l \in L \setminus L_{\sigma}} c_{i,l} D_l \big) \\
& = -D_j  + \sum_{l \in L \setminus L_{\sigma}} d_{j, l} D_l \quad \textrm{ with } \quad d_{j, l}=- \lambda_{j,l} - \sum_{i \in L_{\sigma}} \lambda_{j,i} c_{i,l}.
\end{split}
\end{align}
The restriction of $W^{\sigma}_j$ to $Z$ is a principal divisor, as the ${W^{\sigma}_i}_{| Z} $ are principal and ${-D_j}_{| Z}$ is linearly equivalent to $ \sum_{l \in L} \lambda_{j,l} Z_l$
by \cref{equ:lambda relation}. 

\begin{lem}
The relation
$ \sum_{j \in J} W^{\sigma}_j + \sum_{i \in L_{\sigma}} W^{\sigma}_i = - \sum_{j \in J} D_j - \sum_{l \in L} D_l$
holds.
\end{lem}

\begin{proof}
Write $W \coloneqq \sum_{j \in J} W^{\sigma}_j + \sum_{i \in L_{\sigma}} W^{\sigma}_i \in \Div_0(\X).$
We have
\begin{align*}
\textrm{for }j \in J \quad
\ord_{D_j}(W) 
& = \ord_{D_j}(W^{\sigma}_j)=-1 \\
\textrm{for }i \in L_{\sigma}  \quad
\ord_{D_i}(W)
& = \ord_{D_i}(W^{\sigma}_i)=-1 \\
\textrm{for }l \in L \setminus L_{\sigma}  \quad
\ord_{D_l}(W)
& = \sum_{j \in J} d_{j,l} + \sum_{i \in L_\sigma} c_{i,l}  = - \sum_{j \in J} \lambda_{j,l} + \sum_{i \in L_\sigma} c_{i,l} ( 1 - \sum_{j \in J} \lambda_{j,i}) = -1
\end{align*}
by \cref{lemma:W_i}, \cref{equ:W_j} and \cref{equ special fb}.
\end{proof}

\subsection{Construction of the sections for a maximal cone} \label{section:one max cone}

Let $\sigma$ be a maximal cone of $\Sigma$.
We denote by $\Ld^{\sigma}_j$ and $\Ld^{\sigma}_i$ the line bundles on $\mathfrak{X}$ induced respectively by $\gO_{\X}(W^{\sigma}_j)$ for $j \in J$, and by $\gO_{\X}(W^{\sigma}_{i})$ for $i \in L_\sigma$.
Since $W^{\sigma}_j$ and $W^{\sigma}_{i}$ are principal on $Z$, the restrictions ${\Ld^{\sigma}_j}_{|Z}$ and ${\Ld^{\sigma}_i}_{|Z}$ are trivial line bundles on $Z$, thus we may choose non-vanishing global sections $s^{\sigma}_j$ and $s^{\sigma}_{i}$ on $Z$.

We now lift the sections $s^{\sigma}_j$ and $s^{\sigma}_{i}$ to global sections of  $\Ld^{\sigma}_j$ and $\Ld^{\sigma}_i$, which we still denote by $s^{\sigma}_j$ and $s^{\sigma}_{i}$. 
Indeed, for any $n \geqslant 1$, write $(\X/Z)_{n}$ for the (non-reduced) subscheme of $\X$ defined by the ideal $\mathscr{I}^n_Z$.  In the exact sequence 
$$ H^0((\X/Z)_{n}, \Ld^{\sigma}_j) \fl H^0((\X/Z)_{n-1}, \Ld^{\sigma}_j) \fl H^1(Z, (\nu^*_{Z/ \X})^{\otimes n}),$$ 
and in the analogous one for $\Ld^{\sigma}_i$,
the right-hand vanishes: the conormal bundle is a direct sum of line bundles on $Z$ which are nef by the hypothesis in \cref{intro:main thm Z} and so are its positive tensor powers, thus their first cohomology group vanishes by \cref{lemma global generated}. We thus extend the sections constructed above to all of the $(\X/Z)_n$ by induction, which yields an extension to $\mathfrak{X} = \varprojlim_n (\X/Z)_n$.

\begin{lem}
\label{lemma sections}Viewing the restrictions of $s^{\sigma}_j$ and $s^{\sigma}_{i}$ to $\mathfrak{X}_{\sigma}$ as functions through the isomorphism $(\Ld^{\sigma}_j)_{| \X_{\sigma}} \simeq \gO_{\X_{\sigma}}(-D_j)$ induced by \cref{equ:W_j} (and similarly for $i$), the $s^{\sigma}_j$ and $s^{\sigma}_{i}$ are equations for $D_{j}$ and $D_i$ on $\mathfrak{X}_{\sigma}$, and thus
$$ w_{\sigma} \coloneqq t \cdot \prod_{j \in J} (s_{j}^{\sigma})^{-1} \cdot \prod_{i \in L_\sigma} (s_{i}^{\sigma})^{-1}$$ is an invertible function on $\mathfrak{X}_{\sigma}$.
\end{lem}

\begin{proof}
We show that $s^{\sigma}_{i}$ is an equation for $D_{i}$ on $\mathfrak{X}_{\sigma}$, the proof is analogous for $s^{\sigma}_j$.
\\Let $\mathcal{U}$ be an open cover of $\X \setminus \big( \cup_{i' \notin J \cup L_\sigma} D_{i'} \big) $ such that ${D_i}_{|U}= \textrm{div}(g_U)$ for any $U \in \mathcal{U}$; this is possible as $D_{i}$ is a Cartier divisor. On $U$, ${W^{\sigma}_{i}}_{|U}=- {D_{i}}_{|U}=\textrm{div}(g_U^{-1})$ and
\begin{align*}
    \Ld^{\sigma}_i (\mathfrak{X}_\sigma \cap U) 
    & \xrightarrow{\simeq} \mathcal{O}_{\mathfrak{X}_\sigma} (\mathfrak{X}_\sigma \cap U) \\
    f & \mapsto fg_U^{-1} \\
    s^{\sigma}_{i}& \mapsto  s^{\sigma}_{i}g_U^{-1} \in \mathcal{O}_{\mathfrak{X}_\sigma}^{\times}(\mathfrak{X}_\sigma \cap U),
\end{align*}
where $s^{\sigma}_{i}g_U^{-1}$ is a regular invertible function on $\mathfrak{X}_\sigma \cap U$, as its reduction to $Z$ is invertible.
Finally, the section $s^{\sigma}_{i}$ is defined globally on $\widehat{\X_{/Z}}$ and on each open $\mathfrak{X}_\sigma \cap U$ gives a local equation of the divisor $D_{i}$, hence it is a equation for $D_{i}$ on $\mathfrak{X}_\sigma$.
\end{proof}

\subsection{Construction for two adjacent maximal cones} \label{section:two max cones}

Let $\sigma$ and $\sigma'$ be two maximal cones of $\Sigma$ intersecting along a face of codimension one, and let $C \subseteq Z$ be the curve associated with the cone $\sigma \cap \sigma'$. Setting $L_{\sigma \sigma'} = L_{\sigma} \cap L_{\sigma'}$, we may write $L_{\sigma} = L_{\sigma \sigma'} \cup \{ i_0 \}$ and $L_{\sigma'} = L_{\sigma \sigma'} \cup \{ i_{\infty} \}$.
The sets $\mathcal{B}=((v_i)_{i \in L_{\sigma \sigma'}},v_{i_0}, (v_j)_{j \in J})$ and $\mathcal{B}'=((v_i)_{i \in L_{\sigma \sigma'}},v_{i_\infty}, (v_j)_{j \in J})$ are bases of $\hat{N} = N \oplus \Z^J$. They induce isomorphisms $\beta, \beta': \Z^r \oplus \Z^J \rightarrow \hat{N}$ such that 
the change of basis 
from $\mathcal{B}$ to $\mathcal{B}'$ is 
\begin{equation*} \label{equ matrix change base}
M_{\mathcal{B}' \mathcal{B}}= \beta' \circ \beta^{-1}=
\begin{blockarray}{cccc}
L_{\sigma \sigma'} & i_0  & J\\
\begin{block}{(ccc)c}
  \Id  & (-C \cdot D_{i})_{i \in L_{\sigma \sigma'} } & 0 & L_{\sigma \sigma'}  \\
  0 & -1 & 0 & i_{\infty}\\
  0 & (-C\cdot D_j)_{j \in J} & \Id & J \\
\end{block}
\end{blockarray}
\quad \textrm{ and }
\begin{pmatrix}
v_i \\ v_{i_\infty} \\ v_j
\end{pmatrix} 
= M_{\mathcal{B}' \mathcal{B}}^{T}
\begin{pmatrix}
v_i \\ v_{i_0} \\ v_j
\end{pmatrix}.
\end{equation*}
Denote by $((\varepsilon_i)_{i \in L_{\sigma \sigma'}}, \varepsilon_{i_0}, (\varepsilon_j)_{j \in J})$ the basis of $\hat{M} \coloneqq \Hom(\hat{N}, \Z)$ dual to $\mathcal{B}$, and $((\varepsilon'_i)_{i \in L_{\sigma \sigma'}}, \varepsilon'_{i_\infty},  (\varepsilon'_j)_{j \in J})$ the basis dual to $\mathcal{B}'$. It follows that 
\begin{equation}\label{equ:dualbasis_two cones}
\begin{pmatrix}
\varepsilon'_i \\ \varepsilon'_{i_\infty} \\ \varepsilon'_j
\end{pmatrix} 
= M_{\mathcal{B}' \mathcal{B}}
\begin{pmatrix}
\varepsilon_i \\ \varepsilon_{i_0} \\ \varepsilon_j
\end{pmatrix}.
\end{equation}
%
\\The isomorphisms $\beta$ and $\beta'$ allow us to view 
$$W^{\sigma} \coloneqq ((W^{\sigma}_i)_{i \in L_{\sigma \sigma'}}, W^{\sigma}_{i_0}, (W^{\sigma}_j)_{j \in J} ) \, \in \, (\Z^r \oplus \Z^J) \otimes \Div_0(\X) \simeq (\Div_0(\X))^{n+1}$$ 
and $W^{\sigma'}$ as elements of $\widehat{N} \otimes \Div_0(\X)$, that we will still denote by $W^{\sigma}$ and $W^{\sigma'}$ .


\begin{lem} \label{lemma line bundles two max cones}
Let $C \subseteq Z$ be the curve associated with the cone $\sigma \cap \sigma'$.
We have
$$\begin{cases}
W^{\sigma'}_{i} = W^{\sigma}_{i}  - (C \cdot D_{i}) W^{\sigma}_{i_0} 
& \textrm{ \quad for } i \in L_{\sigma \sigma'} \\
W^{\sigma'}_{i_{\infty}}= - W^{\sigma}_{i_0}  & \\
W^{\sigma'}_{j} =  W^{\sigma}_{j} - (C \cdot D_j) W^{\sigma}_{i_0} 
&  \textrm{ \quad for } j \in J
\end{cases}$$ 
In other words, the relation $W^{\sigma'} =  (M_{\mathcal{B}' \mathcal{B}  } \otimes \Id) W^{\sigma}$ holds.
\end{lem} 

\begin{proof}
By \cref{equ relation fan} we have $u_{i_\infty}= -u_{i_0} - \sum_{m \in L_{\sigma \sigma'}}  (C \cdot D_{m})u_m$, so
\begin{align*}
\textrm{for $i \in L_{\sigma \sigma'}$, } \quad W^{\sigma'}_{i} 
& = - \frac{\det (\de , (u_{l})_{l \in L_{\sigma \sigma'} \setminus \{i\}}, u_{i_\infty})}{\det (u_i , (u_{l})_{l \in L_{\sigma \sigma'} \setminus \{i\}}, u_{i_\infty})}
\\
& =  - \frac{\det (\de , (u_{l})_{l \in L_{\sigma \sigma'} \setminus \{i\}}, u_{i_0})}{\det (u_i , (u_{l})_{l \in L_{\sigma \sigma'} \setminus \{i\}}, u_{i_0})} 
- \sum_{m \in L_{\sigma \sigma'}} (C \cdot D_m) \frac{\det (\de , (u_{l})_{l \in L_{\sigma \sigma'} \setminus \{i\}}, u_{m})}{\det (u_i , (u_{l})_{l \in L_{\sigma \sigma'} \setminus \{i\}}, u_{i_0})} \\
& = W^\sigma_i - (C \cdot D_i) \frac{\det (\de , (u_{l})_{l \in L_{\sigma \sigma'} \setminus \{i\}}, u_{i})}{\det (u_i , (u_{l})_{l \in L_{\sigma \sigma'} \setminus \{i\}}, u_{i_0})}
=  W^\sigma_i - (C \cdot D_i) W^\sigma_{i_0};\\
\textrm{for $i= i_{\infty}$, } \quad W^{\sigma'}_{i_{\infty}} 
& = \frac{\det (\de, (u_l)_{l \in L_{\sigma \sigma'}})}{\det(u_{i_\infty}, (u_l)_{l \in L_{\sigma \sigma'}})}
= -\frac{\det (\de, (u_l)_{l \in L_{\sigma \sigma'}})}{\det(u_{i_0}, (u_l)_{l \in L_{\sigma \sigma'}})}=- W^{\sigma}_{i_0}.
\end{align*}
For $j \in J$
\begin{align*}
    \sum_{i \in L_{\sigma'}}  \lambda_{j,i} W^{\sigma'}_i 
    & =  - \lambda_{j,i_{\infty}} W^{\sigma}_{i_0} + \sum_{i \in L_{\sigma \sigma'}} \lambda_{j,i} \left(W^{\sigma}_{i} -(C\cdot D_{i}) W^{\sigma}_{i_0}\right) \\
    & =  \Big(- \lambda_{j,i_{\infty}} - \sum_{i \in L_{\sigma \sigma'}} \lambda_{j,i} (C\cdot D_{i}) - \lambda_{j,i_{0}}\Big) W^{\sigma}_{i_0} + \sum_{i \in L_{\sigma}} \lambda_{j,i} W^{\sigma}_{i}  \\
    & = (C \cdot D_j) W^{\sigma}_{i_0} + \sum_{i \in L_{\sigma}} \lambda_{i,j}  W^{\sigma}_{i} \hspace{30pt} \textrm{by \cref{equ fan}} \\
    W^{\sigma'}_j 
    & = - D_j - \sum_{l \in L} \lambda_{j,l} D_l - \sum_{i \in L_{\sigma'}} \lambda_{j,i} W^{\sigma'}_i \\
    & =  - D_j - \sum_{l \in L} \lambda_{j,l} D_l - \sum_{i \in L_{\sigma}} \lambda_{j,i}  W^{\sigma}_{i} -(C \cdot D_j) W^{\sigma}_{i_0} = W^{\sigma}_j -(C \cdot D_j) W^{\sigma}_{i_0}.
\end{align*}
These relations can be summed up as
$ \begin{pmatrix}
W^{\sigma'}_i \\ W^{\sigma'}_{i_\infty} \\ W^{\sigma'}_j
\end{pmatrix}
= M_{\mathcal{B}'\mathcal{B}} \begin{pmatrix}
W^{\sigma}_i \\ W^{\sigma}_{i_0} \\ W^{\sigma}_j
\end{pmatrix}$, i.e. $W^{\sigma'} =  (M_{\mathcal{B}' \mathcal{B}  } \otimes \Id) W^{\sigma}$.
\end{proof}

The inverse $(s^{\sigma}_{i_0})^{-1}$ is a section on $\mathfrak{X}$ of $\Ld^{\sigma'}_{i_\infty}$, so by \cref{lemma line bundles two max cones} the sections
$$\begin{cases}
s^{\sigma'}_{i} \coloneqq s^{\sigma}_{i}\cdot  (s^{\sigma}_{i_0})^{ -(C \cdot D_{i})} & \textrm{ \quad for } i \in L_{\sigma \sigma'} \\
s^{\sigma'}_{i_{\infty}} \coloneqq (s^{\sigma}_{i_0})^{-1} & \\
s^{\sigma'}_{j} \coloneqq s^{\sigma}_{j} \cdot (s^{\sigma}_{i_0} )^{- (C \cdot D_j)}  & \textrm{ \quad for } j \in J 
\end{cases}$$
are sections on $\mathfrak{X}$ of the line bundles $\Ld^{\sigma'}_i$ and $\Ld^{\sigma'}_j$. By \cref{lemma sections} these give equations for $D_i$ and $D_j$ on the open subscheme $\mathfrak{X}_{\sigma'},$ and on $\mathfrak{X}_{\sigma} \cap \mathfrak{X}_{\sigma'}$ we have
\begin{equation} \label{equ:sections_two cones}
\begin{pmatrix}
s^{\sigma'}_i \\s^{\sigma'}_{i_\infty} \\ s^{\sigma'}_j
\end{pmatrix} 
= M_{\mathcal{B}' \mathcal{B}}
\begin{pmatrix}
s^{\sigma}_i \\s^{\sigma}_{i_0} \\ s^{\sigma}_j
\end{pmatrix}
\end{equation}
where the additive notation on the matrix corresponds to the multiplicative notation on the sections. 
Moreover, on $\mathfrak{X}_{\sigma} \cap \mathfrak{X}_{\sigma'}$ we have 
\begin{align} \label{equ:wsigma}
\begin{split}
w_{\sigma'} 
& \coloneqq t \cdot \prod_{j \in J} (s_{j}^{\sigma'})^{-1} \cdot \prod_{i \in L_{\sigma'}} (s_{i}^{\sigma'})^{-1} = t \cdot 
\prod_{j \in J} (s^{\sigma}_{j})^{-1} (s^{\sigma}_{i_0} )^{ C \cdot D_j} 
\cdot \prod_{i \in L_{\sigma \sigma'}} ( s^{\sigma}_{i})^{-1}   (s^{\sigma}_{i_0})^{ C \cdot D_{i}} \cdot  (s^{\sigma}_{i_0}) \\
& = t \cdot 
\prod_{j \in J} (s^{\sigma}_{j})^{-1} \cdot
\prod_{i \in L_{\sigma \sigma'}} ( s^{\sigma}_{i})^{-1} \cdot
(s^{\sigma}_{i_0})^{ \sum_{j \in J} C \cdot D_j + \sum_{i \in L_{\sigma \sigma'}} C \cdot D_{i} + 1}  = w_{\sigma}, 
\end{split}
\end{align}
hence the invertible function $w_{\sigma}$ on $\mathfrak{X}_\sigma$ extends to $\mathfrak{X}_\sigma \cup \mathfrak{X}_{\sigma'}$ by $w_{\sigma'} $.

\subsection{Construction of the morphism}
\label{proof:morphism}
Let $\Gamma$ be the graph with vertices the maximal cones of $\Sigma$ (hence the maximal cones of $\widehat{\Sigma}$) and with an edge between $\sigma$ and $\sigma'$ if and only if $\sigma \cap \sigma'$ is a common face of codimension one.
Note that since $Z$ is proper, if $\mathbb{S} \subset N_{\R}$ is a sphere with center the origin, then $\Sigma \cap \mathbb{S}$ is a triangulation of $\mathbb{S}$. In particular, $\Gamma$ is the 1-skeleton of the dual complex of a triangulation of the sphere, and is thus connected.

Let $\sigma_0 \in \Sigma$ be a maximal cone, and $p_0 \in \Gamma$ the corresponding vertex, that we will use as a reference point. We fix a tuple of sections $s^{\sigma_0}$ of $W^{\sigma_0}$ as in \cref{section:one max cone}.
\\Let $\sigma \in \Sigma$ be a maximal cone, and $p \in \Gamma$ the corresponding vertex. By connectedness of $\Gamma$, there exists a path $\gamma$ from $p_0$ to $p$, hence a sequence of maximal cones $\sigma_0,\ldots, \sigma_q = \sigma$ such that $\sigma_h \cap \sigma_{h+1}$ is a codimension one face of both $\sigma_h$ and $\sigma_{h+1}$, for $h=0,\ldots, q-1$. The construction of \cref{section:two max cones} allows us to construct inductively along $\gamma$ a tuple of sections $s^{\sigma_h}$ of  $W^{\sigma_h}$. 
\begin{lem} \label{lem:indep on loop}
The tuple of sections $s^{\sigma}$ is independent on the choice of path. 
\end{lem}
\begin{proof}
By \cref{equ:sections_two cones}, for any $h=0,\ldots,q-1$, the sections $s^{\sigma_{h+1}}$ are constructed from $s^{\sigma_{h}}$ by multiplication by the matrix for the change of basis from $((v_i)_{i \in L_{\sigma_h}}, (v_j)_{j \in J})$ to $((v_i)_{i \in L_{\sigma_{h+1}}}, (v_j)_{j \in J})$. Thus, by composition, the sections $s^\sigma$ only depends on $s^{\sigma_0}$ and the change of basis from $((v_i)_{i \in L_{\sigma_0}}, (v_j)_{j \in J})$ to $((v_i)_{i \in L_{\sigma}}, (v_j)_{j \in J})$.
\end{proof}
This provides us with a tuple of sections $s^{\sigma} $ of $W^{\sigma}$ for each maximal cone $\sigma \in \Sigma$, and the function
$$w_{\sigma} = t \cdot \prod_{j \in J} (s_{j}^{\sigma})^{-1} \cdot \prod_{i \in L_{\sigma}} (s_{i}^{\sigma})^{-1} \, \in \gO(\mathfrak{X}_{\sigma})^{\times}.$$
By \cref{equ:wsigma} the $w_\sigma$ glue to an invertible function $w$ on $\mathfrak{X}$; $w$ admits a $(n+1)$-th root on $Z$, since it is constant, and by Hensel's lemma we obtain an invertible function $w'$ on $\mathfrak{X}$ such that $(w')^{n+1} =w$.
We use the sections $s^\sigma$ and the function $w'$ to define a morphism $$f_{\sigma} : \mathfrak{X}_{\sigma} \fl \mathfrak{N}_{\sigma}$$ as follows. Denoting by $((\varepsilon_i)_{i \in L_\sigma}, (\varepsilon_j)_{j \in J})$ the dual basis to $((v_i)_{i \in L_{\sigma}}, (v_j)_{j \in J})$, the toric chart $\mathfrak{N}_{\sigma}$ has the following explicit description:
$$ \mathfrak{N}_{\sigma} = \Spf R \{\chi^{\varepsilon_i}, i \in L_{\sigma}\}[[ \chi^{\varepsilon_j}, j \in J ]]/ \{ t - \chi^{\sum_{i \in L_{\sigma}} \varepsilon_i +\sum_{j \in J} \varepsilon_j} \}.$$
Indeed, $\mathfrak{N}_{\sigma}$ is the formal completion along $Z$ of 
$\mathcal{N}_{\hat{\sigma}} \times_{\mathbb{A}^1} R$, where $\mathcal{N}_{\hat{\sigma}}=\Spec k [(\hat{\sigma})^{\vee}\cap \hat{M}] $; 
since $\ord(t)=  \sum_{i \in L_{\sigma}} \varepsilon_i +\sum_{j \in J} \varepsilon_j$ on $\widehat{\sigma}$, the relation $t = \chi^{\sum_{i \in L_{\sigma}} \varepsilon_i +\sum_{j \in J} \varepsilon_j}$ holds.
\\The map $f_{\sigma}$ is now defined at the level of function rings by 
\begin{align*}
\label{equ:morphism one max cone}
f^{\#}_{\sigma}:  \gO(\mathfrak{N}_{\sigma}) 
& \fl \gO(\mathfrak{X}_{\sigma}) \\
\chi^{\varepsilon_i}  
& \mapsto w' s^{\sigma}_{i} \; \textrm{ \quad for } i \in L_{\sigma} \\
\chi^{\varepsilon_j} 
& \mapsto w' s_{j}^{\sigma} \; \textrm{ \quad for } j \in J
\end{align*}
where the sections $s^{\sigma}$ are viewed as functions on $\mathfrak{X}_{\sigma}$ as in \cref{lemma sections}.

\begin{lem}
For any pair of maximal cones $\sigma, \sigma'$ intersecting along a codimension one face, the morphisms $f_{\sigma}$ and $f_{\sigma'}$ coincide on the overlap $\mathfrak{X}_{\sigma} \cap \mathfrak{X}_{\sigma'}$.
\end{lem}
\begin{proof}
The cones $\sigma$ and $\sigma'$ correspond to adjacent vertices in $\Gamma$. Thus, by \cref{lem:indep on loop} we construct $s^\sigma$ from any path joining $\sigma_0$ to $\sigma$, and $s^{\sigma'}$ from $s^\sigma$ by the relation $s^{\sigma'}=M_{\mathcal{B'} \mathcal{B}}\, s^\sigma$
in \cref{equ:sections_two cones}.

The functions $\chi^\varepsilon$ transform into $\chi^{\varepsilon'}$ via the change of dual bases, which is given by $\varepsilon'= M_{\mathcal{B'} \mathcal{B}}\, \varepsilon$ in \cref{equ:dualbasis_two cones}. 
Comparing the two formulas, it follows that $f_{\sigma} = f_{\sigma'}$ on $\mathfrak{X}_{\sigma} \cap \mathfrak{X}_{\sigma'}$.
\end{proof}

\begin{prop}
The morphism of formal $R$-schemes $f : \widehat{\X_{/Z}} \fl \widehat{\mathscr{N}_{/Z}}$ obtained by gluing the morphisms $f_{\sigma}$ is an isomorphism.
\end{prop}
\begin{proof}
We follow the argument in \cite[Proposition 5.4]{NXY}.
\\Since the source and the target have same dimension and are integral, it is enough to check that $f$ is a closed immersion. 
\\If $\mathscr{J}$ is the largest ideal of definition of $\widehat{\mathscr{N}_{/Z}}$, i.e. the defining ideal of $Z \subset \mathscr{N}$, then $f^* \mathscr{J} = \mathscr{I}_{Z}$ is the largest ideal of definition of $\widehat{\X_{/Z}}$. Indeed, since $Z$ is cut out inside $\X$ by the $D_j$ for $j \in J$, the ideal $\mathscr{I}_Z$ is locally generated by the $s_j$ for $j \in J$; the same reasoning shows that $\mathscr{J}$ is locally generated by the $\chi^{\eps_j}$. The equality $f^*\mathscr{J} = \mathscr{I}_Z$ now follows directly from the local definition of $f$.
\\We 
use \cite[4.8.10]{EGA3.1} and the fact that $f$ induces an isomorphism on the reductions to infer that $f$ is a closed immersion, and thus an isomorphism by equality of dimensions. 
\end{proof}

This concludes the proof of \cref{intro:main thm Z}: $\X$ is toric along $Z$.

\subsection{Integral affine structures}

The case in \cref{intro:main thm Z} where $Z$ is an irreducible component $D$ of $\X_k$ is particularly relevant for proving \cref{intro:main quintic}, while the case  where $Z$ is a stratum curve in $\cX_k$ was treated in \cite{NXY}. In these two cases, we give an explicit description of the $\Z$-affine structure on $\Star(\tau_Z)$. 


\subsubsection{Toric irreducible components}

\begin{cor} \label{cor fan structure}
In the setting of \cref{intro:main thm Z}, let $Z=D$ be an irreducible component of $\X_k$. Then there is a natural $\Z$-linear embedding of $\Star(v_D)$ inside the fan $\Sigma_D$ of $D$ which sends the polyhedral decomposition of $\Star(v_D)$ to the cone decomposition of $\Sigma_D$.
\end{cor}

\begin{proof}
By the proof of \cref{intro:main thm Z} and \cref{prop: toric retraction} we have the following diagram:

$$\begin{tikzpicture}
\matrix(m)[matrix of math nodes, row sep=2.5em,
    column sep=2em, text height=1.5ex, text depth=0.25ex]
{\mathfrak{X}_D & \mathfrak{N}_D  \\
    \Star(v_D)  & \mathrm{Int}(\Sigma_1) \\ };
	\path 
	(m-1-1) edge[->] node[auto] {$\simeq$} (m-1-2)
			edge[->] node[auto] {$ {\rho_\X} $} (m-2-1)
         (m-1-2) edge[->] node[auto] {$\val$} (m-2-2)
         (m-2-1) edge[->] node[auto] {$\simeq$} (m-2-2)
         (m-2-2) edge[-] node[auto] {$\varphi$} (m-2-1);
\end{tikzpicture}$$
where the upper arrow is an isomorphism of analytic spaces, and the lower one a homeomorphism. Here $\mathfrak{X}_D$ and $\mathfrak{N}_D$ are the generic fibers (in the sense of Berkovich) of the formal completions $\widehat{\X_{/D}}$ and $\widehat{\mathscr{N}_{/D}}$ respectively, and $\mathrm{Int}(\Sigma_1)$ denotes the interior of the polyhedral complex $\Sigma_1$ obtained by intersecting the fan $\hat{\Sigma} \subset N_\R \times \R$ of the normal bundle of $D$ in $\X$ with $N_{\R} \times \{1\}$. In particular, $\mathrm{Int}(\Sigma_1)$ is embedded in $\Sigma_D \simeq N_\R \simeq \R^n$, the polyhedral decomposition of $\Star(v_D)$ is the same as $\Sigma_D$, and the vertex $v_D$ corresponds to the origin. By \cref{subsec:defn affine structure}, the integral affine structure on $\Star(v_D)$ is the pullback via $\varphi$ of the integral affine structure on $\Sigma_D$, and this concludes the proof.
\end{proof}

\subsubsection{Stratum curve}\label{sec:curve}

Let $X/K$ be a smooth $n$-dimensional projective maximally degenerate Calabi--Yau variety, and 
let $\cX/R$ be a good minimal dlt model of $X$, with reduced special fiber $\X_k = \sum_{i \in I} D_i$. We consider a one-dimensional stratum $C = D_1 \cap \ldots \cap D_n$ of $\X_k$, which is therefore a smooth rational curve, and is such that $(\X, \X_k)$ is an snc pair in a formal neighbourhood of $C$ by \cite[Corollary 4.6]{NXY}. Since $(C, \De_C)$ is log Calabi--Yau, we may write its boundary  as $ \De_C = p_0 + p_{\infty}$, where $p_0= C \cap D_0$ and $p_\infty = C \cap D_\infty$ for two irreducible components $D_0, D_{\infty}$ of $\X_k$ meeting $C$ transversally.  

Following \cite{NXY}, we write $b_i = - (C \cdot D_i)$ for $i=1,\ldots,n$; from $C \cdot \X_k = 0$ we infer $\sum_{i=1}^n b_i =2$. The $\Star(\tau_C)$ consists on the union of two maximal faces corresponding to the zero-dimensional strata $p_0, p_{\infty}$, meeting along $\tau_C$. 

\begin{prop} \label{prop:integral_structure_n}
Let $\rho_{\X}$ be the retraction associated with the model $\X$, and endow $\emph{Sk}(X)$ with the $\Z$-affine structure induced by $\rho_{\X}$ away from the codimension 1 faces of $\Sk(\X)$.
Then $\Star(\tau_C)$ is $\Z$-affine isomorphic to the union of the simplices $\langle v_0,v_1,\ldots,v_n\rangle $ and $\langle v_1,\ldots,v_n, v_\infty\rangle $ in $\R^n$ where

\centering $v_0 = (1,0,\ldots,0)$, $v_1 = (0,1,\ldots,0)$,\ldots, $v_n = 0$ and $v_{\infty} = (-1, b_1,\ldots, b_{n-1})$. 
\end{prop}
\begin{proof}
We write $b = \min_{i\leqslant n} b_i$; we assume $b$ to be negative or zero by the condition $b_1+\ldots+b_n =2$, as the case $n=2$ and $b_1=b_2=1$ is already treated in the proof of \cite[prop. 5.4]{NXY}. 

The blow-up $\X_1$ of the point $p_{\infty}$ in $\X$ yields a new irreducible component $D_{\infty,1}$ (we denote the strict transforms by the same letters for notational simplicity) with multiplicity $N_{\infty,1}=n+1$, the point $p_{\infty,1}= C \cap D_{\infty,1}$ and the intersection numbers $b_{i,1} \coloneqq - (C \cdot D_i)_{\X_1}=b_i +1$. If we repeat the process $s$ times, we obtain the models $\X_s$, the exceptional divisors $D_{\infty,s}$ with multiplicity $N_{\infty,s}=ns+1$, the points $p_{\infty,s}= C \cap D_{\infty,s}$ and the intersection numbers $b_{i,s}\coloneqq -(C \cdot D_i)_{\X_s}=b_i +s$.

For $s=1-b$, we have $\min_{i\leqslant n} \{b_{i,1-b}\} > 0$, and by \cite{NXY} the integral affine structure induced by $\X_{1-b}$ on $\Star(\tau_C)$ is given by $v_0, \ldots, v_n$ and 
\begin{equation} \label{equ:v_infinity}
    v_{\infty,1-b} 
=  \frac{1}{n(1-b)+1}(-1,b_1+1-b,\ldots,b_{n-1}+1-b).
\end{equation}
The sequence of blow-ups $\X_{s+1} \rightarrow \X_s$ induces (weighted) barycentric subdivisions of the faces $\tau_{p_{\infty,s}}$ with vertices such that
\begin{equation} \label{equ:v_infinity_relations}
N_{\infty,s+1}v_{\infty,s+1}= N_{\infty,s}v_{\infty,s} + \sum_{i=1}^{n}v_i.  
\end{equation}
Combining \cref{equ:v_infinity} and \cref{equ:v_infinity_relations}, at each step we obtain that 
$$v_{\infty,s} 
=  \frac{1}{ns+1}(-1,b_1+s,\ldots,b_{n-1}+s),$$ and in particular $v_{\infty} = (-1, b_1,\ldots, b_{n-1})$. The proposition follows from the following lemma.
\end{proof}

\begin{lem}
Let $B = \tau_1 \cup \tau_2$ be the union of two $n$-dimensional simplices along a face of codimension one. Assume we are given a $\Z$-affine structure on $B$, compatible with those on the $\tau_i$'s. 

Suppose there exists a sequence of (weighted) star subdivisions of $\tau_1$ such that $B' \coloneqq \Star(\tau_1 \cap \tau_2)$ (with respect to this subdivision) can be embedded in $\R^n$ compatibly with the $\Z$-affine structure. Then this embedding extends to $B$, and the $\Z$-affine structure on $B$ is uniquely recovered by this embedding.
\end{lem}
\begin{proof}
The assumptions yield two charts for the $\Z$-affine structure on $B$: the $\Z$-affine subsets $B'$ and $\tau_1$. These two charts are glued along $B' \cap \tau_1$ which is a simplex and thus has no non-trivial $\Z$-automorphisms preserving the vertices, hence the affine structure on $B$ is uniquely determined. By induction on the number of star subdivisions, there exists a unique subset $\tilde{B} \subset \R^n$ such that $B' \subset \R^n$ can be obtained as the result of the same star subdivisions of $\tilde{B}$, and uniqueness of the affine structure ensures a $\Z$-affine isomorphism $B \simeq \tilde{B}$.
\end{proof}
\begin{rem}
Consider an irreducible component $D_i$ of $\X_k$ and write $\Delta_{D_i}= \sum_{j \neq i} D_j \cap D_i$. By adjunction, the pair $(D_i, \Delta_{D_i})$ is log Calabi--Yau, i.e. $D_i$ is a smooth projective variety over $k$ and $\Delta_{D_i}$ is a divisor such that $K_{D_i}+ \Delta_{D_i}$ is trivial. By \cite[Theorem 6.14]{EvansMauri} there exists a Lagrangian torus fibration $$\phi: \mathcal{U} \rightarrow B \subseteq \Star(v_{D_i}) \setminus W$$
where $\mathcal{U}$ is a symplectic tubular neighborhood of the $1$-dimensional strata of $\Delta_{D_i}$, $B$ is a retract of $\Star(v_{D_i}) \setminus W$, and $W$ is the union of cells of codimension $\geqslant 2$ in $\Sk(\X)$. The fibration $\phi$ is constructed gluing toric moment maps defined in the neighborhood of each stratum curve of $\Delta_{D_i}$. Evans and Mauri compare the monodromy $T_\phi$ induced by $\phi$ on $B$ to the monodromy $T_{\rho_\X}$ induced by the affinoid torus fibration $${\rho_{\X}}:   \rho_\X^{-1}(\Star(v_{D_i}) \setminus W) \rightarrow \Star(v_{D_i}) \setminus W$$ and conclude that they are dual. This means that given a loop $\gamma \in \pi_1(B) \simeq \pi_1(\Star(v_{D_i}) \setminus W)$, we have $T_{\rho_\X}(\gamma) = (T_\phi(\gamma)^{-1})^{T}$. Thus the affine structure constructed in \cite{NXY} has a symplectic topological analog. The duality is due to the fact that the image of the moment maps is $M_\R$, while the image of the tropicalization map $\val$ is in $N_\R$. 
\end{rem}

\begin{rem}
When $n=2$, an irreducible component $D$ of $\X_k$ has boundary $\De_D \coloneqq  \sum_{i=1}^r C_i$, where $C_i := D_i \cap D$. Since the simple normal crossing curve $\De_D \in \lvert - K_D \rvert$ is an anticanonical curve by adjunction, it follows from general surface theory that $\De_D$ is a cycle of rational curves $(C_i)_{i \le r}$, whose geometry is encoded by the $b_i = -(C_i \cdot D) = - (C_i^2)_D -2$. We label the curves so that for $i \le r$, $C_i \cap C_{i+1} \neq \varnothing$, with convention $C_{r+1} = C_1$.

One can associate to the pair $(D, \De_D)$ a pseudo-fan, which is a singular affine structure on $\R^2$, singular at most at $0$. The singularity at $0$ is a way to measure the defect of $(D, \De_D)$ of being toric: the affine structure extends smoothly at $0$ if and only $(D, \De_D)$ is a toric pair \cite[Proposition 3.9]{Eng}. The construction, as explained in \cite[§1.2]{GHK}, is the following. For each node $p_i = C_i \cap C_{i+1}$, consider a cone $\sigma_i \coloneqq \R_{\ge 0}v_i + \R_{\ge 0}v_{i+1} \subset \R^2$, $(v_i, v_{i+1})$ being a basis of the lattice $\Z^2$. The cones $\sigma_i$ and $\sigma_{i+1}$ are then glued to each other along $\R_{\ge 0} v_{i+1}$, and the affine structure is extended through the edge by pretending that the pair $(D, \De_D)$ is toric. If the pair was toric, the $\sigma_i$'s would be the maximal cones of its fan, and the relation
$$ v_{i+2} + v_i = -(C_{i+1}^2) v_{i+1}$$
would hold by Eq. \ref{equ relation fan}, so that the chart $\psi_i : \sigma_i \cup \sigma_{i+1} \fl \R^2$ that defines the $\Z$-affine structure satisfies $\psi_i(0)=0$, $\psi_i(v_i) = (1,0)$, $\psi_i(v_{i+1})= (0,1)$ and $\psi_i (v_{i+2}) = (-1, -(C_{i+1}^2))$, and is extended by dilatation. The union of the $\sigma_i$'s glued along the successive edges is homeomorphic to $\R^2$, and we obtain this way an $\Z$-affine structure away from the origin, extending to $0$ if and only the pair is toric.

It follows from \cref{prop:integral_structure_n} that the singular $\Z$-affine structure induced by the Berkovich retraction $\rho_\X$ on $\Star(v_D)$ coincides with the one of the pseudo-fan.
\end{rem}

\section{Degeneration of quintic 3-folds} \label{sec:quintic 3-fold}

As discussed in the introduction, given a maximally degenerate family $X=(X_t)_{t \in \D^*}$ of Calabi--Yau varieties, Kontsevich and Soibelman predict that the base $S$ of the conjectural SYZ fibration $\rho_t: X_t \rightarrow S$ matches the essential skeleton of $X$. 
The SYZ fibration induces an integral affine structure (with singularities) on the base $S$, which is relevant to the reconstruction of the mirror family; in the non-archimedean interpretation of mirror symmetry of \cite{KontsevichSoibelman} such structure is induced by Berkovich retractions $X^{\an} \rightarrow \Sk(X)$. 

The results in \cite{NXY} and in \cref{Section toric thm} can be used to construct non-archimedean retractions and 
integral affine structures on $\Sk(X)$, using good minimal dlt models of $X$. As quintic Calabi--Yau hypersurfaces have played a key role in the development of mirror symmetry, it is natural to apply and test the non-archimedean approach to this family. More precisely, let $$X=\{ tF_5(z_1, z_2, z_3, z_4, z_5) + z_1 z_2 z_3 z_4 z_5 =0 \} \subset \CP^4_K,$$ with $F_5$ a generic homogeneous polynomial of degree $5$.
We endow $\Sk(X) \simeq \mathbb{S}^3$ with the simplicial structure induced by the identification with the dual complex of $\X_k$, with $\cX$ being the closure of $X$ in $\CP^4_R$. In this section we prove 

\begin{prop}[\cref{intro:main quintic}]
There exists a continuous retraction $\pi : X^{\an} \rightarrow \Sk(X)$ such that
\begin{itemize}
\itemsep0pt
    \item $\pi$ can be written as a composition $\pi' \circ \rho_{\X'}$, with $\X'$ being an snc model of $X$ and $\pi' : \Sk(\X') \rightarrow \Sk(X)$ a piecewise-linear map (thus $\pi$ pulls back any piecewise-linear function on $\Sk(X)$ to a model function on $X^{\an}$);
    
    \item $\pi$ is an affinoid torus fibration outside a piecewise-linear locus $\Gamma$, that has codimension $2$ and is contained in the $2$-skeleton of $\Sk(X)$;

    
    \item $\pi$ induces an integral affine structure on $\Sk(X) \setminus \Gamma$, which is the same as the one constructed in \cite{Gro05,Li,HJMM}.
\end{itemize}
\end{prop}

\subsection{Setting and plan of the proof} \label{subsec:setting quintic}

The model $\X= \{z_1 z_2 z_3 z_4 z_5 +tF_5(z_1,z_2,z_3,z_4,z_5) =0\} \subset \mathbb{P}^4_R$ has the following properties:

\begin{itemize}
\itemsep0em
    \item[1.] the special fiber $\X_k$ is reduced, consisting of five Weil divisors, i.e. $D_i=\{z_i=t=0\}$. We denote $D_i':= \{z_i=F_5=0\}$;
    \item[2.] the singular locus ${\X}^{\text{sing}}$ of the total space $\X$ is contained in the special fiber, and is the intersection in $\CP_k^4$ of $\{ F_5 =0 \}$ and the union of surfaces $S_{ij} = \{z_i = z_j =0 \}$ for $i \neq j$. In particular, each $D_i$ intersects 
    ${\X}^{\text{sing}}$ 
    along the union of four quintic curves $C_{ij}$ and by genericity of $F_5$, we may assume that $C_{ij}$ does not intersect the torus fixed points of $D_i$;   
    \begin{minipage}{0.4\textwidth}
    $${\X}^{\text{sing}} \cap D_i = \bigcup_{ \substack{j=1 \\ j \neq i}}^5 C_{ij}$$
    $$C_{ij} \subseteq D_i \cap D_j$$
    $$ C_{ij} \cap C_{ij'}=\{5 \text{ points} \} \text{ for }j\neq j'$$
    \end{minipage}
    \begin{minipage} {0.5\textwidth}
    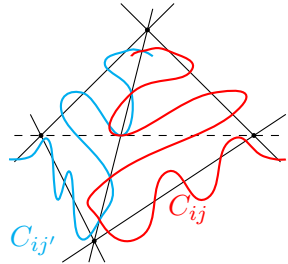
\begin{figure}[H]
    \centering
    \begin{tikzpicture} [scale=0.7]
	\coordinate (01) at (0.1,1.5);
	\coordinate (1) at (-0.1,1.6);
	\coordinate (2) at (-0.3,0.8);
	\coordinate (3) at (-0.5,0);
	\coordinate (4) at (-0.7,-0.8);
	\coordinate (5) at (-0.9,-1.6);
	\coordinate (41) at (-1.6,0.4);
	\coordinate (31) at (-1.2,0.8);
	\coordinate (21) at (-0.8,1.35);
	\coordinate (11) at (-0.4,1.6);
	\coordinate (51) at (-2.4,-0.4);
	\coordinate (12) at (-1.3,-1.9);
	\coordinate (22) at (-1.45,-0.7);
	\coordinate (32) at (-1.7,-1.2);
	\coordinate (42) at (-1.85,-0.1);
	\coordinate (52) at (-2.2,-0.4);
	\coordinate (62) at (-2.6,-0.45);
	\coordinate (03) at (-0.3,1.5);
	\coordinate (43) at (1.6,0.4);
	\coordinate (33) at (1.2,0.8);
	\coordinate (23) at (0.8,1.2);
	\coordinate (13) at (0.4,1.6);
	\coordinate (53) at (2.4,-0.4);
	\coordinate (14) at (-0.1,-1.9);
	\coordinate (24) at (0.35,-0.7);
	\coordinate (34) at (1.1,-1.2);
	\coordinate (44) at (1.55,-0.1);
	\coordinate (54) at (2.2,-0.4);
	\coordinate (64) at (2.6,-0.45);
	\node[right] at  (0.2,-1.4) {\textcolor{red}{$C_{ij}$}};
	\filldraw (-2,0) circle (1pt);
	\filldraw (2,0) circle (1pt);
	\filldraw (0,2) circle (1pt);
	\filldraw (-1,-2) circle (1pt);
	\node[below, left] at (-1.5,-2) {\textcolor{cyan}{$C_{ij'}$}};
	\draw[dashed] (-2.5,0)-- (2.5,0);
	\draw (-0.5,2.5) -- (2.5,-0.5);
	\draw (0.5,2.5) -- (-2.5,-0.5);
	\draw (-0.8,-2.4) -- (-2.2,0.4);
	\draw (-1.6,-2.4) -- (2.6,0.4);
	\draw (0.1,2.4) -- (-1.1,-2.4);
	\draw [cyan, thick, xshift=4cm] plot [smooth, tension=0.8] coordinates {(01)(1) (11) (21) (2) (3) (31) (41) (4) (5) (12) (22) (32) (42) (52) (62)};
	\draw [red, thick, xshift=4cm] plot [smooth, tension=1] coordinates {(03)(1) (13) (23) (2) (3) (33) (43) (4) (5) (14) (24) (34) (44) (54) (64)};
	\end{tikzpicture}
    \caption{Irreducible component $D_i$}
    \label{figure:componentD_i}
    \end{figure}
    \end{minipage}
    \vspace{5pt}
    \item[3.] the pair $(\X, \X_k)$ is dlt, in particular snc away from 
    ${\X}^{\text{sing}}$; we refer to \cref{subsec:local_model_quintic} for a local study of the pair at the singular points;

    \item[4.] the dual complex $\mathcal{D}(\X_k)$ of the special fiber is homeomorphic to the $3$-sphere $\mathbb{S}^3$, and the triangulation of $\mathcal{D}(\X_k)$ is the same as the standard one on the boundary of a $4$-simplex;
    \item[5.] by adjunction the canonical bundle $K_\X$ is trivial.
\end{itemize}
We conclude that $\X$ is a minimal dlt model of the quintic $3$-fold  $X \coloneqq\X_K$, but it is not good in the sense of \cref{prel:models}
since the prime components of the special fiber are not $\Q$-Cartier. In particular, even if the dual complex $\mathcal{D}(\X_k)$ is well-defined, $\X$ does not induce a well-defined retraction of $X^{\text{an}}$ onto $\mathcal{D}(\X_k)=\Sk(X)$.

\vspace{10pt} Our aim is thus to explicitly construct several explicit good minimal dlt models of $X$ starting from $\cX$; we will proceed as follows.

\begin{itemize}
\itemsep0pt
    \item[-] (\cref{subsec:min model Xijkl}) For any order $(i,j,k,l,h)$ on $\{1,\ldots,5\}$, we construct a small resolution of $\X$ by blowing-up in order the four divisors $D_i$, $D_j$, $D_k$ and $D_l$. The resulting resolution is denoted $\X_{ijkl}$, is a good minimal dlt model of $X$ and comes equipped with the Berkovich retraction $$\rho_{\X_{ijkl}}: X^{\an} \rightarrow \Sk(\X_{ijkl}) = \Sk(X) \simeq \mathbb{S}^3.$$ The skeleton $\Sk(\X_{ijkl})$ coincides with $\mathcal{D}(\X_k)$ as simplicial complex; thus, independently on the order, all skeletons $\Sk(\X_{ijkl})$ define the same simplicial structure on $\Sk(X)$.
    
    \item[-] (\cref{subsec:combinatorial retraction}) We construct a model $\cZ$ of $X$ which dominates any model $\X_{ijkl}$, so that $\rho_{\X_{ijkl}}$ factors through $\rho_\cZ$. We then define a combinatorial retraction $\pi'$ which contracts the skeleton $\Sk(\cZ)$ onto $\Sk(X)$. This allows us to consider the composition $$\pi: X^{\an} \xrightarrow{\rho_\cZ} \Sk(\cZ) \xrightarrow{\pi'} \Sk(X)$$ which is at the core of the statement of \cref{intro:main quintic}. The retraction $\pi'$ is constructed so that the composition $\pi$ is locally equal to a $\rho_{\X_{ijkl}}$, the order depending on the region of $\Sk(X)$.
    
    \item[-] (\cref{subsec:local_model_quintic} to \cref{subsec:local_combinatorial}) The constructions and properties of $\X_{ijkl}, \cZ$ and $\pi'$ rely on a local study of the model $\X$: étale locally around each point of ${\X}^{\text{sing}} \cap D_i \cap D_j \cap D_{j'}$, $\X$ is isomorphic to a toric variety and the resolutions are given by refinements of the associated fan.
\end{itemize}

For the description of the discriminant locus $\Gamma$, we will need the following general definition:

\begin{defn} \label{defn:disc in simplex}
Let $\tau$ be a simplex of dimension $m$ and consider the first barycentric subdivision $\tau'$ of $\tau$. For each vertex $v$ of $\tau$, we denote the star of $v$ in $\tau'$ by $\Star(v)'$ and define $\Gamma_{m-1}$ to be the polyhedral complex of dimension $m-1$ given by $$\Gamma_{m-1} \coloneqq \tau \setminus \bigcup_{v \in \tau} \Star(v)' \subset \tau.$$
\end{defn}
For instance, if $m=2$, $\Gamma_1$ is the union of the three line segments joining the barycenter of the triangle to the barycenters of the edges.

\subsection{Local resolution} \label{subsec:local_model_quintic}

We consider a point in ${\X}^{\text{sing}} \cap D_1 \cap D_2 \cap D_3$; the singular points in the other strata curves can be treated analogously. \'Etale locally around such a point, $\X$ is isomorphic to (the base change to $K$ of) the toric variety $ \U\coloneqq V(xyz - wt) \subset \A^5_k$, where $$D_1|_{\U}= \{x=t=0\}, \quad D_2|_{\U}= \{y=t=0\} \,\text{ and } \, D_3|_{\U}= \{z=t=0\}$$ are the components of the special fiber $\{t=0\}$ in $\cU$. We still denote these by $D_1,D_2$ and $D_3$; they form the toric boundary of $\cU$ together with $$D_1'|_{\U}=\{x=w=0\}, \quad D_2'|_{\U}= \{y=w=0\} \,\text{ and } \, D_3'|_{\U}= \{z=w=0\}.$$ 
The pair $(\X, \X_k)_{|\U}=(\U, \sum_{i=1,2,3} D_i|_{\U})$ is log canonical by \cite[Proposition 11.4.24]{CoxLittleSchenck2011}.

We denote the strata surfaces of the special fiber by $D_{ij} \coloneqq D_i \cap D_j$, for $i, j \in \{1,2,3\}$ and $i \neq j$, and the stratum curve by $D_{123}\coloneqq \{x=y=z=t=0\}$. The singular locus ${\U}^{\text{sing}}$ consists of the torus invariant curves $$ C_{12}|_{\U}=\{x=y=w=t=0\}=D_1 \cap D_2 \cap D_1' \cap D_2',$$ $$C_{13}|_{\U}=\{x=z=t=w=0\} \, \text{ and }\, C_{23}|_{\U}=\{y=z=w=t=0\},$$ which intersect each other at the torus invariant point $p \coloneqq \{x=y=z=w=t=0\}$.
%


\tdplotsetmaincoords{60}{35}
\tdplotsetrotatedcoords{0}{0}{130}
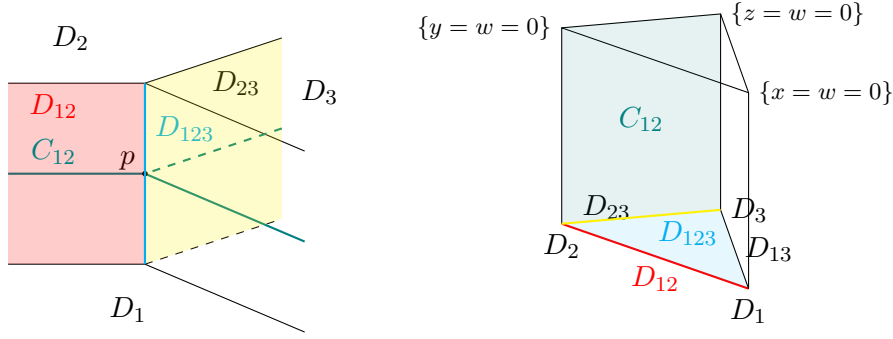
\begin{figure}[H]
\begin{center}
	\begin{tikzpicture}[scale=0.6] 
	\coordinate (1) at (4,-1.5);
	\coordinate (2) at (0,0);
	\coordinate (3) at (2,1);
	\coordinate (21) at (0,2);
	\coordinate (31) at (2,3);
	\coordinate (11) at (4,0.5);
	\coordinate (12) at (2,-0.75);
	\coordinate (13) at (3,-0.25);
	\coordinate (23) at (1,0.5);
	%
	\draw (-3,0)--(0,0);
	\draw (-3,4)--(0,4);
	\draw (0,0)--(3.5,-1.5);
	\draw (0,4)--(3.5,2.5);
	\draw (0,4)--(3,5);
	\draw[dashed] (0,0)--(3,1);
	\draw[dashed, teal,thick] (3,3)--(0,2);
	\draw[teal,thick] (-3,2)--(0,2);
	\draw[teal,thick] (3.5,0.5)--(0,2);
	\filldraw (0,2) circle (1.4pt);
	\node[left] at (0,2.3) {$p$};
	\node[right] at (0,3) {\textcolor{cyan}{$D_{123}$}};
	\node[right] at (-1,-1) {$D_1$};
	\node[left] at (-1,5) {$D_2$};
	\node[right] at  (3.2,3.8) {$D_3$};
	\node at  (2,4) {$D_{23}$};
	\fill[fill=red, fill opacity=0.2] (0,0) -- (0,4)--(-3,4) -- (-3,0);
	\fill[fill=yellow, fill opacity=0.2] (0,0) -- (0,4)--(3,5) -- (3,1);
	\node[above] at (-2,2) {\textcolor{teal}{$C_{12}$}};
	\node[above] at (-2,3) {\textcolor{red}{$D_{12}$}};
	\draw[cyan,thick] (0,0)--(0,4);
	\end{tikzpicture} \hspace{15pt}
\begin{tikzpicture}[tdplot_main_coords,scale=1.5]
\node[below] at (-1,0,0) {$D_2$};
\node[below] at (1,0,0) {$D_1$};
\node[right] at (0,1,0) {$D_3$};
\node[left] at (-1,0,2) {\footnotesize $\{y=w=0\}$};
\node[right] at (1,0,2) {\footnotesize $\{x=w=0\}$};
\node[right] at (0,1,2) {\footnotesize $\{z=w=0\}$};
\node[below] at (0,0,0) {\textcolor{red}{$D_{12}$}};
\node[left] at (-0.5,0.5,0.1) {$D_{23}$};
\node[right] at (0.5,0.5,0) {$D_{13}$};
\node at (-0.5,0.5,1) {\textcolor{teal}{$C_{12}$}};
\node at (0,0.5,0) {\textcolor{cyan}{$D_{123}$}};
\draw (-1,0,0)--(-1,0,2)--(1,0,2)--(1,0,0);
\draw (-1,0,2)--(0,1,2)--(0,1,0);
\draw (0,1,2)--(1,0,2);
\draw[red,thick] (-1,0,0)--(1,0,0);
\fill[fill=cyan, fill opacity=0.1](-1,0,0)--(1,0,0)--(0,1,0);
\fill[fill=teal, fill opacity=0.1] (-1,0,0)--(0,1,0)--(0,1,2)--(-1,0,2);
\draw (1,0,0)--(0,1,0);
\draw[yellow,thick] (-1,0,0)--(0,1,0);
\end{tikzpicture}
\end{center}
\caption{Special fiber of $\U$ and a \emph{slice} of the fan of the toric variety $\U$}
\end{figure}

The toric blow-up $G_1: \U_1 \coloneqq \Bl_{D_1} \cU \rightarrow \U$ along $D_1$ resolves the singularities along $C_{12}$ and $C_{13}$ except at the point $p$. The exceptional locus of $G_1$ consists of two surfaces $S_{12}$ and $S_{13}$ intersecting each other along a curve: these surfaces are mapped by $G_1$ to the respective singular curves, are contained in the strict transform of $D_1$, and correspond to two new edges in the slice of the fan of $\U_1$. The intersection of $S_{12}$ and $S_{13}$ corresponds to a new $2$-dimensional face in the slice of the fan. With a slight abuse of notation we keep the same notation for the strict transforms in $\U_1$.

A resolution of $\U$ is given by the composition of $G_1$ and the toric blow-up $G_{12}: \U_{12} \coloneqq \Bl_{D_2} \U_1 \rightarrow \U_1$ along $D_2$; the latter indeed resolves the singularities along $C_{23}$. The exceptional locus of $G_{12}$ is a surface $S_{23}$, which is mapped by $G_{12}$ to $C_{23}$ and is contained in the strict transform of $D_2$. The morphism $G_{12}$ induces a new $2$-dimensional face in the slice of the fan of $\U_{12}$, and a new edge corresponding to the surface $S_{23}$. In particular, after the blow-up $G_{12}$, the strict transforms of the surface $S_{12}$ and of the divisor $D_3$ have empty intersection.

\begin{figure}[H]
    \centering
\tdplotsetmaincoords{60}{35}
\tdplotsetrotatedcoords{0}{0}{130}

\begin{tikzpicture}[tdplot_main_coords,scale=1.8]
\node[below] at (-1,0,0) {$D_2$};
\node[below] at (1,0,0) {$D_1$};
\node[right] at (0,1,0) {$D_3$};
\node[below] at (0,0,0) {$D_{12}$};
\node[left] at (-0.5,0.5,0.1) {$D_{23}$};
\node[right] at (0.5,0.5,0) {$D_{13}$};
\draw (-1,0,0)--(-1,0,2)--(1,0,2)--(1,0,0);
\draw (-1,0,2)--(0,1,2)--(0,1,0);
\draw (0,1,2)--(1,0,2);
\draw (0,0,0.2)--(0,0,0)--(0.15,0.15,0);
\draw (0.5,0.5,0.2)--(1/2,1/2,0)--(0.35,0.35,0);
\draw (-0.5,0.5,0.2)--(-1/2,1/2,0)--(-0.3,0.4,0);
\fill[fill=red, fill opacity=0.2](0,0,0.2)--(0,0,0)--(0.15,0.15,0);
\fill[fill=red, fill opacity=0.2] (0.5,0.5,0.2)--(1/2,1/2,0)--(0.35,0.35,0);
\fill[fill=red, fill opacity=0.2] (-0.5,0.5,0.2)--(-1/2,1/2,0)--(-0.3,0.4,0);
\draw (-1,0,0)--(1,0,0)--(0,1,0)--(-1,0,0);
\end{tikzpicture}
\begin{tikzpicture}[tdplot_main_coords,scale=1.8]
\node[below] at (-1,0,0) {$D_2$};
\node[below] at (1,0,0) {$D_1$};
\node[right] at (0,1,0) {$D_3$};
\node[below] at (0,0,0) {$D_{12}$};
\node[left] at (-0.5,0.5,0.1) {$D_{23}$};
\node[right] at (0.5,0.5,0) {$D_{13}$};
\draw (-1,0,0)--(-1,0,2)--(1,0,2)--(1,0,0);
\draw (-1,0,2)--(0,1,2)--(0,1,0);
\draw (0,1,2)--(1,0,2);
\draw (0,0,0.2)--(0,0,0)--(0.15,0.15,0);
\draw (0.5,0.5,0.2)--(1/2,1/2,0)--(0.35,0.35,0);
\draw (-0.5,0.5,0.2)--(-1/2,1/2,0)--(-0.3,0.4,0);
\fill[fill=red, fill opacity=0.2](0,0,0.2)--(0,0,0)--(0.15,0.15,0);
\fill[fill=red, fill opacity=0.2] (0.5,0.5,0.2)--(1/2,1/2,0)--(0.35,0.35,0);
\fill[fill=red, fill opacity=0.2] (-0.5,0.5,0.2)--(-1/2,1/2,0)--(-0.3,0.4,0);
\draw (1,0,0)--(-1,0,2);
\draw (1,0,0)--(0,1,2);
\fill[fill=orange, fill opacity=0.1](1,0,0)--(-1,0,2)--(0,1,2);
\draw (-1,0,0)--(1,0,0)--(0,1,0)--(-1,0,0);
\end{tikzpicture}
\begin{tikzpicture}[tdplot_main_coords,scale=1.8]
\node[below] at (-1,0,0) {$D_2$};
\node[below] at (1,0,0) {$D_1$};
\node[right] at (0,1,0) {$D_3$};
\node[below] at (0,0,0) {$D_{12}$};
\node[left] at (-0.5,0.5,0.1) {$D_{23}$};
\node[right] at (0.5,0.5,0) {$D_{13}$};
\draw (-1,0,0)--(-1,0,2)--(1,0,2)--(1,0,0);
\draw (-1,0,2)--(0,1,2)--(0,1,0);
\draw (0,1,2)--(1,0,2);
\draw (0,0,0)--(0.075,0.075,0.1);
\draw (0,0,0.2)--(0,0,0)--(0.15,0.15,0);
\draw (0.5,0.5,0.2)--(1/2,1/2,0)--(0.35,0.35,0);
\draw (-0.5,0.5,0.2)--(-1/2,1/2,0)--(-0.3,0.4,0);
\fill[fill=red, fill opacity=0.2](0,0,0.2)--(0,0,0)--(0.15,0.15,0);
\fill[fill=red, fill opacity=0.2] (0.5,0.5,0.2)--(1/2,1/2,0)--(0.35,0.35,0);
\fill[fill=red, fill opacity=0.2] (-0.5,0.5,0.2)--(-1/2,1/2,0)--(-0.3,0.4,0);
\draw (1,0,0)--(-1,0,2);
\draw (1,0,0)--(0,1,2);
\fill[fill=orange, fill opacity=0.1](1,0,0)--(-1,0,2)--(0,1,2);
\draw (-0.5,0.5,1)--(-1,0,0);
\draw[dashed] (0,1,2)--(-0.5,0.5,1);
\fill[fill=teal, fill opacity=0.1](1,0,0)--(-1,0,0)--(0,1,2);
\draw (-1,0,0)--(1,0,0)--(0,1,0)--(-1,0,0);
\end{tikzpicture}
    \caption{Slices of the fan of $\U$, $\U_1$ and $\U_{12}$}
    \label{fig:resolution_fan}
\end{figure}

The small resolution $\U_{12}$ of $\U$ induces an isomorphism on the strict transform of $D_{13}$ and of $D_{23}$, while its restriction to the strict transform of $D_{12}$ is the blow-up of $D_{12}$ along a general point. These facts can be checked computing the charts of the blow-ups $G_1$ and $G_{12}$. Alternatively, they can be verified looking at the fans of the strata surfaces $D_{ij}$ in the slice of the fans of $\U$, $\U_1$ and $\U_{12}$; indeed, the fan of $D_{ij}$ is induced by the intersection of the slice with a normal plane to the edge corresponding to $D_{ij}$.

\subsection{Local dominating model} \label{subsec:local dominating model}
We now introduce the local model for the dominating model $\cZ \fl \X$ we will construct.
\\We consider the blow-up of the exceptional surfaces $S_{12},S_{13}$ and $S_{23}$ one after the other
\begin{center}
    $\arraycolsep=1.6pt\def\arraystretch{1}
    \begin{array}{ccccccccc}
         \cV_{123} 
         & \xrightarrow[H_{23}]{\text{blow-up of }S_{23}}
         & \cV_{13}
         & \xrightarrow[H_{13}]{\text{blow-up of }S_{13}}
         & \cV_{12}
         & \xrightarrow[H_{12}]{\text{blow-up of }S_{12}}
         & \U_{12} 
         & \xrightarrow[G_{12} \circ G_1]
         & \cU \\
         \cup & & 
         \cup & & 
         \cup & & 
         & & \downarrow \\
         E_{23} & &
         E_{13} & &
         E_{12} & &
         & & \mathbb{A}^1_t.
    \end{array}
    $
\end{center}
As these surfaces are toric strata of $\U_{12}$, the blow-ups are toric as well and the corresponding fans are refinements of the fan of $\U_{12}$. We note that

\begin{itemize}
\itemsep0pt
    \item[-] the dual complexes of the special fibers of $\cU_{12}, \cV_{12}, \cV_{13}$ and $\cV_{123}$ are obtained from the slices of the corresponding fans by removing the vertices corresponding to $D_1'$, $D_2'$ and $D_3'$, as well as each face containing one of these.
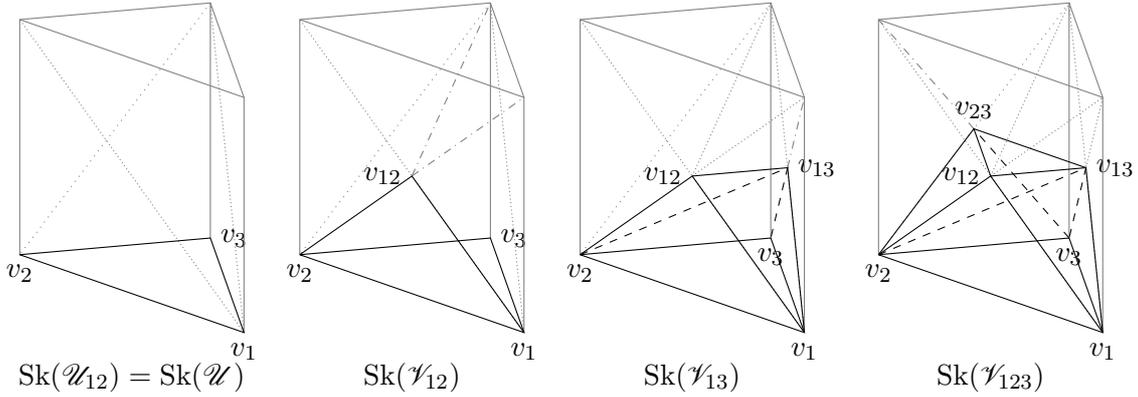
\tdplotsetmaincoords{60}{35}
\begin{figure}[H]
    \centering
\begin{tikzpicture}[tdplot_main_coords,scale=1.8]
\node[below] at (-1,0,0) {$v_2$};
\node[below] at (1,0,0) {$v_1$};
\node[right] at (0,1,0) {$v_3$};
\draw[gray] (-1,0,0)--(-1,0,2)--(1,0,2)--(1,0,0);
\draw[gray] (-1,0,2)--(0,1,2)--(0,1,0);
\draw[gray] (0,1,2)--(1,0,2);
\draw[gray,densely dotted] (1,0,0)--(-1,0,2);
\draw[gray,densely dotted] (1,0,0)--(0,1,2);
\draw[gray,dotted] (0,1,2)--(-1,0,0);
\draw (-1,0,0)--(1,0,0)--(0,1,0)--(-1,0,0);
\node[below] at (0,0,-0.5) {$\Sk(\U_{12})=\Sk(\U)$};
\end{tikzpicture}
\begin{tikzpicture}[tdplot_main_coords,scale=1.8]
\node[below] at (-1,0,0) {$v_2$};
\node[below] at (1,0,0) {$v_1$};
\node[right] at (0,1,0) {$v_3$};
\node[left] at (0,0,1) {$v_{12}$};
\draw[gray] (-1,0,0)--(-1,0,2)--(1,0,2)--(1,0,0);
\draw[gray] (-1,0,2)--(0,1,2)--(0,1,0);
\draw[gray] (0,1,2)--(1,0,2);
\draw[gray,densely dotted] (1,0,0)--(-1,0,2);
\draw[gray,densely dotted] (1,0,0)--(0,1,2);
\draw[gray,dotted] (0,1,2)--(-1,0,0);
\draw (-1,0,0)--(1,0,0)--(0,1,0)--(-1,0,0);
\draw[gray,dashdotted] (0,0,1)--(0,1,2);
\draw[gray,dashdotted] (0,0,1)--(1,0,2);
\draw (-1,0,0)--(0,0,1)--(1,0,0);
\node[below] at (0,0,-0.5) {$\Sk(\cV_{12})$};
\end{tikzpicture}
\begin{tikzpicture}[tdplot_main_coords,scale=1.8]
\node[below] at (-1,0,0) {$v_2$};
\node[below] at (1,0,0) {$v_1$};
\node[below] at (0,1,0) {$v_3$};
\node[left] at (0,0,1) {$v_{12}$};
\node[right] at (1/2,1/2,1) {$v_{13}$};
\draw[gray] (-1,0,0)--(-1,0,2)--(1,0,2)--(1,0,0);
\draw[gray] (-1,0,2)--(0,1,2)--(0,1,0);
\draw[gray] (0,1,2)--(1,0,2);
\draw[gray,densely dotted] (1,0,0)--(-1,0,2);
\draw[gray,densely dotted] (1,0,0)--(0,1,2);
\draw[gray,dotted] (0,1,2)--(-1,0,0);
\draw (-1,0,0)--(1,0,0)--(0,1,0)--(-1,0,0);
\draw[gray,densely dotted] (0,0,1)--(0,1,2);
\draw[gray,densely dotted] (0,0,1)--(1,0,2);
\draw[gray,dashdotted] (1/2,1/2,1)--(1,0,2);
\draw (0,0,1)--(1/2,1/2,1)--(1,0,0);
\draw[dashed] (-1,0,0)--(1/2,1/2,1)--(0,1,0);
\draw (-1,0,0)--(0,0,1)--(1,0,0);
\node[below] at (0,0,-0.5) {$\Sk(\cV_{13})$};
\end{tikzpicture}
\begin{tikzpicture}[tdplot_main_coords,scale=1.8]
\node[below] at (-1,0,0) {$v_2$};
\node[below] at (1,0,0) {$v_1$};
\node[below] at (0,1,0) {$v_3$};
\node[left] at (0,0,1) {$v_{12}$};
\node[right] at (1/2,1/2,1) {$v_{13}$};
\node[above] at (-1/2,1/2,1) {$v_{23}$};
\draw[gray] (-1,0,0)--(-1,0,2)--(1,0,2)--(1,0,0);
\draw[gray] (-1,0,2)--(0,1,2)--(0,1,0);
\draw[gray] (0,1,2)--(1,0,2);
\draw[gray,densely dotted] (1,0,0)--(-1,0,2);
\draw[gray,densely dotted] (1,0,0)--(0,1,2);
\draw[gray,dotted] (0,1,2)--(-1,0,0);
\draw (-1,0,0)--(1,0,0)--(0,1,0)--(-1,0,0);
\draw[gray,densely dotted] (0,0,1)--(0,1,2);
\draw[gray,densely dotted] (0,0,1)--(1,0,2);
\draw[gray,densely dotted] (1/2,1/2,1)--(1,0,2);
\draw[gray,dashdotted] (-1/2,1/2,1)--(-1,0,2);
\draw (0,0,1)--(1/2,1/2,1)--(1,0,0);
\draw[dashed] (-1,0,0)--(1/2,1/2,1)--(0,1,0);
\draw (-1,0,0)--(0,0,1)--(1,0,0);
\draw (-1,0,0)--(-1/2,1/2,1)--(0,0,1);
\draw (-1/2,1/2,1)--(1/2,1/2,1);
\draw[dashed] (0,1,0)--(-1/2,1/2,1);
\node[below] at (0,0,-0.5) {$\Sk(\cV_{123})$};
\end{tikzpicture}
\end{figure}
Then $\Sk(\cV_{123})$ consists of four $3$-cells: $\langle v_{13},v_1,v_2,v_3\rangle $, $\langle v_{13},v_1,v_2,v_{12}\rangle $, $\langle v_{23},v_{13},v_2,v_3\rangle $ and $\langle v_{23},v_{13},v_2,v_{12}\rangle $; it has only one edge in the interior, which is $\langle v_2,v_{13}\rangle $.

\item[-] The remaining $3$-dimensional simplices of the slice of the fan of $\cV_{123}$ are \vspace{-15pt} 

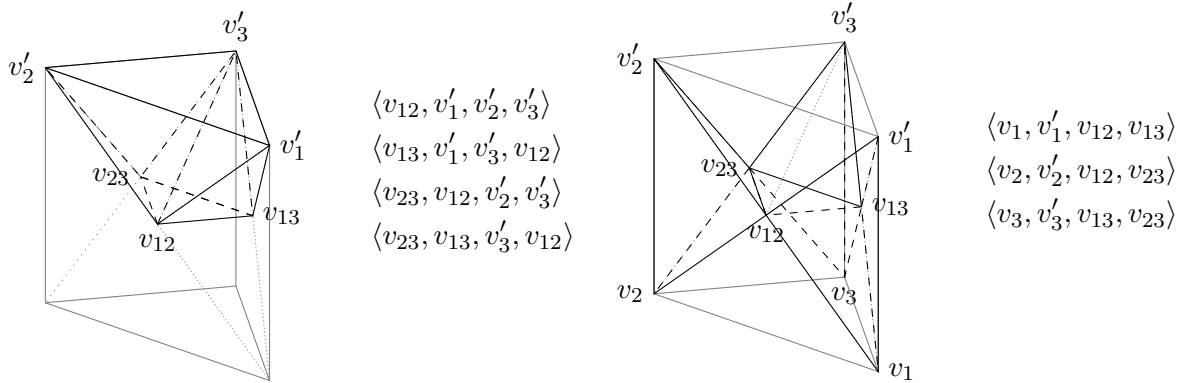
\begin{minipage}{0.25\textwidth}
\tdplotsetmaincoords{60}{35}
\begin{figure}[H]
    \centering
\begin{tikzpicture}[tdplot_main_coords,scale=1.8]
\node[left] at (-1,0,2) {$v_2'$};
\node[right] at (1,0,2) {$v_1'$};
\node[above] at (0,1,2) {$v_3'$};
\node[below] at (0,0,1) {$v_{12}$};
\node[right] at (1/2,1/2,1) {$v_{13}$};
\node[left] at (-1/2,1/2,1) {$v_{23}$};
\draw[gray] (-1,0,0)--(-1,0,2)--(1,0,2)--(1,0,0);
\draw[gray] (-1,0,0)--(1,0,0)--(0,1,0)--(-1,0,0);
\draw[gray] (-1,0,2)--(0,1,2)--(0,1,0);
\draw[gray] (0,1,2)--(1,0,2);
\draw[gray,densely dotted] (1,0,0)--(-1,0,2);
\draw[gray,densely dotted] (1,0,0)--(0,1,2);
\draw[gray,dotted] (0,1,2)--(-1,0,0);
\draw[gray,densely dotted] (0,0,1)--(0,1,2);
\draw[gray,densely dotted] (0,0,1)--(1,0,2);
\draw[gray,densely dotted] (1/2,1/2,1)--(1,0,2);
\draw[gray,dashdotted] (-1/2,1/2,1)--(-1,0,2);
\draw (-1,0,2)--(1,0,2)--(0,1,2)--(-1,0,2);
\draw (0,0,1)--(1/2,1/2,1)--(1,0,2)--(0,0,1)--(-1,0,2);
\draw[dashed] (-1/2,1/2,1)--(1/2,1/2,1)--(0,1,2)--(-1/2,1/2,1)--(-1,0,2);
\draw[dashed] (0,0,1)--(0,1,2);
\draw[dashed] (-1/2,1/2,1)--(1/2,1/2,1);
\draw[dashed] (0,0,1)--(-1/2,1/2,1);
\end{tikzpicture}    
\end{figure}
\end{minipage} \hfill
\begin{minipage} {0.2\textwidth}
\begin{align*}
&\langle v_{12},v_1',v_2',v_3'\rangle \\
&\langle v_{13},v_1',v_3',v_{12}\rangle \\
&\langle v_{23},v_{12},v_2',v_3'\rangle \\
&\langle v_{23},v_{13},v_3',v_{12}\rangle 
\end{align*}
\end{minipage} \hfill
\begin{minipage}{0.25\textwidth}
\tdplotsetmaincoords{60}{35}
\begin{figure}[H]
    \centering
\begin{tikzpicture}[tdplot_main_coords,scale=1.8]
\node[left] at (-1,0,0) {$v_2$};
\node[right] at (1,0,0) {$v_1$};
\node[below] at (0,1,0) {$v_3$};
\node[left] at (-1,0,2) {$v_2'$};
\node[right] at (1,0,2) {$v_1'$};
\node[above] at (0,1,2) {$v_3'$};
\node[below] at (0,0,1) {$v_{12}$};
\node[right] at (1/2,1/2,1) {$v_{13}$};
\node[left] at (-1/2,1/2,1) {$v_{23}$};
\draw[gray] (-1,0,0)--(-1,0,2)--(1,0,2)--(1,0,0);
\draw[gray] (-1,0,0)--(1,0,0)--(0,1,0)--(-1,0,0);
\draw[gray] (-1,0,2)--(0,1,2)--(0,1,0);
\draw[gray] (0,1,2)--(1,0,2);
\draw[gray,densely dotted] (1,0,0)--(-1,0,2);
\draw[gray,densely dotted] (1,0,0)--(0,1,2);
\draw[gray,dotted] (0,1,2)--(-1,0,0);
\draw[gray,densely dotted] (0,0,1)--(0,1,2);
\draw[gray,densely dotted] (0,0,1)--(1,0,2);
\draw[gray,densely dotted] (1/2,1/2,1)--(1,0,2);
\draw[gray,dashdotted] (-1/2,1/2,1)--(-1,0,2);
\draw (1,0,2)--(0,0,1)--(1,0,0)--(1,0,2);
\draw[dashed] (-1/2,1/2,1)--(-1,0,0);
\draw (-1,0,2)--(-1,0,0)--(0,0,1)--(-1/2,1/2,1)--(-1,0,2)--(0,0,1);
\draw[dashed] (1/2,1/2,1)--(0,0,1);
\draw[dashed] (1,0,2)--(1/2,1/2,1)--(1,0,0);
\draw (-1/2,1/2,1)--(0,1,2)--(1/2,1/2,1);
\draw[dashed] (-1/2,1/2,1)--(0,1,0)--(1/2,1/2,1);
\draw (-1/2,1/2,1)--(1/2,1/2,1);
\draw[dashed] (0,0,1)--(-1/2,1/2,1);
\draw[dashed] (0,1,0)--(0,1,2);
\end{tikzpicture}    
\end{figure}
\end{minipage} \hfill
\begin{minipage} {0.2\textwidth}
\begin{align*}
&\langle v_{1},v_1',v_{12},v_{13}\rangle \\
&\langle v_{2},v_2',v_{12},v_{23}\rangle \\
&\langle v_{3},v_{3}',v_{13},v_{23}\rangle 
\end{align*}
\end{minipage} \hfill

\item[-] The Berkovich retractions associated with the models $\U_{12}$, $\cV_{12}$ and $\cV_{13}$ map $v_{12}$ and $v_{13}$ to $v_1$, and $v_{23}$ to $v_2$. 

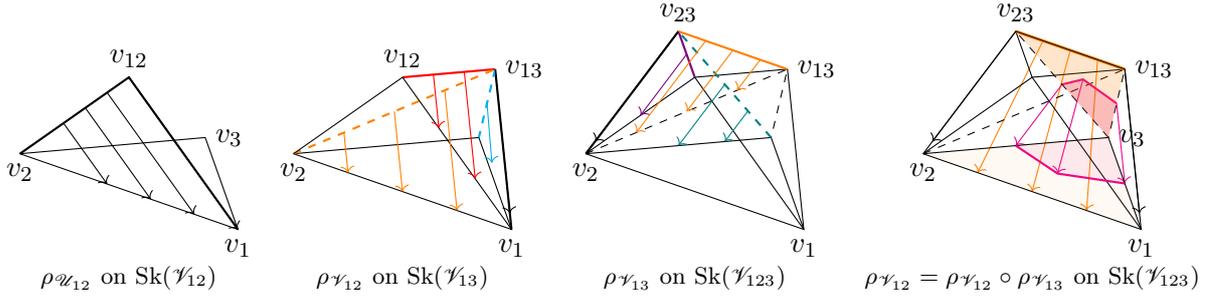
\tdplotsetmaincoords{60}{35}
\begin{figure}[H]
    \centering
\hspace{15pt}
\begin{tikzpicture}[tdplot_main_coords,scale=1.75]
\node at (0,0,-3/4) {\footnotesize $\rho_{\U_{12}}$ on $\Sk(\cV_{12})$};
\node[below] at (-1,0,0) {$v_2$};
\node[below] at (1,0,0) {$v_1$};
\node[right] at (0,1,0) {$v_3$};
\node[above] at (0,0,1) {$v_{12}$};
\draw (-1,0,0)--(1,0,0)--(0,1,0)--(-1,0,0);
\draw[thick] (-1,0,0)--(0,0,1)--(1,0,0);
\draw[->] (0.1,0,0.9)--(1,0,0);
\draw[->] (-0.2,0,0.8)--(0.6,0,0);
\draw[->] (-0.4,0,0.6)--(0.2,0,0);
\draw[->] (-0.6,0,0.4)--(-0.2,0,0);
\end{tikzpicture}
%
%
\begin{tikzpicture}[tdplot_main_coords,scale=1.75]
\node at (0,0,-3/4) {\footnotesize $\rho_{\cV_{12}}$ on $\Sk(\cV_{13})$};
\node[below] at (-1,0,0) {$v_2$};
\node[below] at (1,0,0) {$v_1$};
\node[above] at (0,0,1) {$v_{12}$};
\node[right] at (1/2,1/2,1) {$v_{13}$};
\draw (-1,0,0)--(1,0,0)--(0,1,0)--(-1,0,0);
\draw[thick] (1/2,1/2,1)--(1,0,0);
\draw[thick,red] (0,0,1)--(1/2,1/2,1);
\draw[thick,dashed,orange] (-1,0,0)--(1/2,1/2,1);
\draw[thick,dashed,cyan](1/2,1/2,1)--(0,1,0);
\draw (-1,0,0)--(0,0,1)--(1,0,0);
\draw[->,cyan] (1/4,3/4,0.5)--(0.4,0.6,0.1);
\draw[->,red] (1/6,1/6,1)--(1/3,0,2/3);
\draw[->,red] (1/3,1/3,1)--(2/3,0,1/3);
\draw[->] (1/2,1/2,1)--(0.95,0.05,0.1);
\draw[->,orange] (-5/8,1/8,0.25)--(-0.5,0,0);
\draw[->,orange] (1/8,3/8,3/4)--(0.5,0,0);
\draw[->,orange] (-0.25,0.25,0.5)--(0,0,0);
\end{tikzpicture}
%
%
\begin{tikzpicture}[tdplot_main_coords,scale=1.75]
\node at (0,0,-3/4) {\footnotesize $\rho_{\cV_{13}}$ on $\Sk(\cV_{123})$};
\node[below] at (-1,0,0) {$v_2$};
\node[below] at (1,0,0) {$v_1$};
\node[right] at (1/2,1/2,1) {$v_{13}$};
\node[above] at (-1/2,1/2,1) {$v_{23}$};
\draw (-1,0,0)--(1,0,0)--(0,1,0)--(-1,0,0);
\draw (0,0,1)--(1/2,1/2,1)--(1,0,0);
\draw[dashed] (-1,0,0)--(1/2,1/2,1)--(0,1,0);
\draw (-1,0,0)--(0,0,1)--(1,0,0);
\draw[thick] (-1,0,0)--(-1/2,1/2,1);
\draw[thick,violet](-1/2,1/2,1)--(0,0,1);
\draw[thick,orange] (-1/2,1/2,1)--(1/2,1/2,1);
\draw[thick,dashed,teal] (0,1,0)--(-1/2,1/2,1);
\draw[->,orange] (0,1/2,1)--(-1/4,1/4,1/2);
\draw[->,orange] (-1/4,1/2,1)--(-5/8,1/8,0.25);
\draw[->,orange] (1/4,1/2,1)--(1/8,3/8,3/4);
\draw[->] (-1/2,1/2,1)--(-0.95,0.05,0.1);
\draw[->,violet] (-1/4,1/4,1)--(-1/2,0,1/2);
%
\draw[->,teal] (-1/4,3/4,1/2)--(-1/2,1/2,0);
\draw[->,teal] (-1/8,7/8,1/4)--(-1/4,3/4,0);
\end{tikzpicture}
%
%
  \begin{tikzpicture}[tdplot_main_coords,scale=1.75]
\tikzset{p2/.style={preaction={
			draw,orange,
			double=orange,
			double distance=.8\pgflinewidth,
}}}
\node at (0,0,-3/4) {\footnotesize $\rho_{\cV_{12}} =\rho_{\cV_{12}} \circ \rho_{\cV_{13}}$ on $\Sk(\cV_{123})$};
\node[below] at (-1,0,0) {$v_2$};
\node[below] at (1,0,0) {$v_1$};
\node[right] at (0,1,0) {$v_3$};
\node[right] at (1/2,1/2,1) {$v_{13}$};
\node[above] at (-1/2,1/2,1) {$v_{23}$};
\draw[->,orange] (0,1/2,1)--(0,0,0);
\draw[->,orange] (-1/4,1/2,1)--(-1/2,0,0);
\draw[->,orange] (1/4,1/2,1)--(1/2,0,0);
\draw[->,magenta] (-1/4,3/4,1/2)--(-1/2,1/2,0);
\draw[->,magenta] (1/4,3/4,0.5)--(0.5,0.5,0);
\draw[->,magenta](0,2/3,2/3)--(0,1/3,0);
\draw (-1,0,0)--(1,0,0)--(0,1,0)--(-1,0,0);
\draw (0,0,1)--(1/2,1/2,1)--(1,0,0);
\draw[dashed] (-1,0,0)--(1/2,1/2,1)--(0,1,0);
\draw (-1,0,0)--(0,0,1)--(1,0,0);
\draw (-1,0,0)--(-1/2,1/2,1)--(0,0,1);
\draw[p2] (-1/2,1/2,1)--(1/2,1/2,1);
\draw[magenta,thick] (1/4,3/4,1/2) -- (0,2/3,2/3)--(-1/4,3/4,1/2);
\draw[magenta,thick]  (-1/2,1/2,0)--(0,1/3,0)--(1/2,1/2,0);
\draw[dashed] (0,1,0)--(-1/2,1/2,1);
\draw[->] (-1/2,1/2,1)--(-0.95,0.05,0.1);
\draw[->] (1/2,1/2,1)--(0.95,0.05,0.1);
\fill[fill=magenta, fill opacity=0.2] (1/4,3/4,1/2) -- (0,1,0)--(-1/4,3/4,1/2)--(0,2/3,2/3);
\fill[fill=magenta, fill opacity=0.07] (0,1,0) -- (1/2,1/2,0)--(0,1/3,0)--(-1/2,1/2,0);
\fill[fill=orange, fill opacity=0.2] (0,1,0) -- (1/2,1/2,1)--(-1/2,1/2,1);
\fill[fill=orange, fill opacity=0.05] (0,1,0) -- (1,0,0)--(-1,0,0);
\end{tikzpicture}
\end{figure}
\end{itemize}

We study more in details the retraction $\rho_{\cV_{12}}$ near the vertex $v_3$, as this will be relevant later in the construction of the local combinatorial retraction (see \cref{equ:combinatorial retraction vertx}). We observe that $\rho_{\cV_{12}}$ collapses the convex hull $P(v_1,v_2,v_3,v_{13},v_{23})$ of $v_1,v_2,v_3,v_{13}$ and $v_{23}$ onto the face $\langle v_1,v_2,v_3\rangle $. If we identify the skeleton $\Sk(\cV_{123})$ with the polyhedron in $\R^3_{(x,y,z)}$ below, 
$\rho_{\cV_{12}}$ on $P(v_1,v_2,v_3,v_{13},v_{23})$ is written explicitly as follows:
\vspace{-10pt}

\tdplotsetmaincoords{60}{35}
\begin{minipage}[t]{0.55\textwidth}
\vspace{10pt}
\begin{align} \label{equ:Berk retraction vertex}
\begin{split}
   &\text{for }
   (x,y,z) \in  P(v_1,v_2,v_3,v_{13},v_{23}) \setminus \{v_3\}, \\
   &\rho_{\cV_{12}} \big((x,y,z)\big)= \Big(x+\left(1-\frac{t}{2}\right)z, y+ \frac{t}{2}z,0\Big) \\
   &\text{where } t=\frac{2y}{x+y+1}. 
\end{split}
\end{align}
\end{minipage} \hspace{-60pt}
\begin{minipage}[t]{0.5\textwidth}
\begin{figure}[H]
    \begin{tikzpicture}[tdplot_main_coords,scale=2.5] 
\node[left] at (-1,0,0.05) {$(1,0,0)=v_2$};
\node[right] at (1,0,0) {$v_1=(0,1,0)$};
\node[right] at (0,1,0) {$v_3=(-1,0,0)$};
\node[left] at (0,0,1) {$(\frac{1}{2},\frac{1}{2},1)=v_{21}$};
\node[right] at (1/2,1/2,1) {$v_{13}=(-\frac{1}{2},\frac{1}{2},1)$};
\node[left] at (-1/2,1/2,1) {$(0,0,1)=v_{23}$};
\node[left] at(-1/2,1/2,3/2) {$z$};
\node[below] at (-3/2,-1/2,0) {$x$};
\node[below] at (7/4,-1/4,0){$y$};
\node[left] at (0,0,-0.05) {$(\frac{1}{2},\frac{1}{2},0)$};
\node[right] at (1/2,1/2,0) {$(-\frac{1}{2},\frac{1}{2},0)$};
\node[above] at (0.2,1/2,1.5) {$(-\frac{1}{4},\frac{1}{4},1)$};
\node[below] at (1/2,-1/2,-0.1) {$(0,\frac{1}{3},0)$};
\node[above] at (1,1/2,3/2) {$(-\frac{1}{2},\frac{1}{6},\frac{2}{3})$};
\draw[densely dotted,very thin] (0,1/3,0)--(1/2,-1/2,-0.1);
\draw[densely dotted,very thin]  (0,2/3,2/3) --(1,1/2,3/2);
\draw[densely dotted,very thin]  (0,1/2,1)  --(0,1/2,1.45) ;
\filldraw (0,0,0) circle (0.2pt);
\filldraw (1/2,1/2,0) circle (0.2pt);
\filldraw (1/4,3/4,1/2) circle (0.2pt);
\filldraw (0,1/2,1) circle (0.2pt);
\filldraw (0,1/3,0) circle (0.2pt);
\filldraw (0,2/3,2/3) circle (0.2pt);
\draw[->, very thin] (-1/2,1/2,0)--(-1/2,1/2,3/2);
\draw[->, very thin] (-1/2,1/2,0)--(-3/2,-1/2,0);
\draw[->, very thin] (-1/2,1/2,0)--(7/4,-1/4,0);
\draw[gray, very thin] (0,0,0)--(0,1,0);
\draw[gray, very thin] (1,0,0)--(-1/2,1/2,0);
\draw[gray, very thin] (-1,0,0)--(1/2,1/2,0);
\draw[gray, very thin] (0,1/2,1)--(0,1,0);
\draw[gray, very thin] (-1/2,1/2,1)--(1/4,3/4,1/2);
\draw[gray, very thin] (1/2,1/2,1)--(-1/4,3/4,1/2);
\draw (-1,0,0)--(1,0,0)--(0,1,0)--(-1,0,0);
\draw (0,0,1)--(1/2,1/2,1)--(1,0,0);
\draw[dashed] (-1,0,0)--(1/2,1/2,1)--(0,1,0);
\draw (-1,0,0)--(0,0,1)--(1,0,0);
\draw (-1,0,0)--(-1/2,1/2,1)--(0,0,1);
\draw (-1/2,1/2,1)--(1/2,1/2,1);
\draw[dashed] (0,1,0)--(-1/2,1/2,1);
\end{tikzpicture}
    \label{fig:coordinates_additional_faces}
\end{figure}
\end{minipage}
%

We give a picture of the retraction $\rho_{\cV_{12}}$ for various values of $t$:
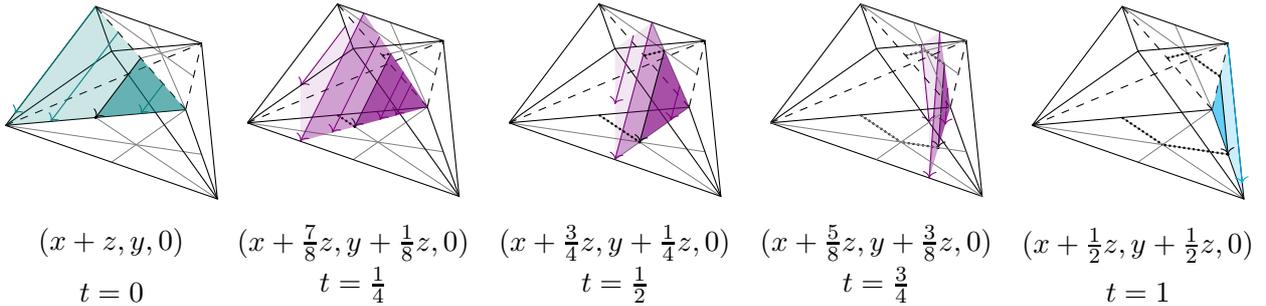
\begin{figure}[H]
    \centering
    \begin{tikzpicture}[tdplot_main_coords,scale=1.7] 
\draw[gray, very thin] (0,0,0)--(0,1,0);
\draw[gray, very thin] (1,0,0)--(-1/2,1/2,0);
\draw[gray, very thin] (-1,0,0)--(1/2,1/2,0);
\draw[gray, very thin] (0,1/2,1)--(0,1,0);
\draw[gray, very thin] (-1/2,1/2,1)--(1/4,3/4,1/2);
\draw[gray, very thin] (1/2,1/2,1)--(-1/4,3/4,1/2);
\draw (-1,0,0)--(1,0,0)--(0,1,0)--(-1,0,0);
\draw (0,0,1)--(1/2,1/2,1)--(1,0,0);
\draw[dashed] (-1,0,0)--(1/2,1/2,1)--(0,1,0);
\draw (-1,0,0)--(0,0,1)--(1,0,0);
\draw (-1,0,0)--(-1/2,1/2,1)--(0,0,1);
\draw (-1/2,1/2,1)--(1/2,1/2,1);
\draw[dashed] (0,1,0)--(-1/2,1/2,1);
\fill[fill=teal, fill opacity=0.2]  (0,1,0) -- (-1,0,0)--(-1/2,1/2,1);
\fill[fill=teal, fill opacity=0.5]  (0,1,0) -- (-1/4,3/4,1/2)--(-1/2,1/2,0);
\draw[->,teal] (-3/8,5/8,3/4)--(-3/4,1/4,0);
\draw[->,black] (-1/4,3/4,1/2)--(-1/2,1/2,0);
\draw[->,teal] (-1/8,7/8,1/4)--(-1/4,3/4,0);
\draw[->,teal] (-1/2,1/2,1)--(-0.95,0.05,0.1);
\filldraw (-1/2,1/2,0) circle (0.2pt);
\node[below] at (0,0,-1/2) {
	$(x+z,y,0)$};
\node[below] at (0,0,-1) {
	$t=0$};
\end{tikzpicture}
\begin{tikzpicture}[tdplot_main_coords,scale=1.7] 
\draw[gray, very thin] (0,0,0)--(0,1,0);
\draw[gray, very thin] (1,0,0)--(-1/2,1/2,0);
\draw[gray, very thin] (-1,0,0)--(1/2,1/2,0);
\draw[gray, very thin] (0,1/2,1)--(0,1,0);
\draw[gray, very thin] (-1/2,1/2,1)--(1/4,3/4,1/2);
\draw[gray, very thin] (1/2,1/2,1)--(-1/4,3/4,1/2);
\draw[->,violet] (-1/4,1/2,1)--(-1/2,0,0);
\draw[->,violet] (-1/4,1/4,1)--(-1/2,0,1/2);
\draw[->,black] (-1/7,5/7,4/7)--(-2/7,18/42,0);
\filldraw (-2/7,18/42,0) circle (0.2pt);
\draw[densely dotted, thick] (-1/2,1/2,0)--(-2/7,18/42,0);
\draw[densely dotted, thick] (-1/4,3/4,1/2)--(-1/7,5/7,4/7);
\draw[->,violet] (-3/16,5/8,3/4)--(-3/8,1/4,0);
\draw[->,violet] (-3/32,13/16,3/8)--(-3/16,5/8,0);
\draw[->,violet] (-1/16,7/8,1/4)--(-1/8,3/4,0);
\draw (-1,0,0)--(1,0,0)--(0,1,0)--(-1,0,0);
\draw (0,0,1)--(1/2,1/2,1)--(1,0,0);
\draw[dashed] (-1,0,0)--(1/2,1/2,1)--(0,1,0);
\draw (-1,0,0)--(0,0,1)--(1,0,0);
\draw (-1,0,0)--(-1/2,1/2,1)--(0,0,1);
\draw (-1/2,1/2,1)--(1/2,1/2,1);
\draw[dashed] (0,1,0)--(-1/2,1/2,1);
\fill[fill=violet, fill opacity=0.1] (0,1,0)-- (-1/4,1/2,1)--(-1/4,1/4,1)--(-1/2,0,1/2)--(-1/2,0,0)--(0,1,0);
\fill[fill=violet, fill opacity=0.3] (0,1,0)-- (-1/4,1/2,1)--(-1/2,0,0);
\fill[fill=violet, fill opacity=0.5]  (0,1,0) -- (-1/7,5/7,4/7)--(-2/7,18/42,0);
\node[below] at (0,0,-1/2) {
	$(x+\frac{7}{8}z,y+\frac{1}{8}z,0)$};
\node[below] at (0,0,-7/8) {
	$t=\frac{1}{4}$};
\end{tikzpicture}
\begin{tikzpicture}[tdplot_main_coords,scale=1.7] 
\draw[gray, very thin] (0,0,0)--(0,1,0);
\draw[gray, very thin] (1,0,0)--(-1/2,1/2,0);
\draw[gray, very thin] (-1,0,0)--(1/2,1/2,0);
\draw[gray, very thin] (0,1/2,1)--(0,1,0);
\draw[gray, very thin] (-1/2,1/2,1)--(1/4,3/4,1/2);
\draw[gray, very thin] (1/2,1/2,1)--(-1/4,3/4,1/2);
\draw (-1,0,0)--(1,0,0)--(0,1,0)--(-1,0,0);
\draw (0,0,1)--(1/2,1/2,1)--(1,0,0);
\draw[dashed] (-1,0,0)--(1/2,1/2,1)--(0,1,0);
\draw (-1,0,0)--(0,0,1)--(1,0,0);
\draw (-1,0,0)--(-1/2,1/2,1)--(0,0,1);
\draw (-1/2,1/2,1)--(1/2,1/2,1);
\draw[dashed] (0,1,0)--(-1/2,1/2,1);
\fill[fill=violet, fill opacity=0.1]  (0,1,0) -- (0,1/2,1)--(0,0,1)--(0,0,0)--(0,1,0);
\fill[fill=violet, fill opacity=0.5] (0,1,0)-- (0,2/3,2/3)--(0,1/3,0);
\fill[fill=violet, fill opacity=0.3] (0,1,0) --(0,1/2,1)--(0,0,0);
\draw[->,violet] (0,1/2,1)--(0,0,0);
\draw[->,violet] (0,1/4,1)--(0,0,1/2);
\draw[->,violet] (0,7/8,1/4)--(0,3/4,0);
\draw[->,black] (0,2/3,2/3)--(0,1/3,0);
\filldraw (0,1/3,0) circle (0.2pt);
\draw[densely dotted, thick] (-1/2,1/2,0)-- (0,1/3,0);
\draw[densely dotted, thick] (-1/4,3/4,1/2)--(0,2/3,2/3);
\node[below] at (0,0,-1/2) {
	$(x+\frac{3}{4}z,y+\frac{1}{4}z,0)$};
\node[below] at (0,0,-7/8) {
	$t=\frac{1}{2}$};
\end{tikzpicture}
\begin{tikzpicture}[tdplot_main_coords,scale=1.7] 
\draw[->,violet] (1/16,7/8,1/4)--(1/8,3/4,0);
\draw[->,violet] (1/4,1/2,1)--(1/2,0,0);
\draw[->,violet] (1/4,1/4,1)--(1/2,0,1/2);
\draw[->,black] (1/7,5/7,4/7)--(2/7,18/42,0);
\filldraw (2/7,18/42,0) circle (0.2pt);
\draw[densely dotted, thick] (-1/2,1/2,0)-- (0,1/3,0);
\draw[densely dotted, thick] (2/7,18/42,0)-- (0,1/3,0);
\draw[densely dotted, thick] (-1/4,3/4,1/2)--(0,2/3,2/3);
\draw[densely dotted, thick] (1/7,5/7,4/7)--(0,2/3,2/3);
\draw (-1,0,0)--(1,0,0)--(0,1,0)--(-1,0,0);
\draw (0,0,1)--(1/2,1/2,1)--(1,0,0);
\draw[dashed] (-1,0,0)--(1/2,1/2,1)--(0,1,0);
\draw (-1,0,0)--(0,0,1)--(1,0,0);
\draw (-1,0,0)--(-1/2,1/2,1)--(0,0,1);
\draw (-1/2,1/2,1)--(1/2,1/2,1);
\draw[dashed] (0,1,0)--(-1/2,1/2,1);
\fill[fill=violet, fill opacity=0.1] (0,1,0)-- (1/4,1/2,1)--(1/4,1/4,1)--(1/2,0,1/2)--(1/2,0,0)--(0,1,0);
\fill[fill=violet, fill opacity=0.3] (0,1,0)-- (1/4,1/2,1)--(1/2,0,0);
\fill[fill=violet, fill opacity=0.5]  (0,1,0) -- (1/7,5/7,4/7)--(2/7,18/42,0);
\draw[gray, very thin] (0,0,0)--(0,1,0);
\draw[gray, very thin] (1,0,0)--(-1/2,1/2,0);
\draw[gray, very thin] (-1,0,0)--(1/2,1/2,0);
\draw[gray, very thin] (0,1/2,1)--(0,1,0);
\draw[gray, very thin] (-1/2,1/2,1)--(1/4,3/4,1/2);
\draw[gray, very thin] (1/2,1/2,1)--(-1/4,3/4,1/2);
\node[below] at (0,0,-1/2) {
	$(x+\frac{5}{8}z,y+\frac{3}{8}z,0)$};
\node[below] at (0,0,-7/8) {
	$t=\frac{3}{4}$};
\end{tikzpicture}
\begin{tikzpicture}[tdplot_main_coords,scale=1.7] 
\draw[gray, very thin] (0,0,0)--(0,1,0);
\draw[gray, very thin] (1,0,0)--(-1/2,1/2,0);
\draw[gray, very thin] (-1,0,0)--(1/2,1/2,0);
\draw[gray, very thin] (0,1/2,1)--(0,1,0);
\draw[gray, very thin] (-1/2,1/2,1)--(1/4,3/4,1/2);
\draw[gray, very thin] (1/2,1/2,1)--(-1/4,3/4,1/2);
\draw (-1,0,0)--(1,0,0)--(0,1,0)--(-1,0,0);
\draw (0,0,1)--(1/2,1/2,1)--(1,0,0);
\draw[dashed] (-1,0,0)--(1/2,1/2,1)--(0,1,0);
\draw (-1,0,0)--(0,0,1)--(1,0,0);
\draw (-1,0,0)--(-1/2,1/2,1)--(0,0,1);
\draw (-1/2,1/2,1)--(1/2,1/2,1);
\draw[dashed] (0,1,0)--(-1/2,1/2,1);
\fill[fill=cyan, fill opacity=0.2]  (0,1,0) -- (1,0,0)--(1/2,1/2,1);
\fill[fill=cyan, fill opacity=0.5]  (0,1,0) --(1/4,3/4,0.5)--(0.5,0.5,0);
\draw[->,cyan] (1/2,1/2,1)--(0.95,0.05,0.1);
\draw[->,black] (1/4,3/4,0.5)--(0.5,0.5,0);
\filldraw (0.5,0.5,0) circle (0.2pt);
\draw[densely dotted, thick] (-1/2,1/2,0)-- (0,1/3,0);
\draw[densely dotted, thick] (0.5,0.5,0)-- (0,1/3,0);
\draw[densely dotted, thick] (-1/4,3/4,1/2)--(0,2/3,2/3);
\draw[densely dotted, thick] (1/4,3/4,0.5)--(0,2/3,2/3);
\node[below] at (0,0,-1/2) {
	$(x+\frac{1}{2}z,y+\frac{1}{2}z,0)$};
\node[below] at (0,0,-1) {
	$t=1$};
\end{tikzpicture}
\end{figure}

\noindent For purposes which will be clear in the construction of the local combinatorial retraction in \cref{subsec:local_combinatorial}, we consider a further toric blow-up. Let $H_{123}: \mathscr{G} \rightarrow \cV_{123}$ be the blow-up along the disjoint toric strata $D_2 \cap E_{13}$ and $E_{12} \cap D_3'$; this yields two new components in the toric boundary, denoted by $E_{123}$ and $E'_{123}$. It follows that the slice of the fan of $\mathscr{G}$ is obtained from the slice of $\cV_{123}$ as star subdivision along the edges $\langle v_2,v_{13}\rangle $ and $\langle v_{12},v_3'\rangle $. 
\vspace{-10pt}

\begin{minipage}{0.65\textwidth}
\noindent In particular, the skeleton $\Sk(\mathscr{G})$ is obtained from $\Sk(\cV_{123})$ by 
\begin{itemize}
\itemsep0pt
    \item[1.] the star subdivision of the edge $\langle v_2,v_{13}\rangle $, which turns the four $3$-cells of $\Sk(\cV_{123})$ into eight $3$-cells;
    \item[2.] adding an additional $3$-cell $\tau = \langle v_{12},v_{13},v_{23},v'_{123}\rangle $, where we denote by $v'_{123}$ the new vertex corresponding to $E'_{123}$.
\end{itemize}
\end{minipage} \hfill
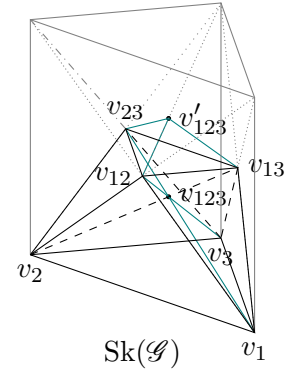
\begin{minipage}{0.35\textwidth}
\tdplotsetmaincoords{60}{35}
\begin{figure}[H]
\centering
\begin{tikzpicture}[tdplot_main_coords,scale=1.8]
\filldraw (0,1/3,4/3) circle (0.4pt);
\filldraw (0,1/3,2/3) circle (0.4pt);
\draw[teal] (-1/2,1/2,1)--(0,1/3,4/3)--(0,0,1);
\draw[teal] (1/2,1/2,1)--(0,1/3,4/3);
\draw[teal] (0,0,1)--(0,1/3,2/3)--(0,1,0);
\draw[teal] (-1/2,1/2,1)--(0,1/3,2/3)--(1,0,0);
\node[right] at (0,1/3,4/3){$v'_{123}$};
\node[right] at (0,1/3,2/3){$v_{123}$};
\node[below] at (-1,0,0) {$v_2$};
\node[below] at (1,0,0) {$v_1$};
\node[below] at (0,1,0) {$v_3$};
\node[left] at (0,0,1) {$v_{12}$};
\node[right] at (1/2,1/2,1) {$v_{13}$};
\node[above] at (-1/2,1/2,1) {$v_{23}$};
\draw[gray] (-1,0,0)--(-1,0,2)--(1,0,2)--(1,0,0);
\draw[gray] (-1,0,2)--(0,1,2)--(0,1,0);
\draw[gray] (0,1,2)--(1,0,2);
\draw[gray,densely dotted] (1,0,0)--(-1,0,2);
\draw[gray,densely dotted] (1,0,0)--(0,1,2);
\draw[gray,dotted] (0,1,2)--(-1,0,0);
\draw (-1,0,0)--(1,0,0)--(0,1,0)--(-1,0,0);
\draw[gray,densely dotted] (0,0,1)--(0,1,2);
\draw[gray,densely dotted] (0,0,1)--(1,0,2);
\draw[gray,densely dotted] (1/2,1/2,1)--(1,0,2);
\draw[gray,dashdotted] (-1/2,1/2,1)--(-1,0,2);
\draw (0,0,1)--(1/2,1/2,1)--(1,0,0);
\draw[dashed] (-1,0,0)--(1/2,1/2,1)--(0,1,0);
\draw (-1,0,0)--(0,0,1)--(1,0,0);
\draw (-1,0,0)--(-1/2,1/2,1)--(0,0,1);
\draw (-1/2,1/2,1)--(1/2,1/2,1);
\draw[dashed] (0,1,0)--(-1/2,1/2,1);
\node[below] at (0,0,-0.3) {$\Sk(\mathscr{G})$};
\end{tikzpicture}
\end{figure}
\end{minipage}

\noindent The diagram below summarizes the resolutions of $\U$ we constructed and studied so far:

\begin{center}
    $\arraycolsep=1.6pt\def\arraystretch{1}
    \begin{array}{ccccccccc}
         \mathscr{G}
         & \xrightarrow[H_{123}]{
            \substack{
             \text{blow-up of } \\
            D_2 \cap E_{13}, E_{12} \cap D'_3
            }
            }
         & \cV_{123}
         & \xrightarrow[H_{23} \circ H_{13} \circ H_{12}]{\substack{
             \text{blow-up of } \\
            S_{12},S_{13},S_{23}
            }}
         & \U_{12} 
         & \xrightarrow[G_{12} \circ G_1 ]{\substack{
             \text{blow-up of } \\
            D_1, D_2
            }}
         & \U. \\
         \cup & & 
         \cup & & 
         \cup & & \\
         E_{123}, E'_{123} & &
         E_{12}, E_{13}, E_{23} & &
         S_{12}, S_{13}, S_{23} & &
    \end{array}
    $
\end{center}

\subsection{Local combinatorial retraction} 
\label{subsec:local_combinatorial}
Other resolutions of $\cU$ can be obtained by blowing-up the divisors of the special fiber in a different order. Given any order $(i_1,i_2,i_3)$ on $\{1,2,3\}$, we denote 

\begin{center}
    $\arraycolsep=1.6pt\def\arraystretch{1}
    \begin{array}{ccccccccc}
         \cV_{i_1 i_2 i_3}
         & \xrightarrow{\substack{\text{blow-up of } \\S_{i_1i_3}, S_{i_2 i_3}}}
         & \cV_{i_1 i_2}
         & \xrightarrow{\substack{\text{blow-up of }\\S_{i_1i_2}}}
         & \U_{i_1i_2}
         & \xrightarrow{\substack{\text{blow-up of }\\D_{i_1},D_{i_2}}}
         & \U. \\
         \cup & & 
         \cup & & 
         \cup & &  \\
         E_{i_1 i_3}, E_{i_2 i_3} & &
         E_{i_1 i_2} & &
         S_{i_1i_2}, S_{i_1 i_3}, S_{i_2 i_3} & &
    \end{array}
    $
\end{center}

The refinement of the fan of $\U$ corresponding to $\cV_{i_1 i_2 i_3}$ is such that the skeleton $\Sk(\cV_{i_1 i_2 i_3})=\Sk(\cV_{123})$ as subspaces in the Berkovich space of $\cU_K$; it is independent on the chosen order so that we simply denote this subspace by $\Sk(\cV)$. However, the models $\cV_{i_1 i_2 i_3}$ and $\cV_{123}$ induce in general different simplicial subdivisions and different retractions onto $\langle v_1,v_2,v_3\rangle $. For instance, the only edge in the interior of $\Sk(\cV_{i_1 i_2 i_3})$ is $\langle v_{i_2},v_{i_1 i_3}\rangle $, which indeed depends on the chosen order.
Here below we illustrate the skeletons and the Berkovich retractions in a couple of examples. 

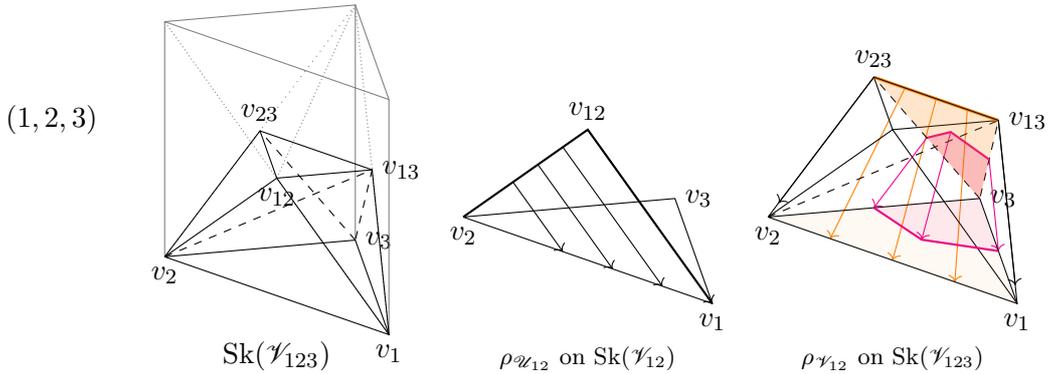
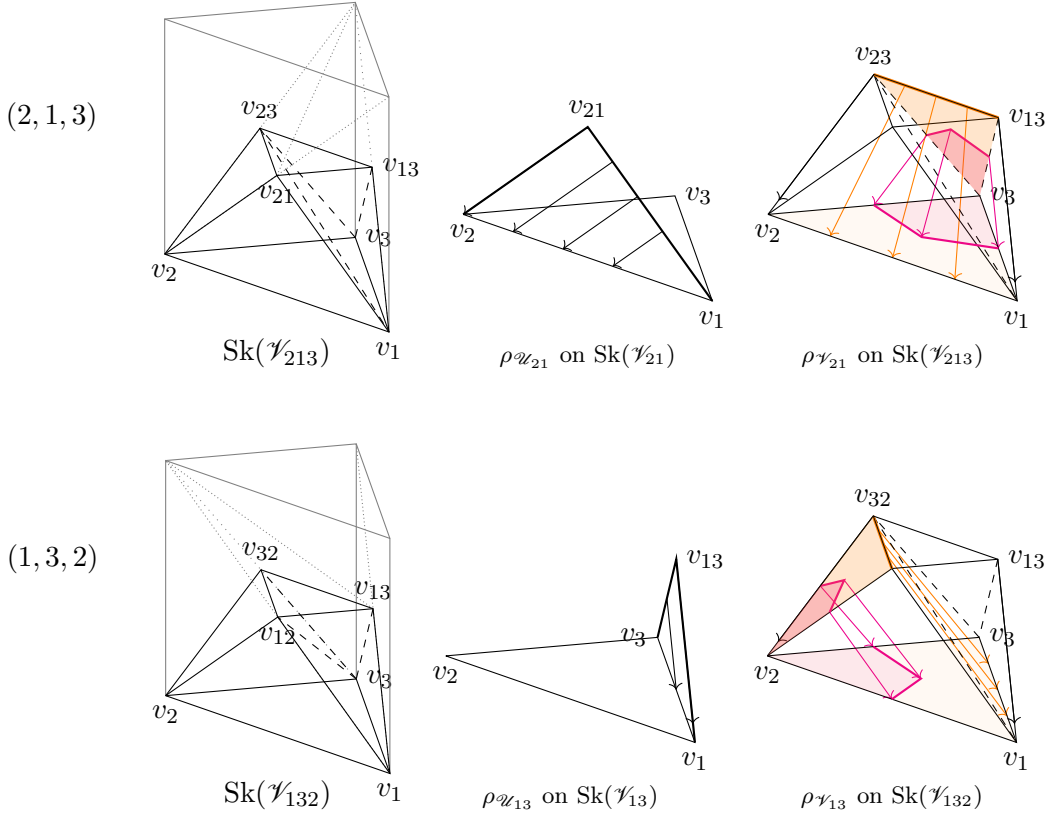
\tdplotsetmaincoords{60}{35}
\begin{figure}[H]
    \centering
\begin{tikzpicture}[tdplot_main_coords,scale=1.8] 
\node[left] at (-1.5,0,1) 
{$(1,2,3)$};
\node[below] at (-1,0,0) {$v_2$};
\node[below] at (1,0,0) {$v_1$};
\node[right] at (0,1,0) {$v_3$};
\node[below] at (0,0,1) {$v_{12}$};
\node[right] at (1/2,1/2,1) {$v_{13}$};
\node[above] at (-1/2,1/2,1) {$v_{23}$};
\draw[gray] (-1,0,0)--(-1,0,2)--(1,0,2)--(1,0,0);
\draw[gray] (-1,0,2)--(0,1,2)--(0,1,0);
\draw[gray] (0,1,2)--(1,0,2);
\draw[gray,densely dotted] (0,0,1)--(0,1,2);
\draw[gray,densely dotted] (1,0,0)--(-1,0,2);
\draw[gray,densely dotted] (1,0,0)--(0,1,2);
\draw[gray,dotted] (0,1,2)--(-1,0,0);
\draw (-1,0,0)--(1,0,0)--(0,1,0)--(-1,0,0);
\draw (0,0,1)--(1/2,1/2,1)--(1,0,0);
\draw[dashed] (-1,0,0)--(1/2,1/2,1)--(0,1,0);
\draw (-1,0,0)--(0,0,1)--(1,0,0);
\draw (-1,0,0)--(-1/2,1/2,1)--(0,0,1);
\draw (-1/2,1/2,1)--(1/2,1/2,1);
\draw[dashed] (0,1,0)--(-1/2,1/2,1);
\node[below] at (0,0,-0.3) {$\Sk(\cV_{123})$};
\end{tikzpicture}
\begin{tikzpicture}[tdplot_main_coords,scale=2] 
\node at (0,0,-3/4) {\footnotesize $\rho_{\cU_{12}}$ on $\Sk(\cV_{12})$};
\node[below] at (-1,0,0) {$v_2$};
\node[below] at (1,0,0) {$v_1$};
\node[right] at (0,1,0) {$v_3$};
\node[above] at (0,0,1) {$v_{12}$};
\draw (-1,0,0)--(1,0,0)--(0,1,0)--(-1,0,0);
\draw[thick] (-1,0,0)--(0,0,1)--(1,0,0);
\draw[->] (0.1,0,0.9)--(1,0,0);
\draw[->] (-0.2,0,0.8)--(0.6,0,0);
\draw[->] (-0.4,0,0.6)--(0.2,0,0);
\draw[->] (-0.6,0,0.4)--(-0.2,0,0);
\end{tikzpicture}
\begin{tikzpicture}[tdplot_main_coords,scale=2] 
\tikzset{p2/.style={preaction={
			draw,orange,
			double=orange,
			double distance=.8\pgflinewidth,
}}}
\node at (0,0,-3/4) {\footnotesize $\rho_{\cV_{12}}$ on $\Sk(\cV_{123})$};
\node[below] at (-1,0,0) {$v_2$};
\node[below] at (1,0,0) {$v_1$};
\node[right] at (0,1,0) {$v_3$};
\node[right] at (1/2,1/2,1) {$v_{13}$};
\node[above] at (-1/2,1/2,1) {$v_{23}$};
\draw[->,orange] (0,1/2,1)--(0,0,0);
\draw[->,orange] (-1/4,1/2,1)--(-1/2,0,0);
\draw[->,orange] (1/4,1/2,1)--(1/2,0,0);
\draw[->,magenta] (-1/4,3/4,1/2)--(-1/2,1/2,0);
\draw[->,magenta] (1/4,3/4,0.5)--(0.5,0.5,0);
\draw[->,magenta](0,2/3,2/3)--(0,1/3,0);
\draw (-1,0,0)--(1,0,0)--(0,1,0)--(-1,0,0);
\draw (0,0,1)--(1/2,1/2,1)--(1,0,0);
\draw[dashed] (-1,0,0)--(1/2,1/2,1)--(0,1,0);
\draw (-1,0,0)--(0,0,1)--(1,0,0);
\draw (-1,0,0)--(-1/2,1/2,1)--(0,0,1);
\draw[p2] (-1/2,1/2,1)--(1/2,1/2,1);
\draw[magenta,thick] (1/4,3/4,1/2) -- (0,2/3,2/3)--(-1/4,3/4,1/2);
\draw[magenta,thick]  (-1/2,1/2,0)--(0,1/3,0)--(1/2,1/2,0);
\draw[dashed] (0,1,0)--(-1/2,1/2,1);
\draw[->] (-1/2,1/2,1)--(-0.95,0.05,0.1);
\draw[->] (1/2,1/2,1)--(0.95,0.05,0.1);
\fill[fill=magenta, fill opacity=0.2] (1/4,3/4,1/2) -- (0,1,0)--(-1/4,3/4,1/2)--(0,2/3,2/3);
\fill[fill=magenta, fill opacity=0.07] (0,1,0) -- (1/2,1/2,0)--(0,1/3,0)--(-1/2,1/2,0);
\fill[fill=orange, fill opacity=0.2] (0,1,0) -- (1/2,1/2,1)--(-1/2,1/2,1);
\fill[fill=orange, fill opacity=0.05] (0,1,0) -- (1,0,0)--(-1,0,0);
\end{tikzpicture}
\end{figure}

\begin{figure} [H]
    \centering
\begin{tikzpicture}[tdplot_main_coords,scale=1.8] 
\node[left] at (-1.5,0,1) 
{$(2,1,3)$};
\node[below] at (-1,0,0) {$v_2$};
\node[below] at (1,0,0) {$v_1$};
\node[right] at (0,1,0) {$v_3$};
\node[below] at (0,0,1) {$v_{21}$};
\node[right] at (1/2,1/2,1) {$v_{13}$};
\node[above] at (-1/2,1/2,1) {$v_{23}$};
\draw[gray] (-1,0,0)--(-1,0,2)--(1,0,2)--(1,0,0);
\draw[gray] (-1,0,2)--(0,1,2)--(0,1,0);
\draw[gray] (0,1,2)--(1,0,2);
\draw[gray,densely dotted] (0,0,1)--(0,1,2);
\draw[gray,densely dotted] (-1,0,0)--(0,1,2);
\draw[gray,densely dotted] (-1,0,0)--(1,0,2);
\draw[gray,dotted] (0,1,2)--(1,0,0);
\draw (-1,0,0)--(1,0,0)--(0,1,0)--(-1,0,0);
\draw (0,0,1)--(1/2,1/2,1)--(1,0,0);
\draw[dashed] (1/2,1/2,1)--(0,1,0);
\draw[dashed] (-1/2,1/2,1)--(1,0,0);
\draw (-1,0,0)--(0,0,1)--(1,0,0);
\draw (-1,0,0)--(-1/2,1/2,1)--(0,0,1);
\draw (-1/2,1/2,1)--(1/2,1/2,1);
\draw[dashed] (0,1,0)--(-1/2,1/2,1);
\node[below] at (0,0,-0.3) {$\Sk(\cV_{213})$};
\end{tikzpicture}
\begin{tikzpicture}[tdplot_main_coords,scale=2] 
\node at (0,0,-3/4) {\footnotesize $\rho_{\cU_{21}}$ on $\Sk(\cV_{21})$};
\node[below] at (-1,0,0) {$v_2$};
\node[below] at (1,0,0) {$v_1$};
\node[right] at (0,1,0) {$v_3$};
\node[above] at (0,0,1) {$v_{21}$};
\draw (-1,0,0)--(1,0,0)--(0,1,0)--(-1,0,0);
\draw[thick] (-1,0,0)--(0,0,1)--(1,0,0);
\draw[->] (-0.1,0,0.9)--(-1,0,0);
\draw[->] (0.2,0,0.8)--(-0.6,0,0);
\draw[->] (0.4,0,0.6)--(-0.2,0,0);
\draw[->] (0.6,0,0.4)--(0.2,0,0);
\end{tikzpicture}
\begin{tikzpicture}[tdplot_main_coords,scale=2] 
\tikzset{p2/.style={preaction={
			draw,orange,
			double=orange,
			double distance=.8\pgflinewidth,
}}}
\node at (0,0,-3/4) {\footnotesize $\rho_{\cV_{21}}$ on $\Sk(\cV_{213})$};
\node[below] at (-1,0,0) {$v_2$};
\node[below] at (1,0,0) {$v_1$};
\node[right] at (0,1,0) {$v_3$};
\node[right] at (1/2,1/2,1) {$v_{13}$};
\node[above] at (-1/2,1/2,1) {$v_{23}$};
\draw[->,orange] (0,1/2,1)--(0,0,0);
\draw[->,orange] (-1/4,1/2,1)--(-1/2,0,0);
\draw[->,orange] (1/4,1/2,1)--(1/2,0,0);
\draw[->,magenta] (-1/4,3/4,1/2)--(-1/2,1/2,0);
\draw[->,magenta] (1/4,3/4,0.5)--(0.5,0.5,0);
\draw[->,magenta](0,2/3,2/3)--(0,1/3,0);
\draw (-1,0,0)--(1,0,0)--(0,1,0)--(-1,0,0);
\draw (0,0,1)--(1/2,1/2,1)--(1,0,0);
\draw[dashed] (-1/2,1/2,1)--(1,0,0);
\draw[dashed] (0,1,0)--(-1/2,1/2,1);
\draw[dashed] (1/2,1/2,1)--(0,1,0);
\draw (-1,0,0)--(0,0,1)--(1,0,0);
\draw (-1,0,0)--(-1/2,1/2,1)--(0,0,1);
\draw[p2] (-1/2,1/2,1)--(1/2,1/2,1);
\draw[magenta,thick] (1/4,3/4,1/2) -- (0,2/3,2/3)--(-1/4,3/4,1/2);
\draw[magenta,thick]  (-1/2,1/2,0)--(0,1/3,0)--(1/2,1/2,0);
\draw[->] (-1/2,1/2,1)--(-0.95,0.05,0.1);
\draw[->] (1/2,1/2,1)--(0.95,0.05,0.1);
\fill[fill=magenta, fill opacity=0.2] (1/4,3/4,1/2) -- (0,1,0)--(-1/4,3/4,1/2)--(0,2/3,2/3);
\fill[fill=magenta, fill opacity=0.07] (0,1,0) -- (1/2,1/2,0)--(0,1/3,0)--(-1/2,1/2,0);
\fill[fill=orange, fill opacity=0.2] (0,1,0) -- (1/2,1/2,1)--(-1/2,1/2,1);
\fill[fill=orange, fill opacity=0.05] (0,1,0) -- (1,0,0)--(-1,0,0);
\end{tikzpicture}
\end{figure}

\begin{figure} [H]
    \centering
\begin{tikzpicture}[tdplot_main_coords,scale=1.8] 
\node[left] at (-1.5,0,1) 
{$(1,3,2)$};
\node[below] at (-1,0,0) {$v_2$};
\node[below] at (1,0,0) {$v_1$};
\node[right] at (0,1,0) {$v_3$};
\node[below] at (0,0,1) {$v_{12}$};
\node[above] at (1/2,1/2,1) {$v_{13}$};
\node[above] at (-1/2,1/2,1) {$v_{32}$};
\draw[gray] (-1,0,0)--(-1,0,2)--(1,0,2)--(1,0,0);
\draw[gray] (-1,0,2)--(0,1,2)--(0,1,0);
\draw[gray] (0,1,2)--(1,0,2);
\draw[gray,densely dotted] (1,0,0)--(0,1,2);
\draw[gray,densely dotted] (1,0,0)--(-1,0,2);
\draw[gray,densely dotted] (1/2,1/2,1)--(-1,0,2);
\draw[gray,dotted] (0,1,0)--(-1,0,2);
\draw (-1,0,0)--(1,0,0)--(0,1,0)--(-1,0,0);
\draw (0,0,1)--(1/2,1/2,1)--(1,0,0);
\draw[dashed] (1/2,1/2,1)--(0,1,0);
\draw[dashed] (0,1,0)--(0,0,1);
\draw (-1,0,0)--(0,0,1)--(1,0,0);
\draw (-1,0,0)--(-1/2,1/2,1)--(0,0,1);
\draw (-1/2,1/2,1)--(1/2,1/2,1);
\draw[dashed] (0,1,0)--(-1/2,1/2,1);
\node[below] at (0,0,-0.3) {$\Sk(\cV_{132})$};
\end{tikzpicture}
\begin{tikzpicture}[tdplot_main_coords,scale=2] 
\node at (0,0,-3/4) {\footnotesize $\rho_{\cU_{13}}$ on $\Sk(\cV_{13})$};
\node[below] at (-1,0,0) {$v_2$};
\node[below] at (1,0,0) {$v_1$};
\node[left] at (0,1,0) {$v_3$};
\node[right] at (1/2,1/2,1) {$v_{13}$};
\draw (-1,0,0)--(1,0,0)--(0,1,0)--(-1,0,0);
\draw[thick] (1,0,0)--(1/2,1/2,1)--(0,1,0);
\draw[->] (1/2,1/2,1)--(0.95,0.05,0.1);
\draw[->] (1/4,3/4,0.5)--(0.5,0.5,0);
\end{tikzpicture}
\begin{tikzpicture}[tdplot_main_coords,scale=2] 
\tikzset{p2/.style={preaction={
			draw,orange,
			double=orange,
			double distance=.8\pgflinewidth,
}}}
\node at (0,0,-3/4) {\footnotesize $\rho_{\cV_{13}}$ on $\Sk(\cV_{132})$};
\node[below] at (-1,0,0) {$v_2$};
\node[below] at (1,0,0) {$v_1$};
\node[right] at (0,1,0) {$v_3$};
\node[right] at (1/2,1/2,1) {$v_{13}$};
\node[above] at (-1/2,1/2,1) {$v_{32}$};
\draw[->,orange] (-1/8,1/8,1)--(3/4,1/4,0);
\draw[->,orange] (-1/4,1/4,1)--(1/2,1/2,0);
\draw[->,orange] (-3/8,3/8,1)--(1/4,3/4,0);
\draw[->,magenta] (-3/4,1/4,1/2)--(-1/2,1/2,0);
\draw[->,magenta] (-1/2,0,0.5)--(0,0,0);
\draw[->,magenta](-1/2,1/6,2/3)--(0,1/3,0);
\draw (-1,0,0)--(1,0,0)--(0,1,0)--(-1,0,0);
\draw (0,0,1)--(1/2,1/2,1)--(1,0,0);
\draw[dashed] (-1/2,1/2,1)--(1,0,0);
\draw[dashed] (0,1,0)--(-1/2,1/2,1);
\draw[dashed] (1/2,1/2,1)--(0,1,0);
\draw (-1,0,0)--(0,0,1)--(1,0,0);
\draw (-1,0,0)--(-1/2,1/2,1)--(0,0,1);
\draw[p2] (-1/2,1/2,1)--(0,0,1);
\draw (-1/2,1/2,1)--(1/2,1/2,1);
\draw[magenta,thick] (-3/4,1/4,1/2) -- (-1/2,1/6,2/3)--(-1/2,0,1/2);
\draw[magenta,thick]  (-1/2,1/2,0)--(0,1/3,0)--(0,0,0);
\draw[->] (-1/2,1/2,1)--(-0.95,0.05,0.1);
\draw[->] (1/2,1/2,1)--(0.95,0.05,0.1);
\fill[fill=magenta, fill opacity=0.2] (-3/4,1/4,1/2) -- (-1/2,1/6,2/3)--(-1/2,0,1/2)--(-1,0,0);
\fill[fill=magenta, fill opacity=0.07] (-1,0,0) -- (-1/2,1/2,0)--(0,1/3,0)--(0,0,0);
\fill[fill=orange, fill opacity=0.2] (-1,0,0) -- (0,0,1)--(-1/2,1/2,1);
\fill[fill=orange, fill opacity=0.05] (0,1,0) -- (1,0,0)--(-1,0,0);
\end{tikzpicture}
\end{figure}

The blow-up of $\cV_{i_1 i_2 i_3}$ along the toric strata $D_{i_2} \cap E_{i_1 i_3}$ and $E_{i_1 i_2} \cap D_{i_3}'$ yields a refinement of the fan which coincides with the fan of $\mathscr{G}$, constructed at the end of \cref{subsec:local dominating model}. It follows that the model $\mathscr{G}$ dominates all resolutions $\cV_{i_1 i_2 i_3}$ independently on the order, hence all Berkovich retractions $\rho_{\cV_{i_1 i_2 i_3}}$, $\rho_{\cV_{i_1 i_2}}$ and $\rho_{\U_{i_1 i_2}}$ factors through $\rho_\mathscr{G}$.
\vspace{10pt}

Our goal is to construct a map $\pi$, composing the Berkovich retraction $\rho_\mathscr{G}$ with a collapse $\kappa$ of the additional 3-cell $\tau$ and a combinatorial retraction $\rho$

\begin{center}
    $\arraycolsep=1.6pt\def\arraystretch{1}
    \begin{array}{ccccccccc}
         & \pi: 
         & \U_K^{\an}
         & \xrightarrow{\rho_\mathscr{G}}
         & \Sk(\mathscr{G})
         & \xrightarrow[\text{collapse}]{\kappa}
         & \Sk(\cV)
         & \xrightarrow[\text{retraction}]{\rho}
         & \Sk(\U)
         \\
         & \rotatebox[origin=c]{270}{$=$} 
         & \cup & &
         \cup & & 
         \cup & & \cup \\
         \text{over }\Star(v_j)' \text{\hspace{10pt}}
         & 
         \rho_{\U_{i_1 i_2}}:
         & \pi^{-1}(\Star(v_j)') 
         & \xrightarrow{\rho_\mathscr{G}}
         & \Sk(\mathscr{G})
         & \xrightarrow{\rho_{\cV_{i_1 i_2 j}}}
         & \Sk(\cV_{i_1 i_2 j}) 
         & \xrightarrow{\rho_{\U_{i_1 i_2}}}
         & \Star(v_j)' 
    \end{array}
    $
\end{center}
such that, given any vertex $v_j$ in $\Sk(\U)$,
the restriction of $\pi$ over $\Star(v_j)'$ (the $\Star$ is taken with respect to the first barycentric subdivision, as in \cref{defn:disc in simplex}) is $\rho_{\U_{i_1 i_2}}$ for any order $(i_1,i_2,j)$ on $\{1,2,3\}$, i.e. any order where the index $j$ is the biggest. This guarantees that around each $v_j$, the map $\pi$ is the Berkovich retraction induced by a small resolution $\U_{i_1 i_2}$ where the strict transform of $D_j$ is isomorphic to $D_j$, so that we are in the set-up of \cref{intro:cor}.

\begin{itemize}
\item[-] The retraction $\rho$. We identify again the skeleton $\Sk(\cV)$ with the polyhedron in $\mathbb{R}^3$ described in \cref{subsec:local dominating model}. On the convex hull $P$ of $v_{23}, (-1/4,1/4,1), (0,1/3,1),(0,0,0),v_3$ and $(0,1/3,0)$, the retraction $\rho$ is given as follows
    
\begin{minipage}[t]{0.8 \textwidth}
\begin{align} \label{equ:combinatorial retraction vertx}
\begin{split}
&(x,y,z) \in  P \mapsto 
\begin{cases}
\Big(x+\left(1-\frac{t}{2}\right)z, y+ \frac{t}{2}z,0\Big) & \text{ if } x+\left(1-\frac{t}{2}\right)z \leqslant 0 \\
\Big(0, \frac{y}{x+1},0\Big) & \text{ if } x+\left(1-\frac{t}{2}\right)z \geqslant 0
\end{cases}\\
&\text{where }t=\frac{2y}{x+y+1}.
\end{split}
\end{align}
Here is a pictorial description for certain values of $t$:
\tdplotsetmaincoords{60}{35}
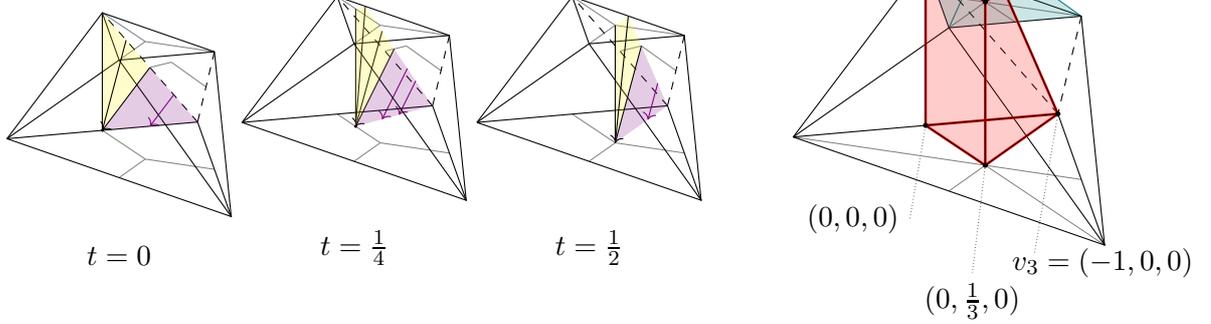
\begin{figure} [H]
   \begin{tikzpicture}[tdplot_main_coords,scale=1.8] 
\draw[gray, very thin] (0,0,0)--(0,1/3,0);
\draw[gray, very thin] (0,1/3,0)--(-1/2,1/2,0);
\draw[gray, very thin] (0,1/3,0)--(1/2,1/2,0);
\draw[gray, very thin] (0,0,1)--(0,1/3,1);
\draw[gray, very thin] (0,1/3,1)--(1/2,1/2,1);
\draw[gray, very thin] (0,1/3,1)--(-1/2,1/2,1);
\draw[gray, very thin] (0,2/3,2/3)--(1/4,3/4,1/2);
\draw[gray, very thin] (0,2/3,2/3)--(-1/4,3/4,1/2);
\draw (-1,0,0)--(1,0,0)--(0,1,0)--(-1,0,0);
\draw (0,0,1)--(1/2,1/2,1)--(1,0,0);
\draw[dashed] (1/2,1/2,1)--(0,1,0);
\draw (-1,0,0)--(0,0,1)--(1,0,0);
\draw (-1,0,0)--(-1/2,1/2,1)--(0,0,1);
\draw (-1/2,1/2,1)--(1/2,1/2,1);
\draw[dashed] (0,1,0)--(-1/2,1/2,1);
\fill[fill=yellow, fill opacity=0.2]  (-1/4,3/4,1/2) -- (-1/2,1/2,0)--(-1/2,1/2,1);
\fill[fill=violet, fill opacity=0.2]  (0,1,0) -- (-1/4,3/4,1/2)--(-1/2,1/2,0);
\draw[black] (-3/8,5/8,3/4)--(-1/2,1/2,0);
\draw[->,black] (-1/4,3/4,1/2)--(-1/2,1/2,0);
\draw[->,violet] (-1/8,7/8,1/4)--(-1/4,3/4,0);
\draw[black] (-1/2,1/2,1)--(-1/2,1/2,0);
\filldraw (-1/2,1/2,0) circle (0.2pt);
\node[below] at (0,0,-1/2) {
	$t=0$};
\end{tikzpicture}
%
\begin{tikzpicture}[tdplot_main_coords,scale=1.8] 
\draw[gray, very thin] (0,0,0)--(0,1/3,0);
\draw[gray, very thin] (0,1/3,0)--(-1/2,1/2,0);
\draw[gray, very thin] (0,1/3,0)--(1/2,1/2,0);
\draw[gray, very thin] (0,0,1)--(0,1/3,1);
\draw[gray, very thin] (0,1/3,1)--(1/2,1/2,1);
\draw[gray, very thin] (0,1/3,1)--(-1/2,1/2,1);
\draw[gray, very thin] (0,2/3,2/3)--(1/4,3/4,1/2);
\draw[gray, very thin] (0,2/3,2/3)--(-1/4,3/4,1/2);
\draw[black] (-1/4,1/2,1)--(-2/7,18/42,0);
\draw[black] (-2/7,18/42,1)--(-2/7,18/42,0.05);
\draw[->,black] (-1/7,5/7,4/7)--(-2/7,18/42,0); 
\filldraw (-2/7,18/42,0) circle (0.2pt);
\draw[black] (-3/16,5/8,3/4)--(-2/7,18/42,0);
\draw[->,violet] (-3/32,13/16,3/8)--(-3/16,5/8,0);
\draw[->,violet] (-1/16,7/8,1/4)--(-1/8,3/4,0);
\draw (-1,0,0)--(1,0,0)--(0,1,0)--(-1,0,0);
\draw (0,0,1)--(1/2,1/2,1)--(1,0,0);
\draw[dashed] (1/2,1/2,1)--(0,1,0);
\draw (-1,0,0)--(0,0,1)--(1,0,0);
\draw (-1,0,0)--(-1/2,1/2,1)--(0,0,1);
\draw (-1/2,1/2,1)--(1/2,1/2,1);
\draw[dashed] (0,1,0)--(-1/2,1/2,1);
\fill[fill=yellow, fill opacity=0.2] (-1/7,5/7,4/7)--(-1/4,1/2,1)--(-2/7,18/42,1)--(-2/7,18/42,0)--(-2/7,18/42,0);
\fill[fill=violet, fill opacity=0.2] (0,1,0)--(-1/7,5/7,4/7)--(-2/7,18/42,0);
\node[below] at (0,0,-1/2) {
	$t=\frac{1}{4}$};
\end{tikzpicture}
%
%
\begin{tikzpicture}[tdplot_main_coords,scale=1.8] 
\draw[gray, very thin] (0,0,0)--(0,1/3,0);
\draw[gray, very thin] (0,1/3,0)--(-1/2,1/2,0);
\draw[gray, very thin] (0,1/3,0)--(1/2,1/2,0);
\draw[gray, very thin] (0,0,1)--(0,1/3,1);
\draw[gray, very thin] (0,1/3,1)--(1/2,1/2,1);
\draw[gray, very thin] (0,1/3,1)--(-1/2,1/2,1);
\draw[gray, very thin] (0,2/3,2/3)--(1/4,3/4,1/2);
\draw[gray, very thin] (0,2/3,2/3)--(-1/4,3/4,1/2);
\draw (-1,0,0)--(1,0,0)--(0,1,0)--(-1,0,0);
\draw (0,0,1)--(1/2,1/2,1)--(1,0,0);
\draw[dashed] (1/2,1/2,1)--(0,1,0);
\draw (-1,0,0)--(0,0,1)--(1,0,0);
\draw (-1,0,0)--(-1/2,1/2,1)--(0,0,1);
\draw (-1/2,1/2,1)--(1/2,1/2,1);
\draw[dashed] (0,1,0)--(-1/2,1/2,1);
\fill[fill=yellow, fill opacity=0.2]  (0,2/3,2/3)--(0,1/3,0) --(0,1/3,1)-- (0,1/2,1)-- (0,2/3,2/3);
\fill[fill=violet, fill opacity=0.2] (0,1,0)-- (0,2/3,2/3)--(0,1/3,0);
\draw[black] (0,1/2,1)--(0,1/3,0);
\draw (0,1/3,1)--(0,1/3,0);
\draw[->,violet] (0,7/8,1/4)--(0,3/4,0);
\draw[->,black] (0,2/3,2/3)--(0,1/3,0);
\node[below] at (0,0,-1/2) {
	$t=\frac{1}{2}$};
\end{tikzpicture}
\end{figure}
\end{minipage}
\hspace{-120pt}
\begin{minipage}[t]{0.4\textwidth}
 \tdplotsetmaincoords{60}{35}
\begin{figure} [H]
    \centering
\begin{tikzpicture}[tdplot_main_coords,scale=2.5] 
\tikzset{p2/.style={preaction={
			draw,red,
			double=red,
			double distance=.4\pgflinewidth,
}}}
\node[left] at (-1/4,0,-1/4) {$(0,0,0)$};
\draw[densely dotted,very thin] (-1/4,0,-1/4) --(-1/2,1/2,0);
\draw[densely dotted,very thin] (-1/4,1,5/4) --(0,1/2,1);
\draw[densely dotted,very thin] (0,1/3,0)--(1/2,-1/2,-0.1);
\draw[densely dotted,very thin] (-1/2,4/3,3/2) --(0,1/3,1);
\draw[densely dotted,very thin] (-1/2,1/2,1.3) --(0,1/3,4/3);
\draw[densely dotted,very thin] (0,1,0) --(3/8,1-3/4,-3/8);
\node[right] at (-2/3,4/3,1.6) {$(0,\frac{1}{3},1)$};
\node[right] at (1/6,1-3/4,-4/8) {$v_3=(-1,0,0)$};
\node[left] at (-1/2,1/2,1) {$(0,0,1)=v_{23}$};
\node[left] at  (-1/2,1/2,1.3) {$(0,\frac{1}{3},\frac{4}{3})=v_{123}'$};
\node[right] at (0,1/2,1.6) {$(-\frac{1}{4},\frac{1}{4},1)$};
\node[below] at (1/2,-1/2,-0.1) {$(0,\frac{1}{3},0)$};
\draw[teal] (-1/2,1/2,1)--(0,1/3,4/3)--(0,0,1);
\draw[teal] (1/2,1/2,1)--(0,1/3,4/3);
\draw[gray, very thin] (0,0,0)--(0,1/3,0);
\draw[gray, very thin] (1,0,0)--(-1/2,1/2,0);
\draw[gray, very thin] (-1,0,0)--(1/2,1/2,0);
\draw[gray, very thin] (0,0,1)--(0,1/3,1);
\draw[gray, very thin] (-1/4,1/4,1)--(1/2,1/2,1);
\draw[gray, very thin] (1/4,1/4,1)--(-1/2,1/2,1);
\draw (-1,0,0)--(1,0,0)--(0,1,0)--(-1,0,0);
\draw (0,0,1)--(1/2,1/2,1)--(1,0,0);
\draw[dashed] (1/2,1/2,1)--(0,1,0);
\draw (-1,0,0)--(0,0,1)--(1,0,0);
\draw (-1,0,0)--(-1/2,1/2,1)--(0,0,1);
\draw (-1/2,1/2,1)--(1/2,1/2,1);
\draw[dashed] (0,1,0)--(-1/2,1/2,1);
\fill[fill=red, fill opacity=0.2] (-1/2,1/2,1) --(-1/2,1/2,0)--(0,1/3,0)--(0,1,0)--(0,1/2,1); 
\fill[fill=teal, fill opacity=0.2] (-1/2,1/2,1)--(0,1/3,4/3)--(1/2,1/2,1)--(0,0,1); 
\draw[p2]  (-1/2,1/2,1) --(-1/2,1/2,0);
\draw[p2]  (0,1/3,1) --(0,1/3,0);
\draw[p2]  (0,1/3,1) --(0,1/2,1);
\draw[p2]  (0,1/3,0) --(0,1,0);
\draw[p2]  (0,1/2,1) --(0,1,0);
\draw[p2]  (0,1/2,1) --(-1/2,1/2,1);
\draw[p2]  (0,1,0) --(-1/2,1/2,0);
\draw[p2]  (-1/2,1/2,1) --(0,1/3,1);
\draw[p2]  (-1/2,1/2,0) --(0,1/3,0);
\filldraw (-1/2,1/2,1) circle (0.3pt);
\filldraw (0,1/3,4/3) circle (0.3pt);
\filldraw (-1/2,1/2,0) circle (0.3pt);
\filldraw (0,1,0) circle (0.3pt);
\filldraw (0,1/2,1) circle (0.3pt);
\filldraw (0,1/3,0) circle (0.3pt);
\filldraw (0,1/3,1) circle (0.3pt);
\end{tikzpicture}
\end{figure}   
\end{minipage}

We extend the definition of $\rho$ to $\Sk(\cV)$ by symmetry along the medians of the triangles $\langle v_1,v_2,v_3\rangle $ and $\langle v_{12},v_{13},v_{23}\rangle $. In particular, we note that the image of $\langle v_{12},v_{13},v_{23}\rangle $ is the graph in $\langle v_1,v_2,v_3\rangle $ of \cref{defn:disc in simplex}.
 
\item[-] The combinatorial retraction $\pi'$. We define the collapse $\kappa$ as the projection of the additional $3$-cell $\tau$ of $\Sk(\mathscr{G})$ onto $\langle v_{12},v_{13},v_{23}\rangle $ along the $z$-direction. We call $\pi' \coloneqq \rho \circ \kappa$ the combinatorial retraction of the skeleton $\Sk(\mathscr{G})$ onto $\Sk(\U)=\langle v_1,v_2,v_3\rangle $.

\item[-] Finally, we check that $\pi'= \rho_{\U_{i_1 i_2}}$ over $\Star(v_j)'$. As the preimage of $\Star(v_j)'$ is disjoint from $\langle v_{12},v_{13},v_{23},v'_{123}\rangle $, we have to prove that $\rho= \rho_{\U_{i_1 i_2}}$. By symmetry of $\rho$, it is enough to check this for $v_3$. Over $\Star(v_3)'$ we have $\rho_{\U_{i_1 i_2}} = \rho_{\cV_{i_1 i_2}}$; there, the expression of $\rho_{\cV_{i_1 i_2}}$ determined in \cref{equ:Berk retraction vertex} coincides with the definition of $\rho$ in \cref{equ:combinatorial retraction vertx}, hence we conclude.
\end{itemize}

\subsection{Good minimal dlt models $\X_{ijkl}$} \label{subsec:min model Xijkl}
We return to the setting of \cref{subsec:setting quintic}. The purpose of this section is to compute various intersection numbers on the small resolutions of $\X$ used to define the retraction $\pi$, as this will allow us to compute the monodromy of the associated $\Z$-affine structure on $\Sk(X)$. 
\\Fix an order $(i,j,k,l,h)$ on $\{1,\ldots,5\}$ and consider the small resolution $\cX_{ijkl}$, obtained from $\X$ by blowing-up the divisors $D_i$, $D_j$, $D_k$, $D_l$ in that order. It now follows from the local study of the singularities of $\X$ that this is indeed a small resolution of $\X$. 
\\In $\X_{ijkl}$ we still denote the strict transforms of the strata of $\X_k$ by $D_m$, by $D_{mm'}=D_m \cap D_{m'}$ and by $D_{mm'm''}=D_m \cap D_{m'} \cap D_{m''}$ with $m,m',m'' \in \{1,\ldots,5\}$. By the study of the local model in \cref{subsec:local_model_quintic}, the exceptional locus of $g_{ijkl}: \X_{ijkl} \rightarrow \X$
\begin{center}
    $\arraycolsep=1.6pt\def\arraystretch{1}
    \begin{array}{cccccccccc}
         g_{ijkl}:
         &\X_{ijkl} 
         & \xrightarrow[G_{ijkl}]{\substack{\text{blow-up}\\\text{of }D_l}}
         & \X_{ijk} 
         & \xrightarrow[G_{ijk}]{\substack{\text{blow-up}\\\text{of }D_k}}
         & \X_{ij} 
         & \xrightarrow[G_{ij}]{\substack{\text{blow-up}\\\text{of }D_j}}
         & \X_i 
         & \xrightarrow[G_{i}]{\substack{\text{blow-up}\\\text{of }D_i}}
         & \cX \\
         & \cup & & 
         \cup & & 
         \cup & & 
         \cup & & \\
         & S_{lh} & &
         S_{kl},S_{kh} & &
         S_{jk},S_{jl},S_{jh} & &
         S_{ij},S_{ik},S_{il},S_{ih} & &
    \end{array}
    $
\end{center}
consists of ten surfaces $S_{mm'}$, with $m,m' \in \{1,\ldots,5\}$ and $m < m'$ in the order $(i,j,k,l,h)$. The surface $S_{mm'}$ is mapped via $g_{ijkl}$ to the singular curve $C_{mm'}$, and is contained in the strict transform of $D_m$. The component $D_h$ (corresponding to the biggest index in the chosen order) is the only one isomorphic to its strict transform.

Given a pair $(m,m')$ with $m<m'$, the morphism $g_{ijkl}$ induces on $D_{mm'}$ the blow-up along $5$ distinct general points on each $D_{mm'm''}$ with $m'<m''$. Thus, the intersection numbers between strata curves and strata divisors in $\X_{ijkl}$ are:
\begin{center}
\begin{tabular}{c|ccccc } 
 &
 $D_{i}$ & $D_{j}$ & $D_{k}$ & $D_{l}$ & $D_{h}$ \\
 \hline
 $D_{ijk}$ & 
 1 & 1 & -4 & 1 & 1  \\
 $D_{ijl}$ & 
 1 & 1 & 1 & -4 & 1  \\
 $D_{ijh}$ & 
 1 & 1 & 1 & 1 & -4  \\
 $D_{ikl}$ & 
 1 & 1 & 1 & -4 & 1 \\
 $D_{ikh}$ & 
 1 & 1 & 1 & 1 & -4 \\
 $D_{ilh}$ & 
 1 & 1 & 1 & 1 & -4 \\
 $D_{jkl}$ & 
 1 & 1 & 1 & -4 & 1 \\
 $D_{jkh}$ & 
 1 & 1 & 1 & 1 & -4 \\
 $D_{jlh}$ & 
 1 & 1 & 1 & 1 & -4 \\
 $D_{klh}$ & 
 1 & 1 & 1 & 1 & -4 
\end{tabular}
\end{center}

\subsection{Dominating model and combinatorial retraction} \label{subsec:combinatorial retraction}
We consider the blow-up $h_{ijkl}:\cW_{ijkl} \rightarrow \X_{ijkl}$ of the surfaces $S_{mm'}$ in lexicographical order with respect to $(i,j,k,l,h)$; we denote by $E_{mm'}$ the corresponding exceptional divisors. The skeleton $\Sk(\cW_{ijkl})$ consists of the union of the skeleton $\Sk(X)$ with four additional $3$-cells for each 2-dimensional face $\langle v_m,v_{m'},v_{m''}\rangle $ of $\Sk(X)$: for each ordered triple $m<m'<m''$, the union of the additional cells is isomorphic to $\Sk(\cV_{123})$ in \cref{subsec:local dominating model}, where we identify $v_{1}=v_m$, $v_{2}=v_{m'}$ and $v_{3}=v_{m''}$. The retraction $\rho_{\X_{ijkl}}$ collapses the additional faces onto $\langle v_m,v_{m'},v_{m''}\rangle $ as $\rho_{\cU_{12}} \circ \rho_{\cV_{12}}$. 

\tdplotsetmaincoords{60}{35}
\begin{minipage}{0.6\textwidth}
\centering Additional $3$-cells of $\Sk(\cW_{ijkl})$ over $\langle v_m,v_{m'},v_{m''}\rangle $\\with a pictorial description of the retraction $\rho_{\cX_{ijkl}}$
\end{minipage}
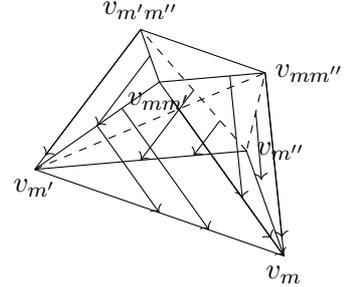
\begin{minipage}{0.4\textwidth}
\begin{figure} [H]
    \centering
\begin{tikzpicture}[tdplot_main_coords,scale=2]
\node[below] at (-1,0,0) {$v_{m'}$};
\node[below] at (1,0,0) {$v_m$};
\node[right] at (0,1,0) {$v_{m''}$};
\node[below] at (0,0,1) {$v_{mm'}$};
\node[right] at (1/2,1/2,1) {$v_{mm''}$};
\node[above] at (-1/2,1/2,1) {$v_{m'm''}$};
\draw[->] (0.1,0,0.9)--(0.9,0,0.1);
\draw[->] (1/2,1/2,1)--(0.95,0.05,0.1);
\draw (-1,0,0)--(1,0,0)--(0,1,0)--(-1,0,0);
\draw (0,0,1)--(1/2,1/2,1)--(1,0,0);
\draw[dashed] (-1,0,0)--(1/2,1/2,1)--(0,1,0);
\draw (-1,0,0)--(0,0,1)--(1,0,0);
\draw (-1,0,0)--(-1/2,1/2,1);
\draw(-1/2,1/2,1)--(0,0,1);
\draw (-1/2,1/2,1)--(1/2,1/2,1);
\draw[dashed] (0,1,0)--(-1/2,1/2,1);
\draw[->] (-1/2,1/2,1)--(-0.95,0.05,0.1);
\draw[->] (-1/4,1/4,1)--(-1/2,0,1/2);
\draw[->] (1/4,3/4,0.5)--(0.4,0.6,0.1);
\draw[->] (1/3,1/3,1)--(2/3,0,1/3);
\draw[->] (-1/4,3/4,1/2)--(-1/2,1/2,0);
\draw[->] (-1/8,7/8,1/4)--(-1/4,3/4,0);
\draw[->] (0.1,0,0.9)--(1,0,0);
\draw[->] (-0.3,0,0.7)--(0.4,0,0);
\draw[->] (-0.5,0,0.5)--(0,0,0);
\end{tikzpicture}
\end{figure}
\end{minipage}

Given another order $(i',j',k',l',h')$ on $\{1,2,3,4,5\}$, the skeleton $\Sk(\cW_{i'j'k'l'})$ coincides with $\Sk(\cW_{ijkl})$ as subspace of $X^{\an}$; we denote this simply by $\Sk(\cW)$. Instead, the triangulation and the retraction depend on the order. Our goal is therefore to construct a model $\cZ$ which dominates all models $\cW_{ijkl}$ regardless of the order, so that all retractions $\rho_{\cW_{ijkl}}$ factors through $\rho_\cZ$.

Along the same lines of \cref{subsec:local dominating model}, we define $\cZ$ as the blow-up of $\cW_{ijkl}$ along $D_{m'} \cap E_{mm''}$ and $E_{mm'} \cap D_{m''}'$, for all ordered triples $m<m'<m''$ in the order $(i,j,k,l,h)$:

\begin{center}
    $\arraycolsep=1.6pt\def\arraystretch{1}
    \begin{array}{ccccccccc}
         \cZ
         & \xrightarrow[\text{for all }m<m'<m'']{
            \substack{
             \text{blow-up of } \\
            D_{m'} \cap E_{mm''}, E_{mm'} \cap D_{m''}'
            }
            }
         & \cW_{ijkl}
         & \xrightarrow[h_{ijkl}]{\substack{
             \text{blow-up of }S_{mm'} \\\text{for all }m<m'
            }}
         & \cX_{ijkl}
         & \xrightarrow[g_{ijkl}]{\substack{
             \text{blow-up of } \\
            D_i,D_j,D_k,D_l
            }}
         & \X. \\
         \cup & & 
         \cup & & 
         \cup & & \\
         E_{mm'm''}, E'_{mm'm''} & &
         E_{mm'} & &
         S_{mm'} & &
    \end{array}
    $
\end{center}
We denote by $E_{mm'm''}$ and $E'_{mm'm''}$ the corresponding exceptional divisors, and deduce from the local study of these morphisms in \cref{subsec:local dominating model} that $\Sk(\mathscr{Z})$ is obtained from $\Sk(\cW)$ by adding a new $3$-cell $\tau_{m m' m''} := \langle v_{mm'},v_{mm''},v_{m'm''},v'_{mm'm''}\rangle $ for each triple $m<m'<m''$. 

We now define the combinatorial retraction of $\Sk(\cZ)$ onto $\Sk(X)$: given the 2-cell $\langle v_m,v_{m'},v_{m''}\rangle $, we identify $v_1=v_m$, $v_2=v_{m'}$ and $v_3=v_{m''}$ and contract onto $\langle v_m,v_{m'},v_{m''}\rangle $ the additional cells of $\Sk(\cZ)$ over $\langle v_m,v_{m'},v_{m''}\rangle $, via the combinatorial retraction $\pi' = \rho \circ \kappa$ constructed in \cref{subsec:local_combinatorial}. With a slight abuse of notation, we still denote this map by $\pi'$. By construction, the composition $\pi=\pi' \circ \rho_{\cZ}$
\begin{center}
    $\arraycolsep=1.6pt\def\arraystretch{1}
    \begin{array}{ccccccc}
         & \pi: 
         & X^{\an}
         & \xrightarrow{\rho_\cZ}
         & \Sk(\cZ)
         & \xrightarrow[\substack{\text{combinatorial}\\\text{retraction}}]{\pi'}
         & \Sk(X)
         \\
         & \rotatebox[origin=c]{270}{$=$} 
         & \cup & &
         \cup & & 
         \cup\\
         \text{over }\Star(v_m)' \text{\hspace{10pt}}
         & 
         \rho_{\cX_{ijkl}}:
         & \pi^{-1}(\Star(v_m)') 
         & \xrightarrow{\rho_\cZ}
         & \Sk(\cZ)
         & \xrightarrow{\rho_{\cX_{ijkl}}}
         & \Star(v_m)' 
    \end{array}
    $
\end{center}
coincides with $\rho_{\cX_{ijkl}}$ over $\Star(v_m)'$ for any order $(i,j,k,l,m)$ on $\{1,2,3,4,5\}$. In other words, around each vertex $v_m$, the map $\pi$ is the Berkovich retraction induced by a small resolution $\cX_{ijkl}$ of $\X$, where the strict transform of $D_m$ is isomorphic to $D_m$, thus in particular $\mathring{D_m} \subset D_m$ is a torus embedding.

For each $2$-dimensional face $\langle v_m,v_{m'},v_{m''}\rangle $ of $\Sk(X)$, we denote by $\Gamma_{mm'm''}$ the graph defined in \cref{defn:disc in simplex}, and its vertices by $p_{mm'}, p_{mm''}, p_{m'm''}$ and $p_{mm'm''}$. We set

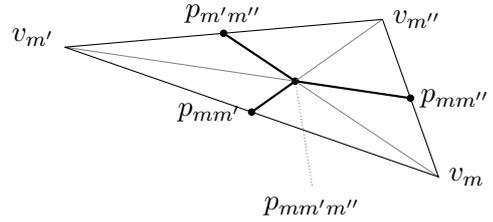
\begin{minipage}{0.45\textwidth}
$$\Gamma \coloneqq \bigcup_{m<m'<m''} \Gamma_{mm'm''}.$$
\end{minipage} \hfill
\begin{minipage}{0.5\textwidth}
\begin{figure} [H]
    \centering
\begin{tikzpicture}[tdplot_main_coords,scale=3] 
\node[left] at (-1,0,0.05) {$v_{m'}$};
\node[right] at (1,0,0) {$v_m$};
\node[right] at (0,1,0) {$v_{m''}$};
\node[left] at (0,0,0) {$p_{mm'}$};
\node[right] at (1/2,1/2,0) {$p_{mm''}$};
\node[above] at (-1/2,1/2,0) {$p_{m'm''}$};
\node[below] at (1/2,-1/4,-0.1) {$p_{mm'm''}$};
\draw[densely dotted,very thin] (0,1/3,0)--(1/2,-1/4,-0.1);
\filldraw (0,0,0) circle (0.4pt);
\filldraw (1/2,1/2,0) circle (0.4pt);
\filldraw (-1/2,1/2,0) circle (0.4pt);
\filldraw (0,1/3,0) circle (0.4pt);
\draw[gray, very thin] (0,0,0)--(0,1,0);
\draw[gray, very thin] (1,0,0)--(-1/2,1/2,0);
\draw[gray, very thin] (-1,0,0)--(1/2,1/2,0);
\draw (-1,0,0)--(1,0,0)--(0,1,0)--(-1,0,0);
\draw[thick] (0,0,0)--(0,1/3,0)--(1/2,1/2,0);
\draw[thick] (0,1/3,0)--(-1/2,1/2,0);
\end{tikzpicture}
\end{figure}
\end{minipage} \hfill
By construction, around any point of $\Sk(X) \setminus \Gamma$, the retraction $\pi$ is equal to the Berkovich retraction induced by 
a suitable good minimal dlt model $\X_{ijkl}$ of $X$.
It follows from the results in \cite{NXY} that $\pi$ induces an integral affine structure with singularities on $\Sk(X)$. By \cref{intro:main thm Z} and \cref{intro:cor}, we obtain that this integral affine structure has no singularities outside $\Gamma$. We will furthermore prove in the next subsection that this affine structure does not extend across any edge of $\Gamma$, i.e. is indeed singular along $\Gamma$.

\subsection{Comparison with other constructions}
\label{sec:comparison}
We conclude by comparing \cref{intro:main quintic} with some other constructions in the literature.
\begin{enumerate}
    \item The affine structure we obtain in \cref{intro:main quintic} coincides with the one defined in \cite{Gross2001}, \cite[§1]{Gro05} and in \cite{HJMM}. Indeed, it follows from the construction and \cref{cor fan structure} that the affine structure induced by $\pi$ yields the fan structure (in the sense of \cite{Gro05}) coming from $D_i$ at each vertex $v_{D_i}$, and that those are glued (after removing $\Gamma$) along the maximal cells viewed as standard simplices; this is the very definition of the singular affine structure in the Gross--Siebert program, and the one considered in \cite{HJMM}. In particular, the monodromy of the singular affine structure can be computed using \cite[Proposition 2.13]{Gro05}, \cite[§2.3]{HZ}.

    \item In \cite{Li} Li considers hypersurfaces $X_t$ in the Fermat family $$\X = \{ z_0 \ldots z_{n+1} + t (z_0^{n+2} +\ldots+ z_{n+1}^{n+2}) =0 \} \subset \CP^{n+1}_{\mathbb{C}} \times \D$$ for arbitrary $n$, and constructs special Lagrangian torus fibrations on generic regions of $X_t$. More precisely, endow $X_t$ with the unique Calabi--Yau metric $\omega_t$ in the class induced by $\gO_{\CP}(1)$. Then the family of rescaled metrics $( \log \lvert t \rvert^{-1})^{-1} \omega_t$ on $X_t$ has bounded diameter and converges in the Gromov-Hausdorff sense to a smooth metric on $\mathbb{S}^n \setminus \Gamma$ as $t\rightarrow 0$; here $\mathbb{S}^n$ is triangulated as the boundary of a standard simplex of dimension $n+1$, and $\Gamma$ is the complement of the open stars of the vertices of $\mathbb{S}^n$ in the first barycentric subdivision (see \cref{defn:disc in simplex}). The metric limit obtained this way is a real Monge--Ampère metric with respect to a certain affine structure on $\mathbb{S}^n \setminus \Gamma$, which is described in \cite[\S 3.2, \S3.5]{Li}. Such integral affine structure coincides with the one in the Gross--Siebert program, hence with the one in \cref{intro:main quintic} when $n=3$.

    \item The affine structure in \cref{intro:main quintic} is semi-simple polytopal in the sense of \cite[Definition 4]{RZ21a}. Integral affine manifolds with semi-simple polytopal singularities are the tropical analog of local complete intersections in algebraic geometry, and are the relevant class of affine structures on the base of the topological SYZ fibration in the context of the Gross--Siebert program. Indeed, given such a manifold $B$, Ruddat and Zharkov construct a topological space $Y$ and torus fibration $Y\rightarrow B$ with discriminant of codimension 2 in $B$, inducing the given affine structure. In \cite{RZ21a} the authors describe the strategy in the 3-dimensional case; the general results will appear in \cite{RZ}, building on the local constructions of \cite{RZ21}. 

In the case of the quintic 3-fold, let $p = p_{ijk}$ be a vertex of the discriminant contained in the interior of a 2-face $\tau= \langle v_i, v_j,v_k \rangle$, and $q = p_{ik}$ a vertex contained in the interior of an edge $e= \langle v_i, v_k \rangle$ of $\tau$. Up to relabelling, we may assume that the lattice $L_p$ of invariant vectors around $p$ (i.e. the sections of the sheaf of integral affine tangent vectors on a small neighbourhood of $p$)  is freely generated by $v_i$ and $v_j$, in which case the three monodromy matrices around $p$ are of the form $T = \Id + 5 v_k^{\vee} \otimes w$ for some primitive $w \in L_p$. Hence, writing $L(p) = L_p$ and $L^{\vee}(p) = 5 L_p^{\bot}$, as well as $L(q) = L_q$ and $L^{\vee}(q) = 5 L_q^{\bot}$ we see that we are in the setting of \cite{RZ21a}: the singularities of the affine structure are semi-simple abelian. Moreover, the vertices $p_{ijk}$ are negative, while the vertices $p_{ik}$ are positive.
\\Note that $\tau$ can be canonically realized inside $L_p$, sending the vertex $v_k$ to the origin; in addition we set $\tau^{\vee} = \langle  0, 5 v_k^{\vee} \rangle $ to be the convex hull of $0$ and $ 5 v_k^{\vee}$ in $\subset L^{\vee}(p)$. The three loops described above are canonically indexed by the edges of $\tau$, and hence by the pairs $(e,f)$, with $e$ an edge of $\tau$ and $f$ the edge of $\tau^{\vee}$. The upshot of working with $5 L_p^{\bot}$ instead of $L_p^{\bot}$ (and similarly for $q$) is now that the monodromy along the loop $\gamma_{e,f}$ is now simply given by the formula $T(\gamma_{e,f}) = \Id + e \otimes f$. 
\\Similarly for $q$, we realize the edge $e_{ik}$ inside $L_q$ as the unit segment and set $e^{\vee} = \langle 0, 5 v_i^{\vee}, 5 v_k^{\vee} \rangle  \subset L^{\vee}(q)$. Then we may once again label the three loops around $q$ by pairs $(e, f)$ with $e= e_{ik}$ and $f$ an edge of $e^{\vee}$, so that the formula $T(\gamma_{e,f}) = \Id + e \otimes f$ holds.
\\Since $e_{ik}$ is a face of $\tau$, we conclude from this that our affine structure is semi-simple polytopal.

\item\label{rem:comparison} In \cite{Ruan2001}, Ruan develops a symplectic method based on gradient flow and constructs a Lagrangian torus fibration for the Fermat family of quintic Calabi--Yau hypersurfaces $$\X = \{ z_1 \ldots z_5 + t (z_1^4+ \ldots +z_5^4) =0 \} \subset \mathbb{P}^4_\mathbb{C} \times \D,$$ later extended to generic quintic hypersurfaces in toric varieties. The idea is to realize $\Sk(\X)$ very explicitly as the boundary of the standard 4-simplex $\tau^4 = \{ \sum_{i=1}^5 w_i =1 \} \subset \R_{\geqslant 0}^5$, and to spread the map:
$$ F : \X_0 \fl \partial \tau^4 $$
$$[ z_1:\ldots: z_5 ] \longmapsto \Bigg(\frac{\lvert z_1 \rvert^2}{\lVert z\rVert^2 } ,\ldots, \frac{\lvert z_5 \rvert^2}{\lVert z\rVert^2 }\Bigg)$$
to the nearby fibers using a gradient flow. This yields a Lagrangian fibration on the $\X_t$'s for small enough $t$, which Ruan expects to be deformable towards a special Lagrangian fibration. In addition, he describes the discriminant locus and the monodromy transformations of the expected special Lagrangian fibration, assuming that the singular locus is of codimension $2$. The predictions in \cite[\S 4.4, \S 4.5]{Ruan2001} match the constructions in \cite{Gross2001}, hence those in \cref{intro:main quintic}.
Note that both in Ruan's and in Gross' aforementioned works, the polyhedral decomposition on $\mathbb{S}^3$ is induced by the intersection complex of the central fiber $\X_0$ (i.e., vertices correspond to zero-dimensional strata of the special fiber and so on); we work instead with the dual intersection complex associated with $\X_0$, which is isomorphic to the intersection complex in the examples we are considering.
\item We are grateful to Yuto Yamamoto for the following observation: the retraction from \cref{intro:main quintic} does not coincide with the one from \cite[thm. A]{PS}. Indeed, let $p \in \Gamma$ be a point in the discriminant that is not a vertex. Then the preimage of $p$ by the combinatorial retraction from \cref{subsec:local_combinatorial} is contained in a $2$-dimensional plane, containing the tangent vector of the line connecting $v_3$ and the point on the edge $<v_1, v_2>$ parametrized by $t$, as in the picture \cref{equ:combinatorial retraction vertx}. In particular, the latter depends on $t$.
\\On the other hand, the preimage of $p$ by the tropical contraction from \cite{Ya} is contained in the $2$-plane spanned by the directions $v_1$ and $v_2$, and is thus independent of $t$.
\item \label{rem:K3} Finally, we note that the constructions in \cref{sec:quintic 3-fold} have an analogue in dimension $2$ which was considered in \cite[\S 4.2.5]{KontsevichSoibelman}. Indeed, consider the degeneration of quartic $\textrm{K}3$ surfaces $\cX= \{ z_1 z_2 z_3 z_4 +tF_4(z_1,z_2,z_3,z_4) =0\} \subset \mathbb{P}^3_R$, where $F_4$ is a generic homogeneous polynomial of degree 4, and the resolution $h: \cZ \rightarrow \cX$ obtained by blowing-up the $24$ singular points of $\cX$, namely the points $\{z_i=0, z_j=0, t=0, F_4 =0\}$ for any $i \neq j$. The special fiber of $\cZ$ is $\cZ_k= \sum_{i=1}^4 D_i + \sum_{q=1}^{24} E_q$, where we abusively still denote by $D_i$ the strict transform of $D_i$ inside $\cZ$, and the $E_q$'s are the exceptional divisors. The skeleton $\Sk(\cZ)$ is the boundary of a tetrahedron with four additional $2$-cells glued along each edge of the tetrahedron; following \cite{KontsevichSoibelman} we call such $2$-cells wings. Kontsevich and Soibelman define a retraction
$$\rho: X^{\textrm{an}} \xrightarrow{\rho_\cZ} \Sk(\cZ) \xrightarrow{\pi} \Sk(X) = \Sk(\cX)$$ where $\rho_\cZ$ is the Berkovich retraction onto the skeleton $\Sk(\cZ)$, and $\pi$ is a contraction of the 24 wings of $\Sk(\cZ)$ to the sphere given as follows.
For each edge $e$ of $\mathcal{D}(\X_k)$ we choose a point $a_e=(a_e,0)$ in the interior of $e$, and define the retraction of $W_q$ onto $e$ by

\begin{minipage}{.5\linewidth}
\centering
\begin{tabular}{cll}
    & $(x+y,0)$ & if $x+y \leqslant a_e$ \\
    $\pi: \, (x,y) \mapsto$ & $(x-y,0)$ & if $x-y \geqslant a_e$ \\
    & $(a_e,0)$ & otherwise.
\end{tabular}
\end{minipage} \hspace{20pt}
\begin{minipage} {.5\linewidth}
\begin{tikzpicture}[scale=0.8]
\filldraw (0,0) circle (1.2pt);
\filldraw (-2,0) circle (1pt);
\filldraw (2,0) circle (1pt);
\filldraw (0,2) circle (1pt);
\node[below] at (0,-1) {Picture for $a_e=0$};
\node[below] at (0,0) {$a_e$};
\node[right] at (0,2) {$v_{q}$};
\node[above] at (3,0) {\textcolor{gray}{$x$}};
\node[right] at (0,3) {\textcolor{gray}{$y$}};
\draw (-2,0) -- (0,2) -- (2,0) -- (-2,0);
\draw[fill=cyan, fill opacity=0.2] (-2,0) -- (0,0) -- (-1,1);
\draw[fill=yellow, fill opacity=0.2] (1,1) -- (0,2) -- (-1,1) -- (0,0);
\draw[fill=orange, fill opacity=0.2] (2,0) -- (0,0) -- (1,1);
\draw [->] (-0.9,0.9) -- (-0.1,0.1); 
\draw [->] (-1.1,0.7) -- (-0.5,0.1); 
\draw [->] (-1.3,0.5) -- (-0.9,0.1); 
\draw [->] (-1.5,0.3) -- (-1.3,0.1); 
\draw (0,1.7) -- (0,0.2); 
\draw (-0.3,1.5) -- (0,0.2); 
\draw (-0.5,1.3) -- (0,0.2); 
\draw (-0.7,1.1) -- (0,0.2);
\draw (0,1.7) -- (0,0.2); 
\draw (0.3,1.5) -- (0,0.2); 
\draw (0.5,1.3) -- (0,0.2); 
\draw (0.7,1.1) -- (0,0.2);
\draw [->] (0.9,0.9) -- (0.1,0.1); 
\draw [->] (1.1,0.7) -- (0.5,0.1); 
\draw [->] (1.3,0.5) -- (0.9,0.1); 
\draw [->] (1.5,0.3) -- (1.3,0.1); 
\draw [gray] (2,0) -- (3,0);
\draw [gray] (0,0) -- (0,3);
\end{tikzpicture}
\end{minipage}
\\We note that 
\begin{itemize}
\itemsep0pt
    \item[-] over the interior of any $2$-dimensional face $\tau \subset \Sk(\cX)$, $\rho$ is equal to $\rho_\cZ$, thus it is an affinoid torus fibration
    (see \cref{ex:affinoid torus over max face}).
    \item[-] Around any vertex $v_{D_i}$, $\rho$ is equal to $\rho_{\X_{i}}$, where $\cX_i$ is the small resolution of $\cX$ obtained blowing-up one after the other all the (strict transforms of the) divisors of $\cX_k$ except $D_i$. By \cref{intro:cor}, $\rho$ is an affinoid torus fibration around $v_{D_i}$, and the affine structure induced there is the fan structure induced by $D_i \simeq \mathbb{P}^2$, by \cref{cor fan structure}.
\end{itemize}
As a result, the integral affine structure induced on $\Sk(\X)$ by $\rho$ away from the $a_e$'s coincides with the one from \cite[§1]{Gro05}, and by \cite[Proposition 2.13]{Gro05}, the monodromy around $a_e$ is conjugate to the matrix $T = \begin{pmatrix} 1 & 4\\ 0 & 1 \end{pmatrix}$, as claimed in \cite{KontsevichSoibelman}.
\end{enumerate}
\bibliographystyle{alpha}
\small
\bibliography{biblio.bib}

\begin{thebibliography}{KKMSD73}

\bibitem[BB96]{BB}
V.~V. Batyrev and L.~A. Borisov.
\newblock On {C}alabi-{Y}au complete intersections in toric varieties.
\newblock In {\em Higher-dimensional complex varieties ({T}rento, 1994)}, pages
  39--65. de Gruyter, Berlin, 1996.

\bibitem[Ber99]{Be}
V.~Berkovich.
\newblock Smooth p-adic analytic spaces are locally contractible.
\newblock {\em Invent. Math. 137}, pages 1--84, 1999.

\bibitem[BFJ15]{BFJ1}
S.~Boucksom, C.~Favre, and M.~Jonsson.
\newblock Solution to a non-{A}rchimedean {M}onge-{A}mpère equation.
\newblock {\em J. Amer. Math. Soc. 28}, pages 617--667, 2015.

\bibitem[BFJ16]{BFJ2}
S.~Boucksom, C.~Favre, and M.~Jonsson.
\newblock Singular semipositive metrics in non-{A}rchimedean geometry.
\newblock {\em J. Algebraic Geom. 25}, pages 77--139, 2016.

\bibitem[BGPS14]{BPS}
J.~Burgos~Gil, P.~Philippon, and M.~Sombra.
\newblock Arithmetic geometry of toric varieties. {M}etrics, measures and
  heights.
\newblock {\em Ast\'{e}risque}, (360):vi+222, 2014.

\bibitem[BJ17]{BJ}
S.~Boucksom and M.~Jonsson.
\newblock Tropical and non-{A}rchimedean limits of degenerating families of
  volume forms.
\newblock {\em Journal de l'Ecole polytechnique}, pages 87--139, 2017.

\bibitem[CLS11]{CoxLittleSchenck2011}
D.~A. Cox, J.~B. Little, and H.~K. Schenck.
\newblock {\em Toric varieties}, volume 124 of {\em Graduate Studies in
  Mathematics}.
\newblock American Mathematical Society, Providence, RI, 2011.

\bibitem[dFKX17]{deFernexKollarXu2012}
T.~{de}~{Fernex}, J.~Koll{\'a}r, and C.~Xu.
\newblock The dual complex of singularities.
\newblock In {\em Higher dimensional algebraic geometry, in honour of Professor
  Yujiro Kawamatas 60th birthday}, volume~74, pages 103--130. Adv. Stud. Pure
  Math., 2017.

\bibitem[EM21]{EvansMauri}
J.D. Evans and M.~Mauri.
\newblock Constructing local models for {L}agrangian torus fibrations.
\newblock {\em Ann. H. Lebesgue}, 4:537--570, 2021.

\bibitem[Eng18]{Eng}
P.~Engel.
\newblock Looijenga’s conjecture via integral-affine geometry.
\newblock {\em J. Differential Geom. 109}, pages 467--495, 2018.

\bibitem[Ful93]{Fu}
W.~Fulton.
\newblock {\em Introduction to toric varieties}.
\newblock Princeton University Press, 1993.

\bibitem[GHK15]{GHK}
M.~Gross, P.~Hacking, and S.~Keel.
\newblock Mirror symmetry for log {C}alabi-{Y}au surfaces {I}.
\newblock {\em Inst. Hautes Etudes Sci. Publ. Math. 155}, pages 65--168, 2015.

\bibitem[GJKM19]{GJKM}
W.~Gubler, P.~Jell, K.~K\"{u}nnemann, and F.~Martin.
\newblock Continuity of plurisubharmonic envelopes in non-archimedean geometry
  and test ideals.
\newblock {\em Ann. Inst. Fourier (Grenoble)}, 69(5):2331--2376, 2019.
\newblock With an appendix by Jos\'{e} Ignacio Burgos Gil and Mart\'{\i}n
  Sombra.

\bibitem[GR04]{GabberRamero}
O.~{Gabber} and L.~{Ramero}.
\newblock Foundations for almost ring theory -- release 6.95.
\newblock {\em Preprint}, 2004.

\bibitem[Gra62]{Gra}
H.~Grauert.
\newblock Uber {M}odifikationen und exzeptionelle analytische {M}engen.
\newblock {\em Math. Ann. 146}, pages 331--368, 1962.

\bibitem[Gro61]{EGA3.1}
A.~Grothendieck.
\newblock \'{E}l\'{e}ments de g\'{e}om\'{e}trie alg\'{e}brique. {III}.
  \'{E}tude cohomologique des faisceaux coh\'{e}rents. {I}.
\newblock {\em Inst. Hautes \'{E}tudes Sci. Publ. Math.}, (11):167, 1961.

\bibitem[Gro01]{Gross2001}
M.~Gross.
\newblock Topological mirror symmetry.
\newblock {\em Invent. Math.}, 144(1):75--137, 2001.

\bibitem[Gro05]{Gro05}
M.~Gross.
\newblock Toric degenerations and {B}atyrev-{B}orisov duality.
\newblock {\em Math. Ann.}, 333(3):645--688, 2005.

\bibitem[GS06]{GrossSiebert2006}
M.~Gross and B.~Siebert.
\newblock {Mirror Symmetry via Logarithmic Degeneration Data I}.
\newblock {\em J. Diff. Geom.}, 72(2):169--338, 2006.

\bibitem[GW00]{GW}
M.~Gross and P.~Wilson.
\newblock Large complex structure limits of {K}3 surfaces.
\newblock {\em J. Differential Geom. 55}, pages 475--546, 2000.

\bibitem[HJMM22]{HJMM}
J.~Hultgren, M.~Jonsson, E.~Mazzon, and N.~McCleerey.
\newblock Tropical and non-archimedean {M}onge–{A}mpère equations for a
  class of {C}alabi–{Y}au hypersurfaces.
\newblock {\em ArXiv e-prints}, 2022.

\bibitem[HZ02]{HZ}
C.~{Haase} and I.~{Zharkov}.
\newblock {Integral affine structures on spheres and torus fibrations of
  Calabi-Yau toric hypersurfaces I}.
\newblock {\em ArXiv e-prints}, 2002.

\bibitem[KKMSD73]{KK}
G.~Kempf, F.~Knudsen, D.~Mumford, and B.~Saint-Donat.
\newblock {\em Toroidal {E}mbeddings {I}}.
\newblock Lect. Notes in Math., Springer-Verlag, 1973.

\bibitem[KM98]{KM}
J.~Koll\'{a}r and S.~Mori.
\newblock {\em Birational geometry of algebraic varieties}, volume 134 of {\em
  Cambridge Tracts in Mathematics}.
\newblock Cambridge University Press, 1998.

\bibitem[Knu71]{Knu}
D.~Knutson.
\newblock {\em Algebraic {S}paces}.
\newblock Lect. Notes in Math., Springer-Verlag, 1971.

\bibitem[Kol13]{Kollar2013}
J.~Koll{\'a}r.
\newblock {\em Singularities of the minimal model program}, volume 200 of {\em
  Cambridge Tracts in Mathematics}.
\newblock Cambridge University Press, 2013.

\bibitem[KS06]{KontsevichSoibelman}
M.~Kontsevich and Y.~Soibelman.
\newblock {\em Affine Structures and Non-Archimedean Analytic Spaces}, pages
  321--385.
\newblock Birkh{\"a}user Boston, 2006.

\bibitem[{Li}20]{Li2}
Y.~{Li}.
\newblock {Metric SYZ conjecture and non-archimedean geometry}.
\newblock {\em ArXiv e-prints}, 2020.

\bibitem[Li22]{Li}
Y.~Li.
\newblock Strominger-{Y}au-{Z}aslow conjecture for {C}alabi-{Y}au hypersurfaces
  in the {F}ermat family.
\newblock {\em Acta Math.}, 229(1):1--53, 2022.

\bibitem[MN15]{MN}
M.~Musțăta and J.~Nicaise.
\newblock Weight functions on non-archimedean analytic spaces and the
  {K}ontsevich-{S}oibelman skeleton.
\newblock {\em Algebraic Geom. 2, no. 3}, pages 365--404, 2015.

\bibitem[Mor93]{Morrison1993}
D.~R. Morrison.
\newblock Mirror symmetry and rational curves on quintic threefolds: a guide
  for mathematicians.
\newblock {\em Journal of the American Mathematical Society}, 6(1):223--223,
  1993.

\bibitem[Mus02]{Mu}
M.~Musțăta.
\newblock Vanishing theorems on toric varieties.
\newblock {\em Tohoku Math. J. 54}, page 451–470, 2002.

\bibitem[Nak04]{Na}
N.~Nakayama.
\newblock {\em Zariski-decomposition and abundance}.
\newblock MSJ Memoirs, vol. 14, Mathematical Society of Japan, 2004.

\bibitem[NX16]{NX}
J.~Nicaise and C.~Xu.
\newblock The essential skeleton of a degeneration of algebraic varieties.
\newblock {\em Amer. Math. J. 138(6)}, pages 1645--1667, 2016.

\bibitem[NXY19]{NXY}
J.~Nicaise, C.~Xu, and T.~Y. Yu.
\newblock The non-archimedean {SYZ} fibration.
\newblock {\em Compositio Mathematica}, 155(5):953--972, 2019.

\bibitem[Oda18]{Odaka2018}
Y.~Odaka.
\newblock {Tropical Geometric Compactification of Moduli, II: Ag Case and
  Holomorphic Limits}.
\newblock {\em International Mathematics Research Notices}, 2018.

\bibitem[PS22]{PS}
L.~Pille-Schneider.
\newblock Hybrid toric varieties and the non-archimedean {S}{Y}{Z} fibration on
  {C}alabi-{Y}au hypersurfaces.
\newblock {\em ArXiv e-prints}, 2022.

\bibitem[RS20]{RS20}
H.~Ruddat and B.~Siebert.
\newblock Period integrals from wall structures via tropical cycles, canonical
  coordinates in mirror symmetry and analyticity of toric degenerations.
\newblock {\em Publ. Math. Inst. Hautes \'{E}tudes Sci.}, 132:1--82, 2020.

\bibitem[Rua01]{Ruan2001}
W.~Ruan.
\newblock Lagrangian torus fibration of quintic hypersurfaces. {I}. {F}ermat
  quintic case.
\newblock In {\em Winter {S}chool on {M}irror {S}ymmetry, {V}ector {B}undles
  and {L}agrangian {S}ubmanifolds ({C}ambridge, {MA}, 1999)}, volume~23 of {\em
  AMS/IP Stud. Adv. Math.}, pages 297--332. Amer. Math. Soc., Providence, RI,
  2001.

\bibitem[RZ]{RZ}
H.~Ruddat and I.~Zharkov.
\newblock Topological {S}trominger--{Y}au--{Z}aslow fibrations.
\newblock {\em In preparation}.

\bibitem[RZ21a]{RZ21a}
H.~Ruddat and I.~Zharkov.
\newblock Compactifying torus fibrations over integral affine manifolds with
  singularities.
\newblock {\em 2019-20 MATRIX Annals, MATRIX Book Series 4}, pages 609--622,
  2021.

\bibitem[RZ21b]{RZ21}
H.~Ruddat and I.~Zharkov.
\newblock Tailoring a pair of pants.
\newblock {\em Adv. Math.}, 381, 2021.

\bibitem[Thu07]{Th}
A.~Thuillier.
\newblock Géométrie toroïdale et géométrie analytique non archimédienne.
  {A}pplication au type d’homotopie de certains schémas formels.
\newblock {\em Manuscr. Math. 123 no. 4}, pages 381--451, 2007.

\bibitem[Yam]{Ya2}
Y.~Yamamoto.
\newblock Non-archimedean {S}{Y}{Z} fibrations via tropical contractions.
\newblock {\em In preparation}.

\bibitem[{Yam}21]{Ya}
Y.~{Yamamoto}.
\newblock {Tropical contractions to integral affine manifolds with
  singularities}.
\newblock {\em ArXiv e-prints}, 2021.

\end{thebibliography}
\end{document}